\documentclass[11pt]{article}

\usepackage[utf8]{inputenc}
\usepackage{amsmath}
\usepackage{amssymb}
\usepackage{amsfonts}
\usepackage{amsthm}
\usepackage{bbm}
\usepackage{graphicx}
\usepackage{fullpage}
\usepackage[authoryear]{natbib}

\usepackage{subcaption}
\usepackage{float}
\usepackage[colorlinks=true,
            linkcolor=red,
            urlcolor=blue,
            citecolor=blue]{hyperref}
            
\usepackage[ruled]{algorithm2e}
\usepackage{array}

\newcommand{\p}{\mathbb{P}}
\newcommand{\e}{\mathbb{E}}
\newcommand{\RR}{\mathbb{R}}
\newcommand{\ind}[2]{\ensuremath{#1 \perp\!\!\!\perp #2}}

\usepackage{verbatim}
\usepackage{listings}
\usepackage{color}

\newtheorem{lemma}{Lemma}
\newtheorem{remark}{Remark}

\newtheorem{mydef}{Definition}
\newtheorem{theorem}{Theorem}

\newtheorem{proposition}{Proposition}
\theoremstyle{definition}
\newtheorem{example}{Example}

\newcommand{\var}{\ensuremath{\operatorname{Var}}}
\newcommand{\cov}{\ensuremath{\operatorname{Cov}}}
\newcommand{\T}{^\intercal}
\newcommand{\supp}{\operatorname{supp}}
\newcommand{\tr}{\operatorname{tr}}
\newcommand{\sign}{\operatorname{sign}}
\usepackage{epsfig}
\usepackage{wasysym}

\newcommand{\bm}{\mathfrak{m}}
\newcommand{\bX}{\boldsymbol{X}}
\newcommand{\bv}{\mathbf{v}}
\newcommand{\bz}{\mathbf{z}}
\newcommand{\Mb}{\mathbf{M}}
\newcommand{\cS}{\mathcal{S}}
\newcommand{\Ab}{\mathbf{A}}
\newcommand{\Zb}{\mathbf{Z}}

\newcommand{\bW}{\boldsymbol{W}}
\newcommand{\Vb}{\mathbf{V}}
\newcommand{\Ub}{\mathbf{U}}
\newcommand{\calH}{\mathcal{H}}
\newcommand{\calI}{\mathcal{I}}
\newcommand{\Nb}{\mathbf{N}}
\newcommand{\bb}{\mathbf{b}}
\newcommand{\bP}{\mathbf{P}}

\definecolor{lbcolor}{rgb}{0.95,0.95,0.95}
\usepackage{comment}
\newcommand{\CTSTART}{}

\newcommand{\Var}{\mathrm{Var}}

\newcommand{\EE}{\mathbb{E}}

\newcommand{\bbE}{\mathbb{E}}

\newcommand{\vX}{\boldsymbol{X}}

\newcommand{\bbeta}{\boldsymbol{\beta}}

\newcommand{\bmu}{\boldsymbol{\mu}}

\newcommand{\bSigma}{\boldsymbol{\Sigma}}

\begin{document}

\title{\huge Signed Support Recovery for Single Index Models in High-Dimensions}

\author{Matey Neykov\thanks{Department of Operations Research and Financial Engineering, Princeton University, Princeton, NJ 08544} \and Qian Lin\thanks{Center of Mathematical Sciences and Applications, Harvard University, Cambridge, MA 02138} \and Jun S. Liu\thanks{Department of Statistics, Harvard University, Cambridge, MA 02138}}

%
%
%
%
%
%
%
%
%
%

\date{}

\maketitle

\begin{abstract}
In this paper we study the support recovery problem for single index models $Y=f(\bX\T\bbeta,\varepsilon)$, where $f$ is an unknown link function, $\bX\sim N_p(0,\mathbb{I}_{p})$ and $\bbeta$ is an $s$-sparse unit vector such that $\bbeta_{i}\in \{\pm\frac{1}{\sqrt{s}},0\}$. In particular, we look into the performance of two computationally inexpensive algorithms: (a) the diagonal thresholding sliced inverse regression (DT-SIR) introduced by \cite{lin2015consistency}; and (b) a semi-definite programming (SDP) approach inspired by \cite{amini2008high}. When $s=O(p^{1-\delta})$ for some $\delta>0$, we demonstrate that both procedures can succeed in recovering the support of $\bbeta$ as long as the \textit{rescaled sample size} $\Gamma=\frac{n}{s\log(p-s)}$ is larger than a certain critical threshold. On the other hand, when $\Gamma$ is smaller than a critical value, any algorithm fails to recover the support with probability at least $\frac{1}{2}$ asymptotically. In other words, we demonstrate that both DT-SIR and the SDP approach are optimal (up to a scalar) for recovering the support of $\bbeta$ in terms of sample size. We provide extensive simulations, as well as a real dataset application to help verify our theoretical observations. 
\end{abstract}



\section{Introduction}

Due to the recent advances in technology, collecting data becomes a routine. The `{\it small $n$, large $p$}' characteristic of modern data brings new challenges in interpreting and processing it. Dimension reduction and variable selection procedures become an indispensable step in data exploration. Regrettably, the majority of the classical algorithms were developed to work in the regime $p \ll n$, and hence one needs to exercise caution when applying  existing methods in a high dimensional setting. A firm understanding of the limitations of classical dimension reduction procedures in the modern $p \gg n$ regime, will help facilitate their appropriate application. 

For example, the archetypical unsupervised dimension reduction procedure --- principal component analysis (PCA), has been widely and successfully applied in a range of scientific problems \cite[e.g.]{price2010new, brenner2000adaptive}. The behavior of PCA in high dimensions has been well studied in the recent years. In \cite{johnstone2004sparse, johnstone2009consistency,paul2007asymptotics}, it was shown that in the spiked covariance model, PCA succeeds if and only if $\lim\frac{p}{n}\neq 0$. This result stimulated the statistical community to discuss the minimax rate of estimating the principal space under sparsity constraints (see e.g., \cite{cai2013sparse, vu2012minimax} and references therein) and the tradeoff between the statistical and computational efficiency (see e.g., \cite{berthet2013complexity}).

Another line of dimension reduction research, studies the so-called sufficient dimension reduction (SDR). The goal of SDR is to find the minimal subspace $\mathcal{S}$ such that $\ind Y \bX \big| P_{\mathcal{S}}\bX$, where $\bX \in \mathbb{R}^{p}$ is the predictor vector and $Y\in \mathbb{R}$ is the response \cite[e.g.]{cook2004testing, li1991sliced, cook2005sufficient}. Unlike its unsupervised counterpart, much less attention has been paid to how SDR algorithms behave in a high dimensional setting. The optimal estimation rates of SDR algorithms in terms of sparsity ($s$), dimensionality ($p$), and sample size ($n$)  are unclear.

Sliced inverse regression, proposed by \cite{li1991sliced}, is one of the most popular SDR methods for estimating the space $\mathcal{S}$. When the dimensionality $p$ is larger than or comparable to the sample size $n$, sparsity assumptions are often imposed on the loading vector $\bbeta$ \cite[e.g.]{li2006sparse}. \cite{lin2015consistency} proved that in fact $\bbE[\angle(\widehat{\bbeta},\bbeta)] >0$ if $\rho=\lim\frac{p}{n}\neq 0$ and $\sin\big(\angle(\widehat{\bbeta},\bbeta)\big)=0$ when $\rho=0$ where $\widehat{\bbeta}$ is the SIR estimator of $\bbeta$. In other words, the SIR estimator $\widehat{\bbeta}$ is consistent (up to a sign) if and only if $\rho=\lim\frac{p}{n}=0$. One implication of this result is that in order to estimate $\bbeta$, structural assumptions such as sparsity are necessary in the high dimensional setting. 

In the present paper, inspired by \cite{amini2008high}, we investigate the support recovery problem for single index models (SIM) (\ref{eqn:model:sir_single_index}), under the sparsity assumption $\|\bbeta\|_0 = s$ in the regime $s=O(p^{1-\delta})$ for some $\delta>0$. More formally we study the models:
 \begin{equation}\label{eqn:model:sir_single_index}
Y=f(\bX\T\bbeta,\varepsilon)  \mbox{ with } \bX \sim N_p(0,\mathbb{I}_{p}), 
 \end{equation}
where the noise $\varepsilon$ is independent of $\bX$ and $f, \varepsilon, \bbeta$ belong to the class:  
\begin{equation}\label{model:class}
\begin{aligned}
\mathcal{F}_A = \Big\{(f, \varepsilon, \bbeta) &: \var(\e[Z | f(Z,\varepsilon)]) \geq A \mbox{ where } Z \sim N(0,1), \\ 
&\bbeta_{i} \in \Big\{\pm\frac{1}{\sqrt{s}}, 0\Big\}  \mbox{ and } (f,\varepsilon) \mbox{ is \textit{sliced stable} (see Definition \ref{def:sliced_stable})} \Big\}.
\end{aligned}
\end{equation}
Notice that no generality is lost in assuming that the vector $\bbeta$ is a unit vector since otherwise model (\ref{eqn:model:sir_single_index}) is not identifiable. Model class (\ref{model:class}), further assumes the idealized setting where all non-zero coordinates of $\bbeta$ have a signal strength of exactly the same magnitude, which simplifies our presentation while preserving the inherent complexity of the support recovery problem. We believe that studying support recovery in SIM, would bring us insightful understanding of SIR and other SDR algorithms.

In this paper, we study two procedures for signed support recover of SIM (\ref{eqn:model:sir_single_index}): the DT-SIR introduced by \cite{lin2015consistency} and the SDP approach inspired by \cite{amini2008high}. We let $\Gamma=\frac{n}{s\log(p-s)}$ be the \textit{rescaled sample size}. Our main contribution is to establish the existence of constants $\omega > 0$ and $\Omega > 0$ such that when $\Gamma > \Omega$ both DT-SIR and SDP approaches recover the signed support of $\bbeta$ correctly, with probability converging to $1$ asymptotically. Conversely, we show that when $\Gamma < \omega$ any algorithm fails to recover the support of $\bbeta$ with probability at least $1/2$.  In other words, we show that the optimal sample size of the support recovery problem of model \eqref{eqn:model:sir_single_index} is of the order $s\log(p-s)$. To the best of our knowledge, this optimality result, regarding the sample size of SIM \eqref{eqn:model:sir_single_index}, has not been previously discussed in the literature. Our second contribution is, to establish a sliced stability conjecture formulated by \cite{lin2015consistency}, under the SIM case. We demonstrate that classical conditions of \cite{hsing1992asymptotic} imply sliced stability in Section \ref{sec:sliced:stable}. This technical result might be of independent interest, especially when one wants to discuss the optimal rate problem for other SDR algorithms such as SAVE. We further develop a novel tool along the way of our analysis, which may represent further interest --- a concentration inequality stated under Lemma \ref{bernst}.
\subsection{Related Work}

In the fixed $p$ setting, the first asymptotic results on SIR appeared in the seminal papers \citep{duan1991slicing, hsing1992asymptotic}. Later on \cite{zhu2006sliced} allowed $p$ to diverge slowly with $n$ and established asymptotics in the regime $p = o(n^{1/2})$. In the super high-dimensional setting where $p \gg n$, several algorithms, hinging on regularization such as LASSO \citep{tibshirani1997lasso} and Dantzig Selector \citep{candes2007dantzig} were proposed by \cite{li2006sparse, yu2013dimension}, but these algorithms are not concerned with support recovery. Moreover, the algorithm suggested by \cite{li2006sparse} did not come with theoretical guarantees, and in \cite{yu2013dimension} it is not allowed for $s$ to scale with $p$ and $n$. A generic variable selection procedure was suggested in \cite{zhong2012correlation}, with guarantees of support recovery, in a more general setting than our present paper, but with a much more restrictive relationship ($p = o(n^{1/2})$) than the one we consider.

In the parallel line of research on sparse PCA there have been more developments. In \cite{johnstone2004sparse} the algorithm Diagonal Thresholding (DT) was suggested, to deal with the spiked-covariance model. It was later analyzed by \cite{amini2008high}, who showed that support recovery is achieved by DT in the sparse spiked covariance model,  provided that $n \gtrsim s^2 \log(p)$. \cite{amini2008high} further showed an information theoretic obstruction, in that no algorithm can recover the support of the principal eigenvector if $n \lesssim s \log(p)$. A computationally inefficient algorithm that succeeds in support recovery with high probability as long as $n \gtrsim s \log(p)$ is exhaustively scanning through all ${p \choose s}$ subsets of the coordinates of the principal eigenvector. In order to find feasible procedures, \cite{amini2008high} studied a semidefinite programming (SDP) estimator originally suggested in \cite{d2008optimal} --- and showed that if $n \gtrsim s \log(p)$ and the SDP has a rank $1$ solution, this solution can recover the signed support with high probability. Surprisingly however, \cite{krauthgamer2013semidefinite} showed that the rank $1$ condition, does not hold if $s^2 \log(p) \gtrsim n \gtrsim s \log(p)$. In contrast to the PCA case, our current paper argues that if one is concerned with the support recovery for SIM in the class $\mathcal{F}_A$, no computational and statistical tradeoff exists in the regime $s = O(p^{1-\delta})$, since the computationally tractable algorithms DT-SIR and SDP algorithms solve the support recovery problem of SIM with optimal sample size.

\subsection{Preliminaries and Notation}
We first briefly recall the SIR procedure for SIM. Suppose we observe $n=Hm$ independent and identically distributed (i.i.d.) samples $(Y_{i},\bX_{i})$ from model \eqref{eqn:model:sir_single_index}. SIR proceeds to sort and divide the data into $H$ slices of equal size, according to the order statistics $Y_{(i)}$ of $Y_{i}$. Let the concomitant of $Y_{(i)}$ be $X_{(i)}$. We will also use the double subscript notation $X_{h,i}$ for $X_{((h-1)m+i)}$. Let $S_h = (Y_{((h-1)m)},Y_{(hm)}], 1 \leq h < H$, and $S_H = (Y_{((H-1)m)}, +\infty)$ denote the random intervals partitioned by the points $Y_{(jm)}$(with $Y_{(0)} = -\infty$), $j \leq H-1$. Furthermore let $\bm_j(Y) = \e[\bX^j | Y]$ denote the $j$\textsuperscript{th} coordinate of the centered inverse regression curve $\e[\bX | Y]$, and $\bmu^j_h$ denotes $\e[\bX^j | Y \in S_h]$. Note that conditionally on the values $Y_{((h-1)m)}$ and $Y_{(hm)}$ the quantities $S_h$ and $\bmu_h^j$ become deterministic. 
For any integer $k \in \mathbb{N}$ put $[k] = \{1,\ldots, k\}$ for brevity.  Let $\overline \bX^j_{h,S} = \frac{1}{|S|} \sum_{ i \in S} \bX^j_{h,i}$, where $S \subset [m]$, $j \in [p]$. If $S = [m]$ we omit it from the notation, i.e., $\overline \bX^j_{h} = \frac{1}{m} \sum_{i = 1}^m \bX^j_{h,i}$. Put $\overline \bX^j = \frac{1}{H} \sum_{h = 1}^H \overline \bX^j_{h}$ for the global average. In terms of this notation, the SIR estimator $\Vb$ of the conditional covariance matrix is given by
\begin{align}
\Vb^{jk} = \frac{1}{H}\sum_{h = 1}^H \overline \bX^j_{h}\overline \bX^k_{h} \label{SIRVARestimator},
\end{align}
entrywise. The SIR estimator $\widehat{\bbeta}$ of $\bbeta$ is defined as the principal eigenvector of $\Vb$. We will denote the support of the vector $\bbeta$ by $\cS_{\bbeta}$, i.e. we set $\cS_{\bbeta} := \supp(\bbeta) = \{j : \bbeta_j \neq 0\}$.

Finally, we need some standard notations. For a vector $\bv$, let $\|\bv\|_p$ denote the usual $\ell_p$ norm for $1\leq p \leq \infty$ and $\|\bv\|_0=|\supp(\bv)|$.  
For a real random variable $X$, let
$
\|X\|_{\psi_2} = \sup_{p \geq 1} p^{-1/2} (\e|X|^p)^{1/p}, \|X\|_{\psi_1} = \sup_{p \geq 1} p^{-1} (\e|X|^p)^{1/p}.
$
Recall that a random variable is called \textit{sub-Gaussian} if $\|X\|_{\psi_2} < \infty$ and \textit{sub-exponential} if $\|X\|_{\psi_1} < \infty$. 
For a matrix $\Mb \in \mathbb{R}^{d \times d}$ and two sets $S_1,S_2 \subset [d]$, by double indexing $\Mb_{S_1, S_2}$ we denote the sub-matrix of $\Mb$ with entries $\Mb_{ij}$ for $i \in S_1, j\in S_2$. Furthermore, for a $d \times d$ matrix $\Mb_{d \times d}$, let $\|\Mb\|_{\max} = \max_{jk} |\Mb_{jk}|$ and $\|\Mb\|_{p,q} = \sup_{\|\bv\|_p = 1} \|\Mb \bv\|_q$. In particular, we have
$
\|\Mb\|_{2,2} = \max_{i \in [d]}\{\sigma_i(\Mb)\},
$ 
where $\sigma_i(\Mb)$ is the $i$\textsuperscript{th} singular value of $\Mb$, and
$
\|\Mb\|_{\infty,\infty} = \max_{i \in [d]} \sum_{j =1}^d |\Mb_{ij}|.
$
Let $F_n(x) = \frac{1}{n}\sum_{i = 1}^n \mathbb{I}(Y_i \leq x)$ denote the empirical distribution of the $Y$ sample, and $\Phi$ ({\it resp.} $\phi$) denote the cdf ({\it resp.} pdf) of a standard normal random variable. We will also occasionally use the abbreviation WLOG to stand for ``without loss of generality''.

\subsection{Organization}
The paper is organized as follows. The main results, and confirmatory numerical studies and a real data analysis are presented in Section \ref{secmainres} and Section \ref{secnumericalres} respectively. Proofs of our main results are included in Sections \ref{secpfthm1} and \ref{secpfthm2} while technical lemmas are deferred to Appendix \ref{technical:proofs}. A brief discussion on the potential directions is included in Section \ref{discsec}. In Appendix \ref{app:examples} we provide several concrete examples of SIM.

\section{Main Results} \label{secmainres}

In this section, we first prove a conjecture regarding the coordinate-wise sliced stability conditions introduced by \cite{lin2015consistency}. Second, we establish the optimal rate of support recovery in terms of the sample size.
Throughout the remainder of the paper, we assume that $Y$ is a continuously distributed random variable.

\subsection{Sliced Stability}\label{sec:sliced:stable}

Our proofs of signed support recovery rely on the following property of the inverse regression curve of a SIM.
\begin{mydef}[Sliced Stability] \label{def:sliced_stable}
We call the pair $(f,\varepsilon)$ sliced stable iff there exist constants $0 < l < 1, 1 < K, 0 < M$, such that for any $H \in \mathbb{N}, H > M$, and all partitions of $\mathbb{R} = \{ a_1 = -\infty, \ldots, a_{H+1} = +\infty\}$ with $\frac{l}{H} \leq \p(a_h < Y \leq a_{h + 1}) \leq  \frac{K}{H}$  there exist two constants $0 \leq \kappa(l, K,M) < 1$, $C(l,K,M) > 0$ such that for all $j \in \cS_{\bbeta}$\footnote{Observe that (\ref{slicedstab}) automatically holds for $j \not \in \cS_{\bbeta}$ in our case, since both the LHS and the RHS of (\ref{slicedstab}) are $0$ in this case.}:
\begin{align}
\sum_{h = 1}^H \var[\bm_j(Y) | a_h < Y \leq a_{h+1}] \leq C(l,K,M) H^{\kappa(l,K, M)} \var[\bm_j(Y)]. \label{slicedstab}
\end{align} 
\end{mydef}

The sliced stability assumption, is an implicit assumption on the function $f$ and the error distribution $\varepsilon$. If $\kappa = 0$, the condition means that the cumulative relative variability of the inverse regression curve is bounded for all slicing schemes with sufficiently small slices. If $\kappa > 0$ the cumulative relative variability of the inverse regression curve is allowed to scale sub-linearly with the number of slices (see also Figure \ref{sliced:stab:figure}). Intuitively, sliced stability allows us to ensure that the estimates $\Vb^{jj}$ of $\var[\bm_{j}(Y)]$ become increasingly accurate the more slices we introduce (see Lemma \ref{simplelemma} for a more rigorous treatment). Relying on the subsequent developments of this section, in Example \ref{exmp3} and Remark \ref{class:of:sliced:stable:models} of Appendix \ref{app:examples} we demonstrate that models of the form $Y = G(h(\bX\T \bbeta) + \varepsilon)$, where $G, h$ are continuous and monotone and $\varepsilon$ is a log-concave random variable, satisfy the sliced stability assumption.

\begin{figure}[H]\caption{Below, with a solid line we plot the standardized inverse regression curve $\bm(y) := \frac{\sign(\beta_j) \bm_j(y)}{\sqrt{\var(\bm_j(Y))}}, j \in \cS_{\bbeta}$, for the model $Y = 2 \operatorname{atan}(\bX\T \bbeta) + N(0,1)$ (see also (\ref{X3model})). The different colored parts of the dashed curve represent the conditional densities of $Y | Y \in S_h$, where $S_h$ for $h \in [H]$ (with H = $7$) are the slices illustrated by punctured vertical black lines. Finally the seven points' vertical-axis values represent the variances $\var[\bm(Y) | Y \in S_h]$. Sliced stability ensures that the average value of the variances $\var[\bm(Y) | Y \in S_h]$ has a decay rate of the order $H^{\kappa -1}$.}
\label{sliced:stab:figure}
  \centering
  \includegraphics[width=.6\linewidth]{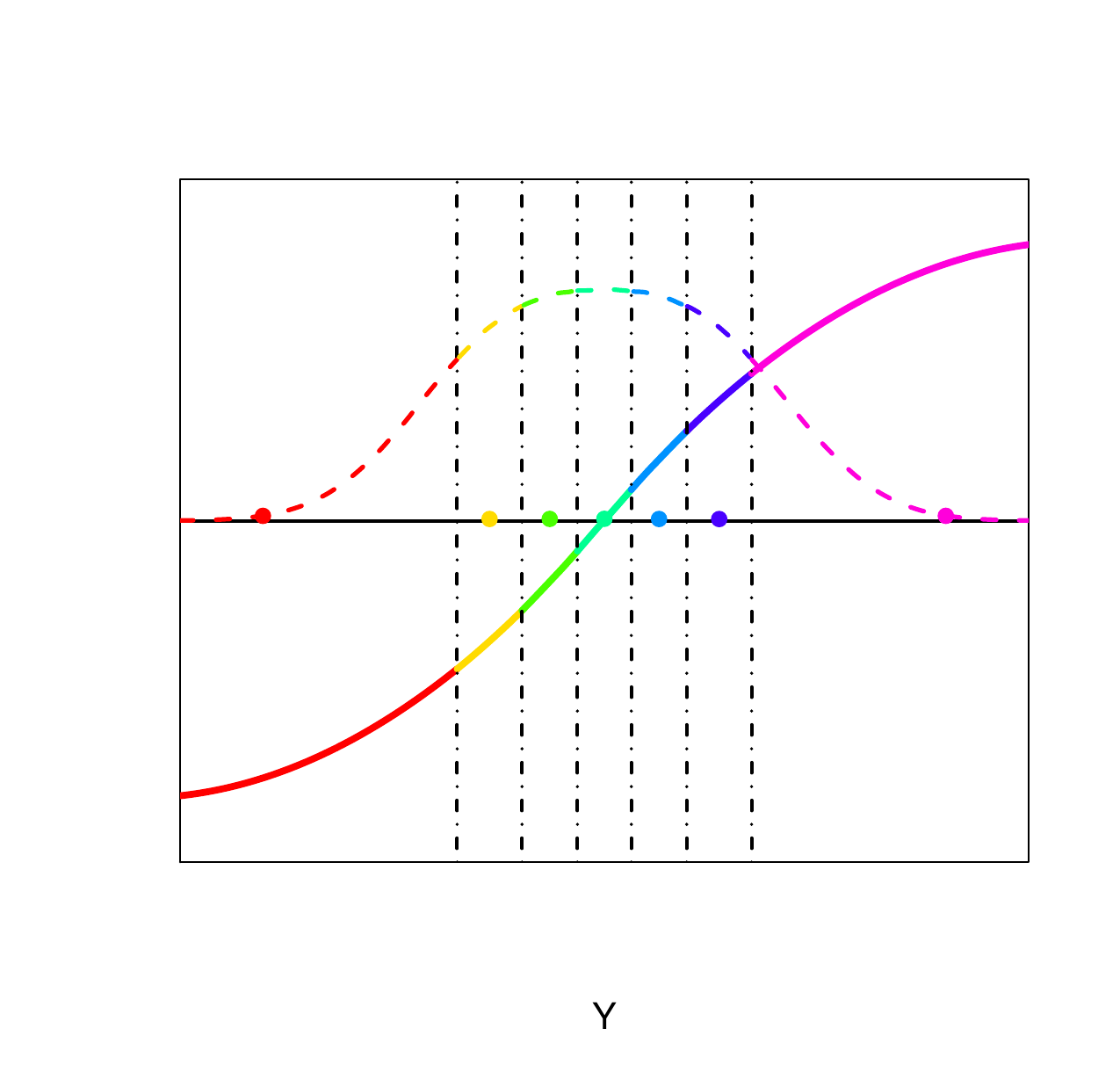}
\end{figure}

%

Definition \ref{def:sliced_stable} is the sliced stability definition from \cite{lin2015consistency} restated in terms of the SIM. \cite{lin2015consistency} conjectured that the sliced stability condition could be implied from the well accepted conditions proposed by \cite{hsing1992asymptotic}, which we state below with a slight modification. The first contribution of this paper is proving this conjecture in the case of SIM. 

Let $\mathcal{A}_H(l, K)$, with $1 < K , 0 < l < 1$, denote all partitions of $\mathbb{R}$ of the sort $\{ -\infty = a_1 \leq a_2 \leq \ldots \leq a_{H+1} = +\infty$\}, such that $\frac{l}{H} \leq \p(a_h \leq Y \leq a_{h + 1}) \leq \frac{K}{H}$. Moreover, for any fixed $B \in \mathbb{R}$, let $\Pi_r(B)$ denote all possible partitions of the closed interval $[-B,B]$ into $r$ points $-B \leq b_1 \leq b_2 \leq \ldots \leq b_r \leq B$.  Define the normalized version of the centered inverse regression curve $\bm(y) := \frac{\sign(\beta_j) \bm_j(y)}{\sqrt{\var(\bm_j(Y))}}, j \in [p]$\footnote{By symmetry $\sign(\beta_j) \bm_j(y) = \sign(\beta_i) \bm_i(y)$ for $i,j \in \cS_{\bbeta}$}, and let $\bm$ satisfy the following smoothness condition:
\begin{align} \label{msmoothness}
\lim_{r \rightarrow \infty} \sup_{b \in \Pi_r(B)} r^{-1/(2 + \xi)}  \sum_{i = 2}^{r} |\bm(b_i) - \bm(b_{i - 1})| = 0,
\end{align}
for any $B > 0$ for some fixed $\xi > 0$. Note that as mentioned in \cite{hsing1992asymptotic}, assumption (\ref{msmoothness}) is weaker than assuming that $\bm$ is of bounded variation, and furthermore the bigger the $\xi$ the more stringent this assumption becomes. In addition, assume that there exists  $B_0 > 0$ and a non-decreasing function $\widetilde \bm:(B_0, \infty) \mapsto \mathbb{R}$, such that:
\begin{align} \label{mglobal}
|\bm(x) - \bm(y)| \leq |\widetilde \bm (|x|) - \widetilde \bm(|y|)|, \mbox{ for } x, y \in (-\infty, -B_0) \mbox { or }  (B_0, +\infty),
\end{align}
and moreover, $\e[|\widetilde \bm(|Y|)|^{(2 + \xi)}] < \infty$ (where in the expectation we set $\widetilde \bm(y) = 0$ for $|y| \leq B_0$).  We are now in a position to formulate the following:

\begin{proposition} \label{slicedstabsuffcond} Assume that the standardized centered inverse regression curve satisfies properties (\ref{msmoothness}) and (\ref{mglobal}) for some $\xi > 0$. Then we have that for any fixed $0 < l < 1 < K$:
\begin{align} \label{varsumgozero}
\lim_{H \rightarrow \infty} \sup_{a \in \mathcal{A}_H(l,K)} \frac{1}{H^{2/(2 + \xi)}}\sum_{h = 1}^{H} \var[\bm(Y) | a_h < Y \leq a_{h + 1}] \rightarrow 0.
\end{align}
\end{proposition}

We defer the proof of Proposition \ref{slicedstabsuffcond} to Appendix \ref{technical:proofs}. It is clear however that (\ref{varsumgozero}) implies the existence of constants $M$, $C(l,K,M)$ such that (\ref{slicedstab}) holds, with $\kappa = \frac{2}{2 + \xi} < 1$.

\subsection{Optimal sample size for support recovery}
Recall that $\Gamma=\frac{n}{s\log(p-s)}$ is the rescaled sample size. First we establish an information theoretic barrier of the support recovery problem, i.e. we show that there exists a positive constant $\omega$ such that, when $\Gamma<\omega$ every algorithm fails with probability at least $1/2$. Then we prove that two algorithms --- DT-SIR and an SDP based procedure achieve this bound, i.e. we demonstrate that there exists a positive constant $\Omega$, such that, when $\Gamma>\Omega$, these two algorithms successfully recover the signed support with probability $1$ as $n \rightarrow \infty$.

\subsubsection{Information theoretic barrier}
Intuitively, when the sample size is small, no algorithm is expected to be able to recover the support successfully. In fact, let us consider the simple linear regression model:
\begin{align}\label{mod:linear}
Y=\vX\T \bbeta+\varepsilon  \mbox{ where } \vX \sim N_p(0,\mathbb{I}_{p}), \varepsilon \sim N(0,\sigma^{2}),
\end{align}
where $\bbeta$ is a unit vector such that $\bbeta_{i} \in \{\pm\frac{1}{\sqrt{s}},0\}$. This model belongs to the class $\mathcal{F}_{(\sigma^2 + 1)^{-1}}$ (to verify sliced stability refer to Example \ref{exmp3} and Remark \ref{class:of:sliced:stable:models} of Appendix \ref{app:examples}). The following result, which is obtained by a slight modification of the arguments in \citep{wainwright2009information}, shows that support recovery is infeasible if $\Gamma$ is small, by effectively arguing that no algorithm works even for model (\ref{mod:linear}).

\begin{proposition} \label{lowerboundprop} Suppose there are $n$ observations from the model \eqref{mod:linear}.  Then there exists a positive constant $\omega$, such that if
$$
\Gamma=\frac{n}{s\log(p-s)} < \omega,
$$
any algorithm for support support recovery will have errors with probability at least $\frac{1}{2}$ asymptotically.
\end{proposition}

\subsubsection{Optimality of DT-SIR}

Recall that we are working with the class of models $\mathcal{F}_{A}$ \eqref{model:class}. One key observation is the following signal strength result.

\begin{lemma} \label{signallemma} 
For any j $\in \cS_{\bbeta}=supp(\bbeta)$, one has:
\begin{align}
\frac{A}{s} \leq \var(\bm_{j}(Y)) \leq \frac{1}{s},
\end{align}
and $\var(\bm_{j}(Y))=0$ for $j \not \in \cS_{\bbeta}$. 
\end{lemma}

By Lemma \ref{signallemma} WLOG we can assume that there exists a $C_V > 0$ such that $\var[\bm_j(Y)] =\frac{C_{V}}{s}$ for any $j \in \cS_{\bbeta}$. Now we are ready to discuss the properties of the DT algorithm which we formulate below. In this paper, we assume that the sparsity $s$ is known.

\begin{algorithm}[H] \label{DTalgo}
\SetKwInOut{Input}{input}
\Input{$(Y_i, \bX_i)_{i = 1}^n$: data, $H$: number of slices, $s$: the sparsity of $\bbeta$}
\begin{enumerate}
\item Calcluate $\Vb^{jj}, j \in [p]$ -- according to formula (\ref{SIRVARestimator});
\item Collect the $s$ highest $\Vb^{jj}$ into the set $\widehat S$;
\item Output the set $\{j: \Vb^{jj} \in \widehat S\}$.
\end{enumerate}
\caption{DT algorithm}
\end{algorithm}
Recall the definitions (see \eqref{SIRVARestimator}) of $\Vb$ and $\Vb^{jj}$ the estimates of $\cov[\bbE[\vX|Y]]$ and $\var[\bm_{j}(Y)]$ respectively. To obtain the result on sample size optimality of Algorithm \ref{DTalgo}, a key point is to establish a concentration inequality for $\Vb^{jj}$. When $j \not \in \cS_{\bbeta}$, a standard deviation inequality for the $\chi^{2}$ distribution is applicable. For $j \in \cS_{\bbeta}$, we need to pay extra effort to obtain the appropriate deviation inequality, which is also the main technical contribution of this paper. Once we have established these deviation inequalities, we can show the following theorem.
\begin{theorem} \label{DTthm} Suppose $s = O(p^{1-\delta})$ for some $\delta > 0$. There exists a positive constant $\Omega$ such that, for any 
\begin{align}
\Gamma=\frac{n}{s\log(p-s)} \geq \Omega, \label{nbound}
\end{align}
the support $S$ is recovered by the DT algorithm  (i.e., $\widehat{S}=S$) with probability converging to $1$ as $n$ increases. Additionally, the number of slices can be held large enough but fixed (depending solely on $C, l, K,M,\kappa,C_V$).
\end{theorem}

\begin{remark}[Choice of $H$]  \label{rmk:on:H} \normalfont
An interesting by-product of our analysis is that we can choose the number of slices $H$ large enough but finite. When dimension $p$ is fixed, this has already been observed in the literature \citep{li1991sliced}, however in high dimensional setting to the best of our knowledge this property has not been discussed. 
\end{remark}

Clearly the DT algorithm does not recover the signed support of $\bbeta$ standalone.  One can apply DT first to select the variables and then apply the SIR procedure to obtain an estimate of the principal eigenvector. To obtain the signed support of $\bbeta$, take the sign of the principal eigenvector. We summarize this algorithm in the following:

\begin{algorithm}[H]
\SetKwInOut{Input}{input} \label{algoDTSIR}
\Input{$(Y_i, \bX_i)_{i = 1}^n$: data, $H$: number of slices, $s$: the sparsity of $\bbeta$}
\begin{enumerate}
\item Perform Algorithm \ref{DTalgo} to obtain the set $\{j: \Vb^{jj} \in \widehat S\}$.
\item Evaluate ${\bv}$ --- the principal eigenvector of the matrix $\Vb_{\widehat S, \widehat S}$.
\item Output $\sign(\bv)$.
\end{enumerate}
\caption{DT-SIR}
\end{algorithm}
We note that this algorithm recovers the signed support, up to multiplication by $\pm 1$. The following result is a direct
corollary of Theorem \ref{DTthm}, whose proof is omitted. 

\begin{proposition} \label{propDTSIR} Under the assumptions of Theorem \ref{DTthm}, with a potentially bigger value of $\Omega$, applying Algorithm \ref{algoDTSIR} restores the signed support (up to a sign) with asymptotic probability converging to $1$. In fact Algorithm \ref{algoDTSIR} works when $\bX \sim N_p(\bmu, \mathbb{I}_p)$ with $\Vb$ substituted with $\widehat \Vb$ where the $jk$\textsuperscript{th} entry of $\widehat \Vb$ is defined as $\widehat \Vb^{jk} = \frac{1}{H} \sum_{h = 1}^H (\overline \bX^j_h - \overline \bX^j) (\overline \bX^k_h - \overline \bX^k)$. 
\end{proposition}

\subsubsection{Optimality of SDP}
Algorithm \ref{algoDTSIR} is a two-step procedure, selecting variables by applying DT and then obtaining the SIR estimate of $\bbeta$. Recent advances of optimization theory provide us with a more sophisticated one-step approach to obtain a sparse principal eigenvector. It is well known that the  principal eigenvector of a symmetric positive definite matrix $\Ab$ is given by:
\begin{align*}
\widehat \bz  = \operatornamewithlimits{argmax}_{\bz \in \mathbb{R}: \|\bz\|_2 = 1} \bz\T  \Ab \bz
\end{align*}
and the principal eigenvalue is $\widehat \bz\T \Ab \bz$. When $\bz$ is sparse, i.e., $\|\bz\|_{0}\leq s$ , the above optimization problem with an additional sparsity constraint is computationally expensive. To remedy this difficulty, \cite{d2007direct} proposed an SDP approach to solve the sparse optimization problem. They suggested to solve the following convex program:
\begin{align} \label{SDPform}
\widehat \Zb = \operatornamewithlimits{argmax}_{\tr(\Zb) = 1, \Zb \in \mathbb{S}_+^{p}} \tr(\Ab \Zb) - \lambda_n \sum_{i,j=1}^p |\Zb_{ij}|,
\end{align}
where $\mathbb{S}_+^{p}$ is the set of all the $p\times p$ positive semi-definite matrices.
If the solution happens to be a rank 1 solution, then it is of the form $\widehat \Zb = \widehat \bz \widehat \bz\T$, and hence we can easily obtain an estimate of the principal eigenvector.

In the following algorithm we summarize the SDP approach, tailored for the signed support recovery of SIM. We remind the reader, that this algorithm recovers the signed support, up to multiplication by a global constant equal to $\pm 1$. 

\begin{algorithm}[H] \label{SDPforSIR}
\SetKwInOut{Input}{input}
\Input{$(Y_i, \bX_i)_{i = 1}^n$: data, $H$: number of slices, $s$: the sparsity of $\bbeta$}
\begin{enumerate}
\item Calcluate the matrix $\Vb$ --- as given in (\ref{SIRVARestimator});
\item Obtain the matrix $\widehat \Zb$ by solving (\ref{SDPform}), with $\Ab = \Vb$;
\item Find the principal eigenvector $\widehat \bz$ of $\widehat \Zb$;
\item Output $\sign(\widehat \bz)$.
\end{enumerate}
\caption{SDP algorithm for SIR}
\end{algorithm}
The performance of this algorithm is guaranteed by the following theorem.

\begin{theorem} \label{SDPthm} Suppose $s = O(p^{1-\delta})$ for some $\delta > 0$, then there is a positive constant $\Omega$ such that, for any 
\begin{align}
\Gamma=\frac{n}{s\log(p-s)} \geq \Omega, \label{nboundSDP}
\end{align}
with a properly chosen the tuning parameter $\lambda_n$ , Algorithm \ref{SDPforSIR} recovers the signed support with probability converging to 1, i.e. $\p(\sign(\widehat \bz) = \sign(\bbeta)) \rightarrow 1$\footnote{We understand that $\widehat \bz$ is selected so that $\widehat \bz\T\bbeta \geq 0$ by convention.}.
\end{theorem}

\begin{remark}\label{centering:SDP:remark} \normalfont In fact Algorithm \ref{SDPforSIR} works when $\bX \sim N_p(\bmu, \mathbb{I}_p)$ with $\Vb$ substituted with $\widehat \Vb$ where the $jk$\textsuperscript{th} entry of $\widehat \Vb$ is defined as $\widehat \Vb^{jk} = \frac{1}{H} \sum_{h = 1}^H (\overline \bX^j_h - \overline \bX^j) (\overline \bX^k_h - \overline \bX^k)$. The proof of this fact is trivial and is omitted.
\end{remark}

\section{Numerical Experiments and Data Analysis}  \label{secnumericalres}

We open this section with extensive numerical studies and in Section \ref{real:data:example} we apply our algorithms to a real dataset.

\subsection{Simulations} 
In this section we compare Algorithms \ref{algoDTSIR} and \ref{SDPforSIR} in terms of signed support recovery. We consider the following scenarios:
\begin{align}
Y & = \bX\T\bbeta + \sin(\bX\T\bbeta) + N(0,1), \label{sinmodel}\\
Y & = 2 \operatorname{atan}(\bX\T\bbeta) + N(0,1), \label{X3model} \\
Y & = (\bX\T\bbeta)^3 + N(0,1), \label{Xe3model}\\
Y & = \operatorname{sinh}(\bX\T\bbeta) + N(0,1). \label{regmodel}
\end{align}

In each simulation we create a sample of size $n$ with dimensionality $p$ and sparsity levels $s = \sqrt{p}$ ($s = \log(p)$ \textit{resp.}) and we vary the rescaled sample size $\Gamma \in[0,30]$ for the DT-SIR and $\Gamma \in [0,40]$ for the SDP approach. The vector $\bbeta$ is selected in the following manner:
$$\beta_j = \frac{1}{\sqrt{s}}, ~ j \in [s-1], ~~ \beta_{s} = -\frac{1}{\sqrt{s}}, \mbox{ and }\beta_j = 0, \mbox{ for } j > s.$$ 
Each simulation is repeated $500$ times, and we report the proportion of correctly recovered signed supports of the vector $\bbeta$. We remark that in general, it should not be expected that the phase transition described in Theorems \ref{DTthm} and \ref{SDPthm} to occur at the same places for these $4$ models.  

We first explore the predictions of Proposition \ref{propDTSIR} and Algorithm \ref{algoDTSIR}. Even though we provide theoretical values of the constants $H$ and $m$, we ran all simulations with $H = 10$ slices. We believe this scenario, is still reflective of the true nature of the DT-SIR algorithm, as the theoretical value of $H$ we provide is not optimized in any fashion. In Figure \ref{figuressqrtp}, we present DT-SIR results from plots for different $p$ values in the regime $s = \sqrt{p}$.
\begin{figure}[H]\caption{Efficiency Curves for DT-SIR, $s = \sqrt{p},~ \Gamma = \frac{n}{s\log(p-s)} \in [0,30]$} \label{figuressqrtp}
\begin{subfigure}{.5\textwidth}
  \centering
  \includegraphics[width=.8\linewidth]{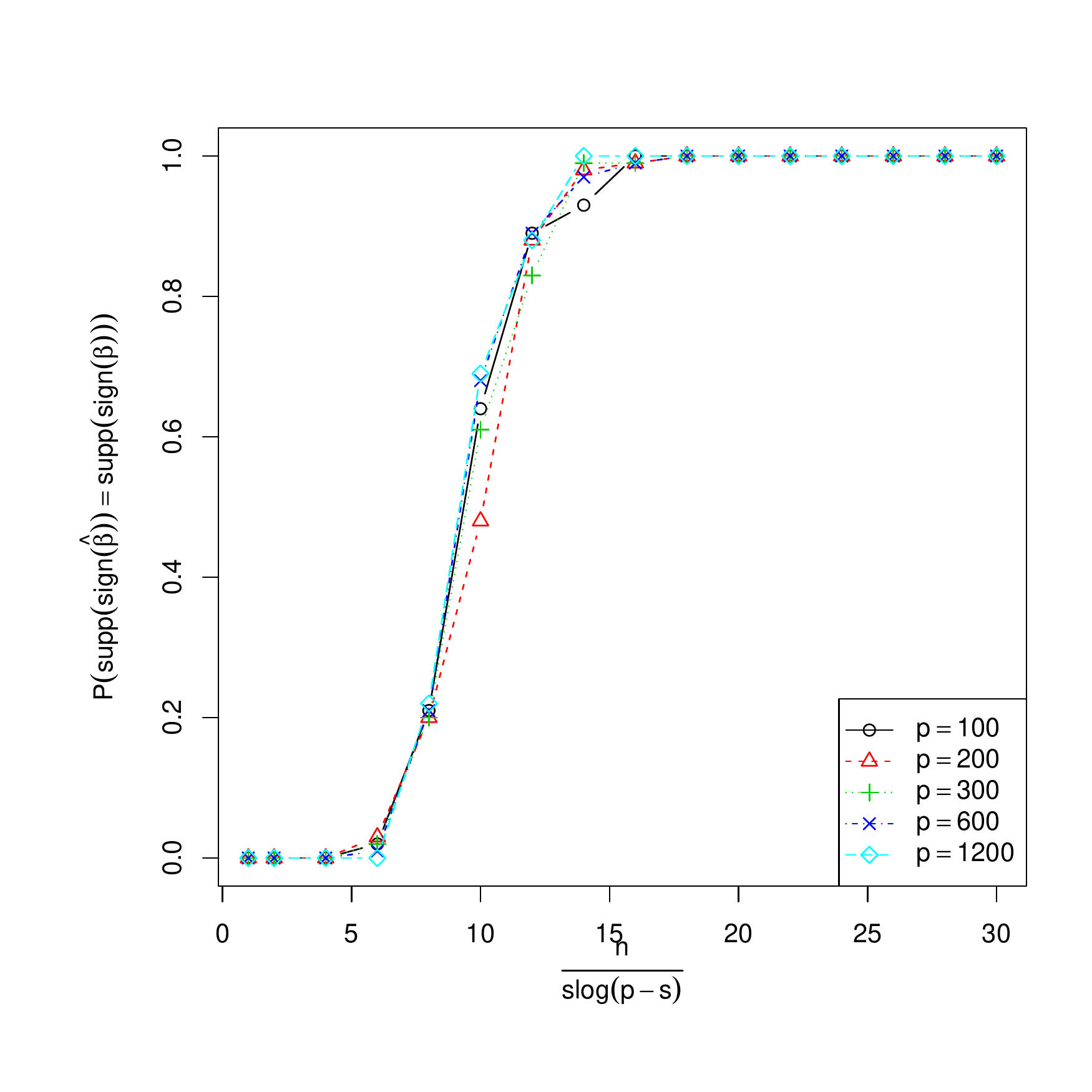}
  \caption{Model (\ref{sinmodel})}
  \label{fig:sfig1}
\end{subfigure}%
\begin{subfigure}{.5\textwidth}
  \centering
  \includegraphics[width=.8\linewidth]{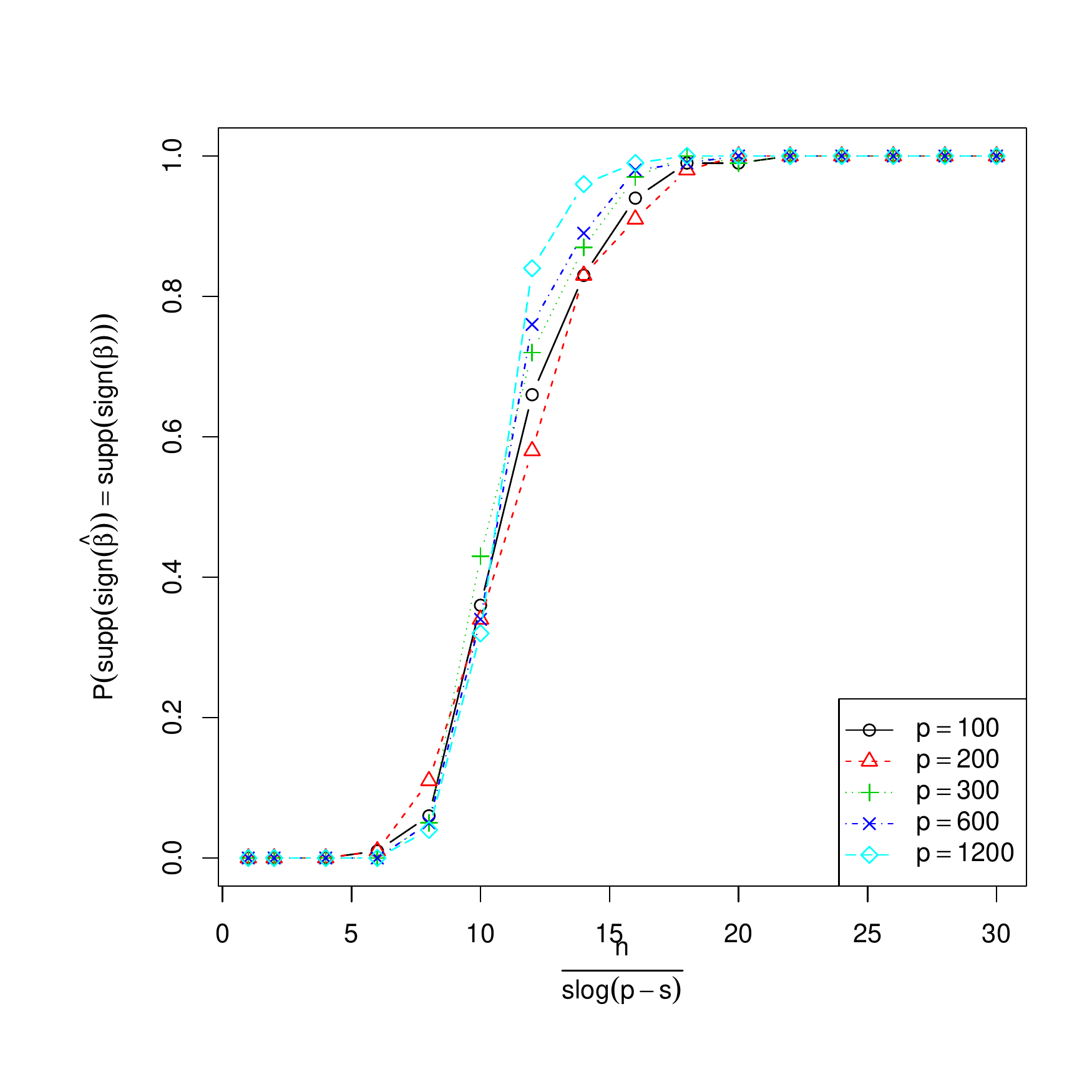}
  \caption{Model (\ref{X3model})}
  \label{fig:sfig2}
\end{subfigure}%

\begin{subfigure}{.5\textwidth}
  \centering
  \includegraphics[width=.8\linewidth]{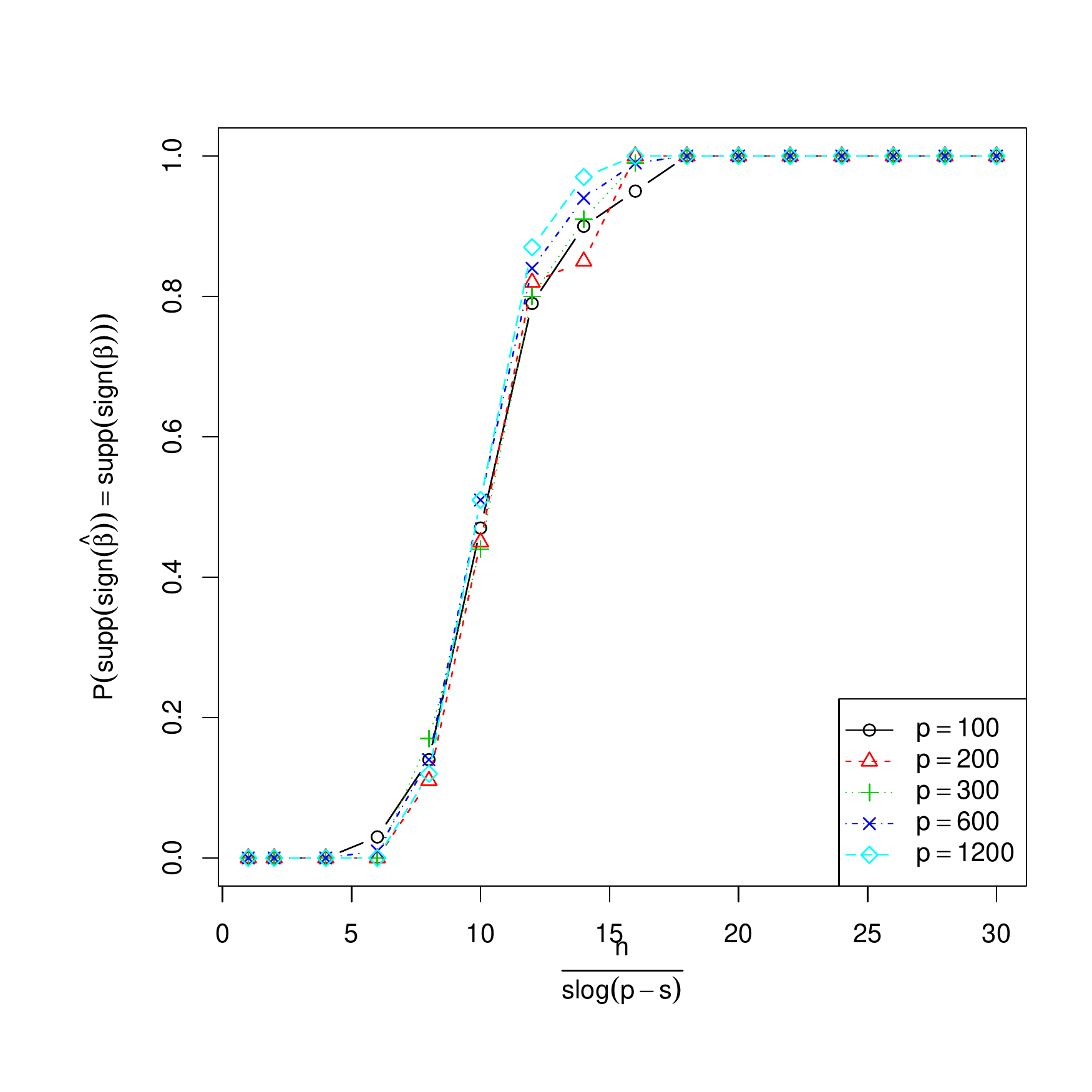}
  \caption{Model (\ref{Xe3model})}
  \label{fig:sfig3}
\end{subfigure}%
\begin{subfigure}{.5\textwidth}
  \centering
  \includegraphics[width=.8\linewidth]{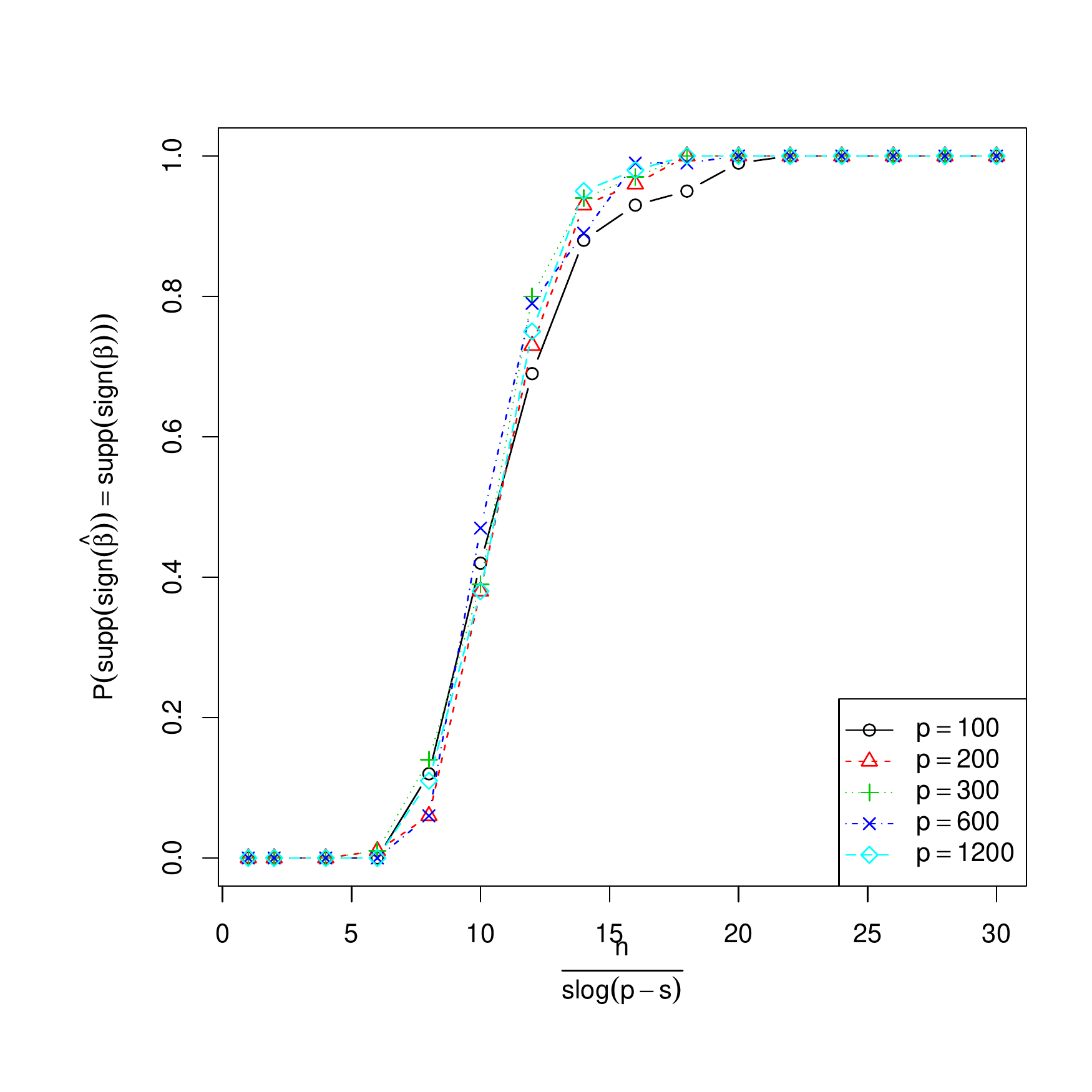}
  \caption{Model (\ref{regmodel})}
  \label{fig:sfig4}
\end{subfigure}%
\end{figure}

On the $X$-axis we have plotted the rescaled sample size $\frac{n}{s \log(p-s)}$ and on the $Y$-axis is the estimated  probability of successful signed support recovery. We would refer to these curves as efficiency curves (EC). The EC plots in the case $s = \sqrt{p}$, are very similar to the corresponding plots in the regime $s = \log(p)$, which can be seen in Figure \ref{log:DT:SIR} in Appendix \ref{app:extra:fig}. Both EC plots are in concordance with the predictions from our theoretical results. We can distinctly see the phase transition occurring in approximately the same place regardless of the values of the dimension $p$.

Next, we present the corresponding ECs for Algorithm \ref{SDPforSIR}. We used the code from an efficient implementation of program (\ref{SDPform}), as suggested in \cite{zhang2011large}. The code was kindly provided to us by the authors of \cite{zhang2011large}. In Figure \ref{SDPfigureslog} we plot the ECs for the four models in the case when $s = \sqrt{p}$. Due to running time limitations we only show scenarios where $p \in \{100,200,300\}$.

\begin{figure}[H]\caption{Efficiency Curves for SDP, $s = \sqrt{p},~ \Gamma = \frac{n}{s\log(p-s)} \in [0,40]$} \label{SDPfigureslog}
\begin{subfigure}{.5\textwidth}
  \centering
  \includegraphics[width=.8\linewidth]{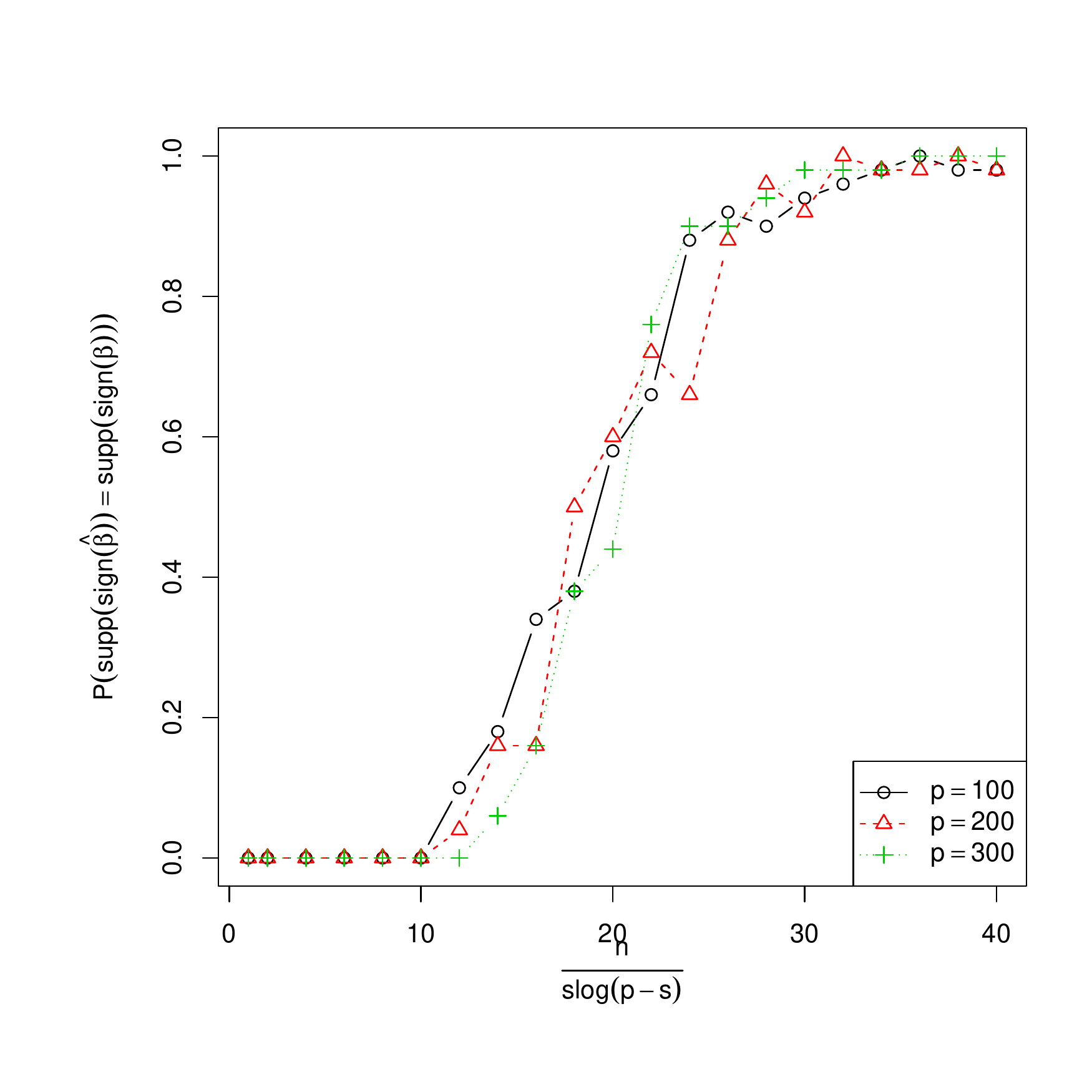}
  \caption{Model (\ref{sinmodel})}
  \label{fig:sfig1}
\end{subfigure}%
\begin{subfigure}{.5\textwidth}
  \centering
  \includegraphics[width=.8\linewidth]{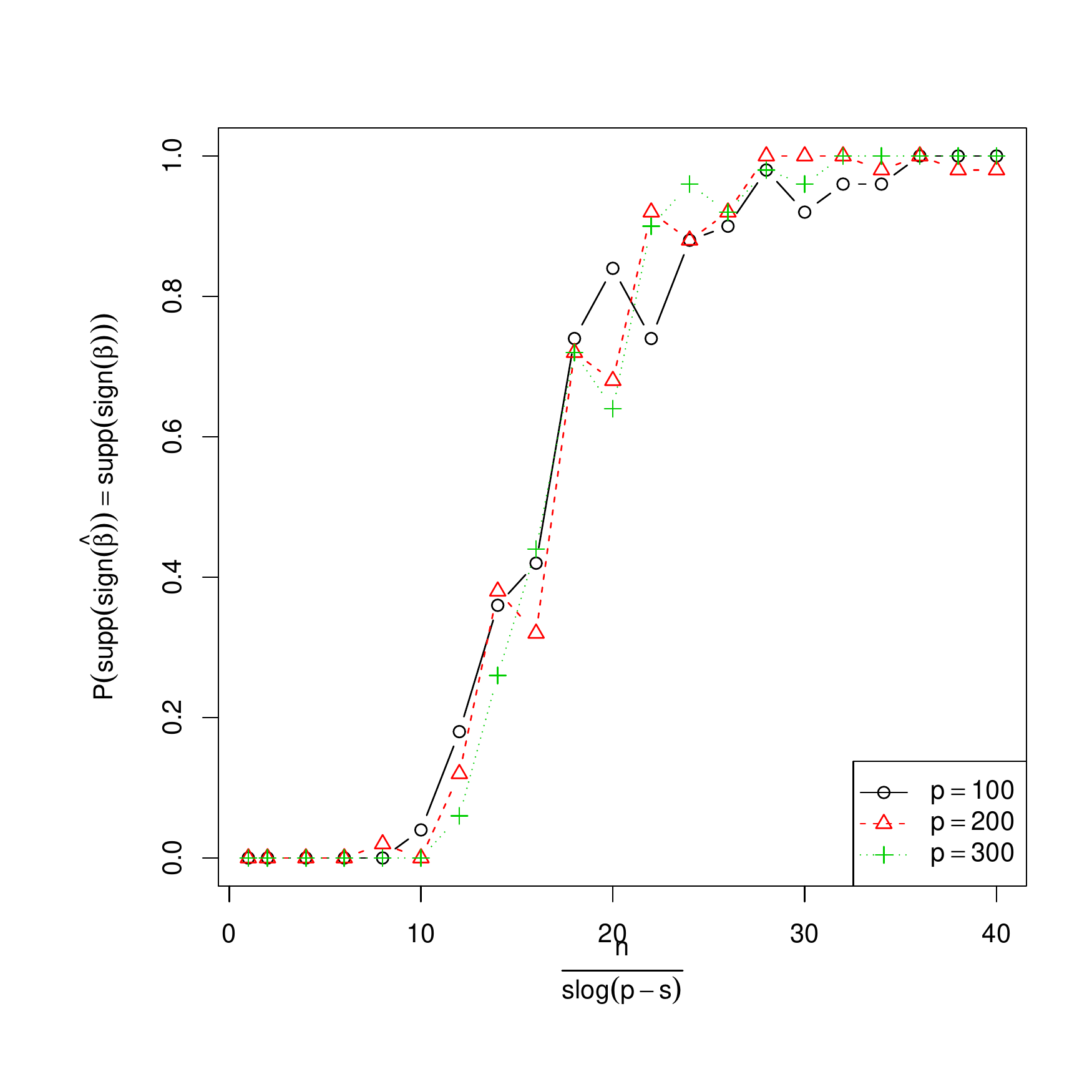}
  \caption{Model (\ref{X3model})}
  \label{fig:sfig2}
\end{subfigure}%

\begin{subfigure}{.5\textwidth}
  \centering
  \includegraphics[width=.8\linewidth]{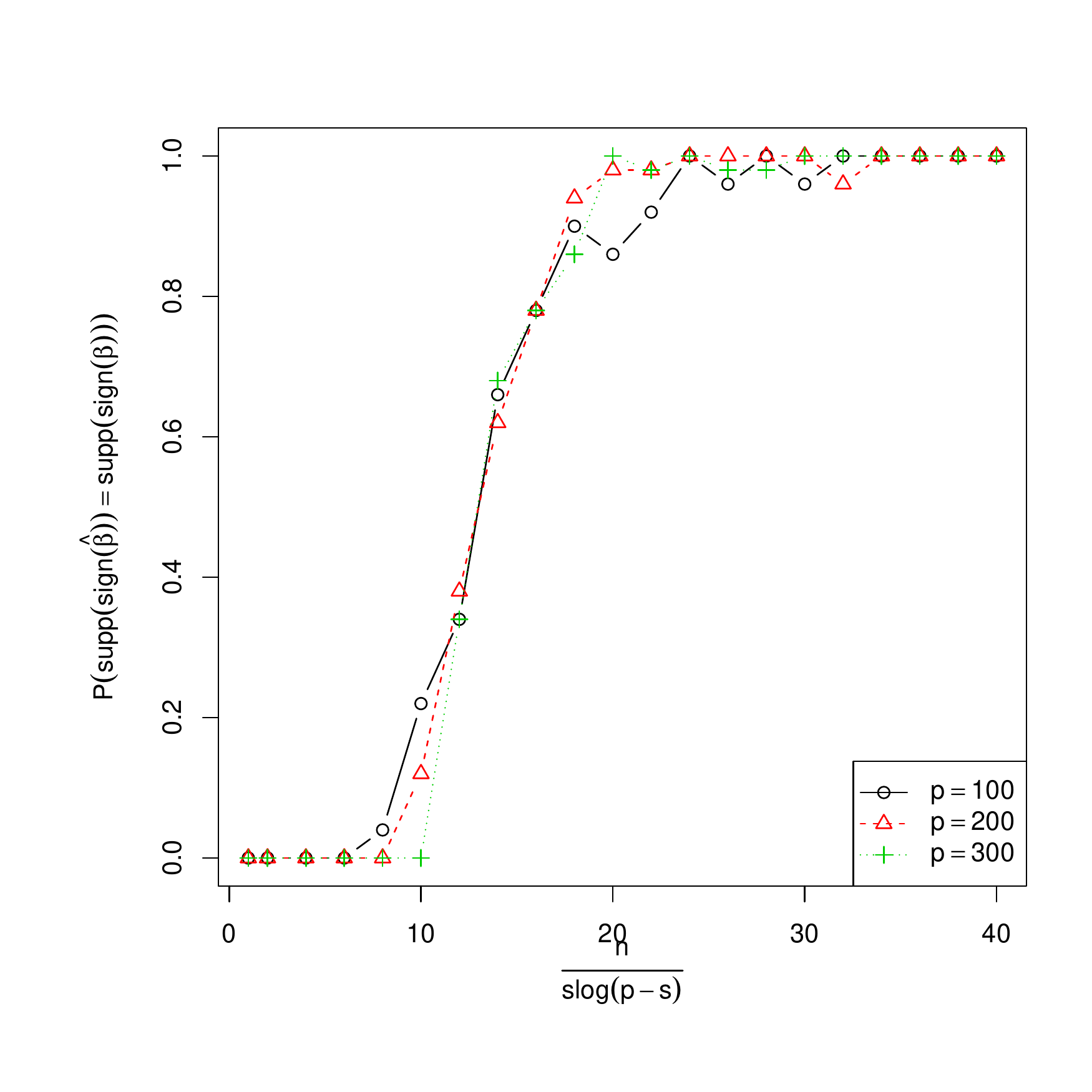}
  \caption{Model (\ref{Xe3model})}
  \label{fig:sfig3}
\end{subfigure}%
\begin{subfigure}{.5\textwidth}
  \centering
  \includegraphics[width=.8\linewidth]{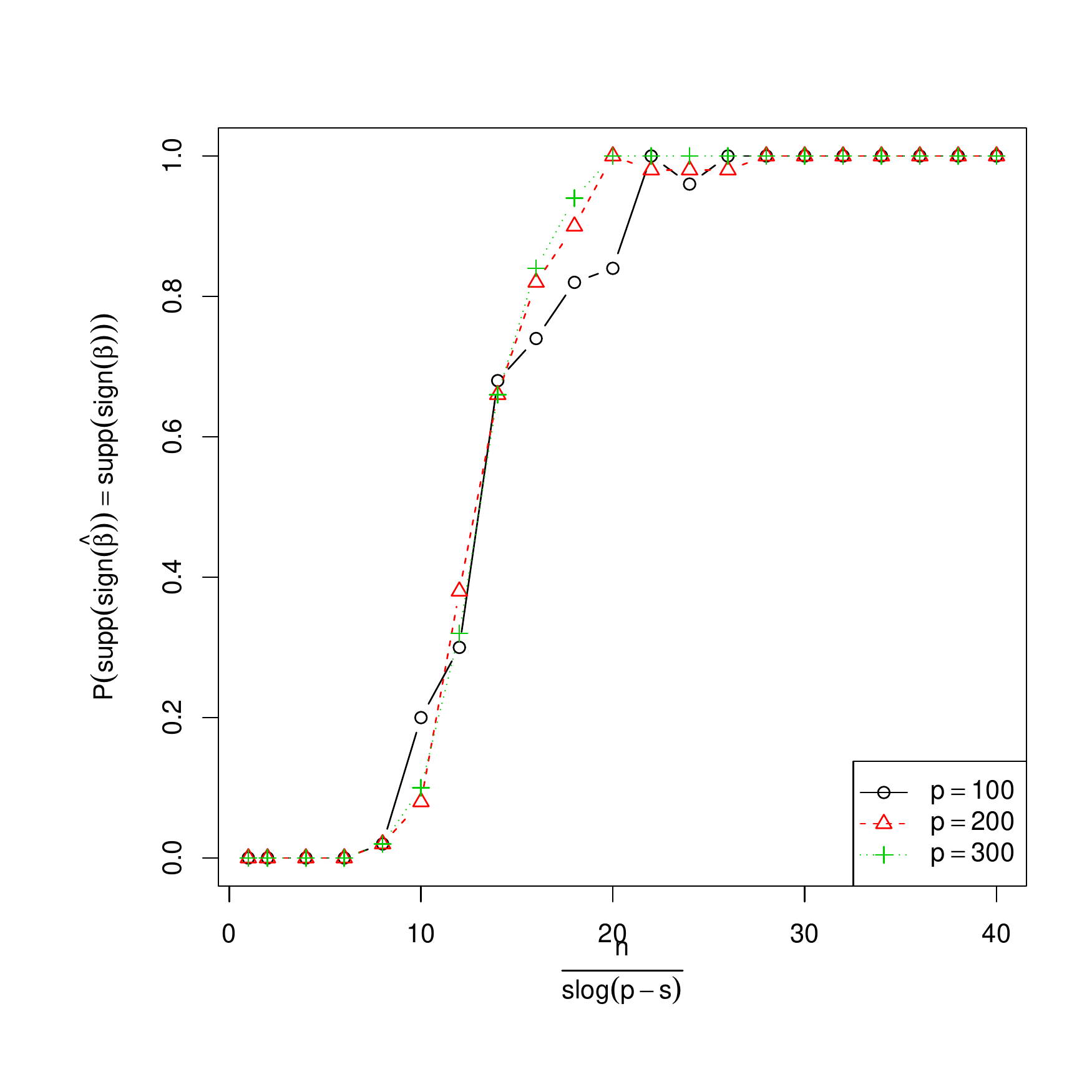}
  \caption{Model (\ref{regmodel})}
  \label{fig:sfig4}
\end{subfigure}%
\end{figure}

Here we have again used $H = 10$ in all scenarios, for simplicity. We observe that phase transitions are occurring in all of the models, and the signed support is being correctly recovered for large enough values of $\Gamma = \frac{n}{s \log(p-s)}$. Plots for the setting $s = \log(p)$ are provided in Figure \ref{log:SDPfigureslog}. An important empirical observation is that the constant $\Omega$ needs to be higher when running the SDP algorithm compared to when running the DT-SIR. This fact coupled with the much slower run-times of the SDP algorithm speak in favor of using the DT-SIR algorithm in practice.

In addition to the above simulations, we also performed numerical studies over the same set of models (\ref{sinmodel}), (\ref{X3model}), (\ref{Xe3model}) and (\ref{regmodel}), where the coefficients of the vector $\bbeta$ were not of equal magnitude. In our second set of simulations we chose $\bbeta = \frac{\widetilde \bbeta}{\|\widetilde \bbeta\|_2}$, where the entries of $\widetilde \bbeta$ were drawn in each simulation as 
\begin{align}\label{random:beta:specification}
\tilde \beta_j = U_j \mbox{ for } j \leq \lfloor s/2 \rfloor, ~~  \tilde \beta_j = -U_j \mbox{ for }s/2 < j \leq s, \mbox{ and }\tilde \beta_j = 0 \mbox{ otherwise,}
\end{align}
 where $U_j$ for $j \in [s]$ are i.i.d. uniform $U(\frac{1}{2},1)$ random variables. The results are similar to the ones reported above, although the efficiency curves required somewhat higher values of the rescaled sample size $\Gamma$, and extra variability can be seen in the curves due to the random choice of $\bbeta$. This is to be expected, since under such a generation scheme, the lowest signal will be of lower magnitude compared to the uniform signal case. We show the results for the case $s = \sqrt{p}$ in Appendix \ref{app:extra:fig} under Figures \ref{sqrt:DT:random:beta} and \ref{sqrt:SDP:random:beta}. Due to space considerations, we do not show results for the case $s = \log p$, as the results are comparable to the results plotted in Figures \ref{log:DT:SIR} and \ref{log:SDPfigureslog}.

\subsection{Real Data Example} \label{real:data:example}

In this sub-section we will illustrate the practical benefits of our algorithm on a real dataset. Our dataset is the mouse embryonic stem cells (ESCs) dataset, and has been previously analyzed by \cite{zhong2012correlation, jiang2014variable}. 

The outcome variable is the expression levels of $12,408$ genes, and is derived using RNA-seq technology in mouse ESCs \citep{cloonan2008stem}. The predictors correspond to $312$ transcription factors (TFs). From the predictors, $12$ TFs are known to be associated with different roles in the ES-cell biology as constituents of important signaling pathways, self renewal regulators, and key reprogramming factors, and the remaining $300$ are putative mouse TFs compiled from the TRANSFAC database. For each gene and TF, the predictor matrix contains a score which is the \textit{transcription factor association strength} score \citep{chen2008integration, ouyang2009chip} for the $12$ known TFs, and motif matching scores for the remaining $300$ genes which were introduced by \cite{zhong2012correlation}. Hence, the design matrix $\mathbb{X} = (\bX_i\T)\T_{i \in [n]}$ is a $n \times p$ matrix, where $n = 12,408$ and $p = 312$, each entry of which represents the score of a TF for the corresponding gene. The EFCs dataset is not truly a high-dimensional dataset in the sense $n < p$, but serves to illustrate that our algorithms can successfully perform variable selection. Since $\mathbb{X}$ is coming from a real dataset with non-homogeneous scores across the predictors, we wouldn't expect it to be centered with identity covariance. To deal with this problem we use the statistics $\widetilde \Vb_{jk}$ in place of $\Vb_{jk}$ where the matrix $\widetilde \Vb$ is defined as:
$$
\widetilde \Vb := \widehat \bSigma^{-1/2} \frac{1}{H} \sum_{h = 1}^H (\overline \bX_h - \overline \bX) (\overline \bX_h - \overline \bX)\T\widehat \bSigma^{-1/2},
$$
and $\widehat \bSigma^{-1/2} := \Big[n^{-1} \sum_{i \in [n]} (\bX_i - \overline \bX)(\bX_i - \overline \bX)\T\Big]^{-1/2}$ is the symmetric square root of the covariance matrix of $\mathbb{X}$. Notice that in this dataset $\widehat \bSigma^{-1/2}$ can be estimated without the application of sparse inverse procedures since $n > p$. As in the simulation section we used $H = 10$ slices. 

We ran the DT and SDP procedures (i.e. Algorithms \ref{DTalgo} and \ref{SDPforSIR}) on the ESCs dataset, and the DT procedure selected $28$ TFs which it found associated with the outcome, while the SDP procedure selected $36$. In Table \ref{ranking:twelve:TFs} we compare the rankings of the $12$ known ES-cell TFs within the TFs selected by DT and SDP algorithms to the same rankings produced by the procedures SIRI-AE \citep{jiang2014variable} and COP \citep{zhong2012correlation}. 

\begin{table}[ht]\caption{Rankings of the $12$ known TFs among the selected TFs by different algorithms}
\centering
\begin{tabular}{rrrrr}
  \hline
 TFs names & DT & SDP & SIRI-AE & COP \\ 
 \hline
E2f1 &   1 &   1 &   1 &   1 \\ 
  Zfx &   2 &   3 &   3 &   3 \\ 
  Mycn &   3 &   2 &   4 &  10 \\ 
  Klf4 &   7 &   4 &   5 &  19 \\ 
  Myc &   4 &   7 &   6 & -- \\ 
  Esrrb &  14 &  10 &   8 & -- \\ 
  Oct4 &  20 &  20 &   9 &  11 \\ 
  Tcfcp2l1 &   8 &   8 &  10 &  36 \\ 
  Nanog &  11 &  17 &  14 & -- \\ 
  Stat3 &  -- &  23 &  17 &  20 \\ 
  Sox2 &  24 & --  &  18 & -- \\ 
  Smad1 &  25 &  25 &  32 &  13 \\ 
   \hline
\end{tabular}
\label{ranking:twelve:TFs}
\end{table}

The top $10$ highly ranked genes of SIRI-AE, DT, SDP and COP contain $8,6,7,3$ known ES-cell TFs respectively. The DT algorithm is the most parsimonious out of all procedures (SIRI-AE and COP have reportedly selected $34$ and $42$ TFs correspondingly), but nevertheless includes all of the known TFs with the exception of the Stat3. The DT algorithm is additionally able to capture both Nanog and Sox2 which are generally believed to be the master ESC regulators. These TFs are completely missed by COP and the Sox2 TF was omitted by SDP. The SIRI-AE algorithm is the only algorithm which is able to recognize all of the $12$ known TFs. We would like to mention that compared to our work the SIRI-AE algorithm is designed to work in a lower dimensional setup, and is able to capture interactions beyond the single-index models, which could help explain its better performance on the ESCs dataset.

All in all, we believe the results of our algorithms are satisfactory, and illustrate that both DT and SDP can be successfully applied in practice. As a remark, in a truly high-dimensional setting when $n < p$, instead of the estimate $\widehat \bSigma^{-1/2}$, in practice one might resort to sparse estimators produced by procedures such as CLIME \citep{cai2011constrained}, and use the corrected statistics $\widetilde \Vb$ in place of $\Vb$.

\section{Proof of Theorem \ref{DTthm}}  \label{secpfthm1}
We will prove this theorem under slightly more general conditions:
\begin{itemize}
\item[i.] Let $\bX^j$, $j\in[p]$ be centered, sub-Gaussian random variables with $\max_{j \in [p]}\|\bX^j\|_{\psi_2} \leq \mathcal{K}$.
\item[ii.] For $j \in \cS_{\bbeta}$, we assume that $\Var[\bm_j(Y)] \geq \frac{C_V}{s}$.
\item[iii.] For $j \in \cS^c_{\bbeta}$, we assume that $\bm_j(Y) = \e[\bX^j|Y] = 0$ a.s.
\end{itemize}
It is easy to see for that for SIM in $\mathcal{F}_{A}$, the above conditions are satisfied in the case when $\bX \sim N_p(0, \mathbb{I}_p)$. We start with a high level outline of the proof. 
As we pointed out, the key argument is to show a deviation inequality for $\Vb^{jj}$, i.e. we will show that $|\Vb^{jj} - \var[\bm_j(Y)]|$ is small (e.g. $<\frac{1}{3}\var[\bm_{j}(Y)]$) with high probability provided that $\Gamma=\frac{n}{s\log(p-s)}$ is large enough. Note that:
\begin{align}
|\Vb^{jj} - \var[\bm_j(Y)]| & = \left| \frac{1}{H}\sum_{h = 1}^H \left(\overline \bX_{h}^j \right)^2 - \int \bm^2_j(y)p_Y(y) dy \right|.\label{varform}
\end{align}
We will show below (see (\ref{forms12})), that the above expression is well approximated by:
\begin{align}
  \left| \frac{1}{H}\sum_{h = 1}^H\left(\overline \bX_{h}^j \right)^2 -  \sum_{h = 1}^H  (\bmu^j_h)^2 \p( Y \in S_h) \right|, \label{varform2}
\end{align}
under sliced stability (where $Y_{(0)} = -\infty$). Intuitively, we have 1)  $\p(Y \in S_h) \approx \frac{1}{H}$ and 2) $\overline \bX^j_{h} \approx \bmu_h^j$, which shows that \eqref{varform2} should be close to 0. To rigorously validate this intuition, we use the Dvoretzky-Kiefer-Wolfowitz (DKW) inequality \citep{massart1990tight} for 1) and concentration inequalities for 2). 

Note that the probability $\p(Y \in S_h)$ is a random variable, where the randomness comes from the two endpoints --- $Y_{(m (h-1) )}$ and $Y_{(m h)}$ of the interval $S_h$. Recall that $F_n$ is the empirical distribution function of $Y$, based on the sample $Y_i$. By the DKW inequality, we have that $\p(\sup_{y} |F_n(y) - F(y)| > \epsilon) \leq 2 \exp(- 2 n \epsilon^2)$, which in addition to the fact that $Y$ has a continuous distribution, implies that for all $h$ we have:
\begin{align} \label{DKWbound}
\frac{1}{H}   - 2 \epsilon \leq \p\left(\frac{h-1}{H} < F_n(Y) < \frac{h}{H} \right) \leq \p\left(\frac{h-1}{H} \leq F_n(Y) \leq \frac{h}{H} \right)  \leq \frac{1}{H}  +  2 \epsilon,
\end{align}
on an event with probability at least $1 - 2 \exp(- 2 n \epsilon^2)$. Denote the event where (\ref{DKWbound}) holds with $E$.

Next, notice that for $h \in [H-2]$ and a random permutation $\pi : [m-1] \mapsto [m-1]$, the collection random variables $\bX^j_{h,\pi_i}, i \in [m-1]$ are conditionally i.i.d. given $(Y_{(m(h-1))}, Y_{(mh)})$, and for $h = H-1$ and a random permutation $\pi : [m] \mapsto [m]$ the random variables $\bX^j_{h,\pi_i}, i \in [m]$ are conditionally i.i.d. given $Y_{(m(H-1))}$, as they can be generated via simple rejection sampling. Hence conditionally on $(Y_{(mh)})_{h \in [H-1]}$ the sample means $\overline \bX^j_{h, 1:(m-1)}$ have corresponding means $\bmu^j_{h}$ when $h \in [H-1]$ and $\overline \bX^j_{H}$ has a mean of $\bmu^j_{H}$. For consistency of notation we will discard a point at random in the $H$\textsuperscript{th} slice, and WLOG (upon re-labeling the points) we assume that the discarded point is $\bX_{H,m}$. Next, we formulate the following concentration result for the sliced means, which we show in Appendix \ref{technical:proofs}: 
\begin{lemma} \label{slicedmeanconc} Let $G \subseteq [p]$. On the event $E$, for $\eta > 0$ we have the following:
\begin{align} \label{maxmeanslicebound}
\p\left(\max_{j \in G, h \in [H]} \left|\overline \bX_{h,1:(m-1)}^j  - \bmu_h^j \right| > \eta \right) \leq 2 |G| H \exp\left(- \frac{\eta^2 (m-1)}{C_1\mathcal{K}^2q^{-1} + C_2\mathcal{K} \eta}\right),
\end{align} 
where $q = \frac{1}{H} - 2\epsilon$, for some absolute constants $C_1, C_2 > 0$.
\end{lemma}
Set $G = [p]$ and denote with $\widetilde E$ the event on which we have 
$$\max_{j \in [p], h \in[H]} \left|\overline \bX_{h,1:(m-1)}^j  - \bmu_h^j \right| \leq \eta.$$
By (\ref{maxmeanslicebound}), (\ref{DKWbound}) and the union bound we have that:
\begin{align*}\p(\widetilde E) & \geq 1 - 2 pH \exp\left(- \frac{\eta^2 (m-1)}{C_1\mathcal{K}^2q^{-1} + C_2\mathcal{K} \eta}\right) - 2 \exp(- 2 n \epsilon^2).
\end{align*}

Next we move on, to show that (\ref{varform}) is close to (\ref{varform2}) on the event $E$, as well as we collect two straightforward inequalities in the following helpful:
 \begin{lemma} \label{simplelemma} Assume that the sliced stability condition (\ref{slicedstab}) holds. Then we have the following inequalities holding on the event $E$, for large enough $H$, and small enough $\epsilon$:
\begin{align}
|\Vb^{jj} - \var[\bm_j(Y)]| & \leq  \left| \frac{1}{H}\sum_{h = 1}^H\left(\overline \bX_{h}^j \right)^2 -  \sum_{h = 1}^H  (\bmu^j_h)^2 \p(Y \in S_h) \right| \label{forms12} \\& + \underbrace{\frac{CH^\kappa}{s} \left(\frac{1}{H} + 2 \epsilon \right)}_{B_1}, \nonumber \\
\sum_{h = 1}^H (\bmu_h^j)^2 & \leq \underbrace{\frac{\frac{C_V}{s} +B_1}{\left(\frac{1}{H} - 2\epsilon\right) }}_{B_2}, \label{expectationsquared}\\
 \sum_{h = 1}^H |\bmu_h^j|  & \leq \underbrace{\frac{\sqrt{\frac{C_V}{s} + B_1}}{ \left(\frac{1}{H} - 2\epsilon\right) }}_{B_3}. \label{expectationabs}
\end{align}
{\bf Note.} We refer to the constants from (\ref{slicedstab}) as $C$ and $\kappa$, dropping the dependence on $K$ and $M$ for brevity, and in fact $C = C(l,K,M)C_V$.
\end{lemma}
Note that by an elementary calculation --- using (\ref{DKWbound}) and Lemma \ref{simplelemma}, on the event $\widetilde{E}$ we get:
\begin{align} \label{bigbound}
|\Vb^{jj} - \var[\bm_j(Y)]|  & \leq \frac{1}{H}\sum_{h = 1}^H \left|\left(\overline \bX_{h}^j \right)^2 - \frac{(m-1)^2}{m^2} (\bmu_h^j)^2 \right|  \\
& + \underbrace{\left(2\epsilon + \frac{1}{H} - \frac{(m-1)^2}{H m^2}\right)B_2}_{I_1} + \underbrace{B_1}_{I_2},
\end{align}
where we used (\ref{forms12}), the triangle inequality and (\ref{expectationsquared}). Consider the following: 
\begin{lemma}\label{calculation} There exists a subset $\widetilde{\widetilde{E\, }} \subset \widetilde{E}$ such that $ \p( {\widetilde{E}\setminus \widetilde{\widetilde{E\, }} }) \leq p\exp(-c \frac{\tau^2n}{4\mathcal{K}^4})$, for some fixed $\tau \in [0, 2 \mathcal{K}^2)$ on which we have the following bound for all $j \in [p]$:
\begin{align}
\frac{1}{H}\sum_{h = 1}^H \left|\left(\overline \bX_{h}^j \right)^2 - \frac{(m-1)^2}{m^2} (\bmu_h^j)^2 \right| & \leq \label{boundlemmaxmean} \underbrace{\frac{(2 \mathcal{K}^2 + \tau)}{m}}_{I_3} + \underbrace{\frac{2\sqrt{2 \mathcal{K}^2 + \tau}}{\sqrt{m}}\sqrt{ 2\eta^2 + 2\frac{B_2}{H}}}_{I_4} \\
& +  \underbrace{\eta^2}_{I_5} +  \underbrace{2\eta\frac{B_3}{H}}_{I_6}\nonumber.
\end{align}
\end{lemma}
We defer the proof of this lemma to the appendix. Next, we provide exact constants, such that $|\Vb^{jj} - \var[\bm_j(Y)]| \leq \frac{C_V}{\rho s}$ for some constant $\rho > 0$ and all $j \in [p]$, so that the probability of the event $\widetilde{\widetilde{E\, }}$ still converges to 1. The remarkable phenomenon here is that the number of slices $H$, can be selected so that it is a constant, which might seem counterintuitive. Select the constants in the following manner:
\begin{align}
H & = \max\left\{M, \left(\frac{\gamma CK}{C_V}\right)^{\frac{1}{1 - \kappa}}\right\}, \label{Hconst} \\
\epsilon & = \min \left\{ \frac{K-1}{2 H}, \frac{1-l}{2H}, \frac{l}{4(\gamma + 1) H}\right\}, \label{epsconst}\\
\eta & = \frac{l \sqrt{C_V}}{2 \sqrt{\gamma (\gamma + 1)}\sqrt{s}}, \label{etaconst} \\
m & \geq \max\left\{\frac{8 \gamma(\gamma + 1 + l)(2 \mathcal{K}^2 + \tau) s}{l C_V}, \frac{4(\gamma + 1)}{l}, \frac{2s \log(p-s)}{\Upsilon} + 1, \frac{8 \mathcal{K}^4 \log(p-s)}{\tau^2 c H}\right\}\label{mconst}
\end{align}
where $\gamma \geq 12$ is a positive constant and $\Upsilon = \frac{l^2 C_V}{4(\gamma(\gamma + 1)) \left(\mathcal{K}^2C_1 H l^{-1} + \frac{C_2 \mathcal{K}l \sqrt{C_V}}{2\sqrt{\gamma(\gamma + 1)}}\right)}$. Simple algebra shows that selecting these constants ensures the following inequality:
$$\max(I_1, I_2, I_3, I_4, I_5, I_6) \leq \frac{C_V}{\gamma s}.$$
By combining (\ref{bigbound}) and (\ref{boundlemmaxmean}) we arrive at:
\begin{align}
|\Vb^{jj} - \var[\bm_j(Y)]| \leq \frac{6 C_V}{\gamma s}, \label{varianceineq}
\end{align}
as promised for $\rho = \frac{\gamma}{6}$. Next we proceed to show that the probability of the event $\widetilde{\widetilde{E\, }}$ converges to $0$. To achieve this, note that in (\ref{epsconst}) we chose $\epsilon$ to be a constant, so evidently $\exp(-n\epsilon^2) \rightarrow 0$. Moreover $p \exp(-c\frac{\tau^2 mH}{4 \mathcal{K}^4}) \leq \frac{p}{(p-s)^2} \rightarrow 0$ by (\ref{mconst}). Finally to show that $\p(\widetilde{\widetilde{E\, }}) \rightarrow 1$, we need:
$$
\frac{\eta^2 (m-1)}{C_1 (H^{-1} - 2 \epsilon)^{-1}\mathcal{K}^2 + C_2 \mathcal{K}\eta} - \log(p) - \log(H) \rightarrow +\infty.
$$
Noting that $H^{-1} - 2\epsilon \geq l H^{-1}$, $\eta \leq \frac{l \sqrt{C_V}}{2 \sqrt{\gamma(\gamma + 1)}}$ and that by (\ref{Hconst}) $H$ is fixed, the above expression is implied by:
$$
\Upsilon \frac{(m-1)}{s} - \log(p) \geq 2 \log(p-s) -  \log(p) \rightarrow +\infty,
$$
where we used (\ref{mconst}) and the fact that $2\log(p-s) \geq \log(p)$ asymptotically (since $s = O(p^{1-\delta})$ for $\delta > 0$). Hence we have guaranteed that  $\p(\widetilde{\widetilde{E\, }}) \rightarrow 1$. 

Now, note that for variables $j \not \in \cS_{\bbeta}$ we have $\bm_j(Y) = \e[\bX^j] = 0$, which implies that for all $h \in [H], j \in \cS_{\bbeta}^c$ we have $\bmu^j_h = 0$. Using (\ref{varianceineq}) and selecting $\gamma = 18$, we conclude that:
$$
\Vb^{jj} \geq \frac{2}{3} \frac{C_V}{s}, j \in \cS_{\bbeta} \mbox{   and   } \Vb^{jj} \leq \frac{C_V}{3 s}, j \in \cS^c_{\bbeta},
$$
and hence separation of the signals is asymptotically possible, as we claimed.

\section{Proof of Theorem \ref{SDPthm}} \label{secpfthm2}

In this section we show that the SDP relaxation will have a rank $1$ solution with high probability and moreover this solution will recover the signed support of the vector $\bbeta$. In contrast to Section \ref{secpfthm1}, here we make full usage of the fact that $\bX \sim N_p(0, \mathbb{I}_p)$. One simple implication of this, is for example the fact that $\max_j \|\bX^j\|_{\psi_2} < 1$, and hence we can set $\mathcal{K} = 1$. For the analysis of the algorithm we set the regularization parameter $\lambda_n = \frac{C_V}{2s}$.

To this end, we restate Lemma 5 from \cite{amini2008high}, which provides a sufficient condition for a global solution of the SDP problem:
\begin{lemma}  \label{SDPlemma} Suppose there exists a matrix $\Ub$ satisfying:
\begin{align}
\Ub_{ij} = \begin{cases} \sign(\widehat \bz_i) \sign(\widehat \bz_j), &\mbox{ if } \widehat \bz_i \widehat \bz_j \neq 0;\\
\in [-1,1], &\mbox{ otherwise}.
\end{cases}
\end{align}
Then if $\widehat \bz$ is the principal eigenvector of the matrix $\Ab - \lambda_n \Ub$, $\widehat \bz \widehat \bz\T$ is the optimal solution to problem (\ref{SDPform}).
\end{lemma}

Recall that the SIR estimate of the variance-covariance matrix has entries:
\begin{align*}
\Vb^{jk} = \frac{1}{H}\sum_{h = 1}^H \overline \bX^j_{h} \overline \bX^k_{h}.
\end{align*}
Denote with $\widetilde \Vb = \Vb - \lambda_n \Ub$, where $\Ub$ is a to be defined sign matrix from Lemma \ref{SDPlemma}. We furthermore consider the decomposition of $\widetilde \Vb$ into blocks --- $\widetilde \Vb_{\cS_{\bbeta}, \cS_{\bbeta}}$,  $\widetilde \Vb_{\cS_{\bbeta}^c, \cS_{\bbeta}}$, $\widetilde \Vb_{\cS_{\bbeta}^c,\cS_{\bbeta}^c}$. Here, these three matrices are sub-matrices of the matrix $\widetilde \Vb$ restricted to entries with indexes in the sets $\cS_{\bbeta}$ or $\cS_{\bbeta}^c$ correspondingly. We observe that $\Ub_{\cS_{\bbeta}, \cS_{\bbeta}} = \sign(\bbeta_{\cS_{\bbeta}})\sign(\bbeta_{\cS_{\bbeta}})\T$. 

We first focus on the $\Vb_{\cS_{\bbeta}, \cS_{\bbeta}}$ matrix. We calculate the value of the covariance of two coordinates $j,k \in \cS_{\bbeta}$:
\begin{align} \label{suppcov}
\cov[\bm_j(Y), \bm_k(Y)] & = \e[\bm_j(Y) \bm_k(Y)] = \operatorname{sign}(\bbeta_j) \operatorname{sign}(\bbeta_k) \e[\bm^2_k(Y)] \\
& = \bbeta_j \bbeta_k C_V,\nonumber 
\end{align}
where we used that $\operatorname{sign}(\bbeta_j) \bm_j(Y) =  \operatorname{sign}(\bbeta_k)\bm_k(Y)$, which follows by noticing that the distribution of $\bX^j | Y$ is the same as the distribution of $\bX^k | Y$ except the potential difference in the signs of the coefficients, because of the symmetry in the problem. 

Next, let $G = \cS_{\bbeta}$ (hence $|G| = s$) in Lemma \ref{slicedmeanconc} to obtain that:
\begin{align} \label{boundonSbeta}
\max_{j \in \cS_{\bbeta}, h \in[H]} \left|\overline \bX_{h,1:(m-1)}^j  - \bmu_h^j \right| \leq \eta,
\end{align}
with probability at least $1 -  2 sH \exp\left(- \frac{\eta^2 (m-1)}{C_1 q^{-1} + C_2 \eta}\right) - 2 \exp(-2n\epsilon^2)$. Let $\widetilde E_{\cS_{\bbeta}} \subset E$, be the event where (\ref{boundonSbeta}) holds. We proceed with formulating a bound similar to Lemma \ref{calculation}, but for the covariance:
\begin{lemma} \label{covbound} There exists an event $\widetilde{\widetilde{E\, }}_{\cS_{\bbeta}}$ with $ \p(\widetilde E_{\cS_{\bbeta}} \setminus \widetilde{\widetilde{E\, }}_{\cS_{\bbeta}}) \leq s \exp(-c \frac{\tau^2n}{4})$ ($c > 0$ is an absolute constant), for some fixed $\tau \in (0, 2]$, such that for all $j,k\in \cS_{\bbeta}$, we have the following inequality:
\begin{align} \label{covboundineq}
 \left|\Vb^{jk} - \cov(\bm_j(Y), \bm_k(Y)) \right|& \leq \left(2\epsilon + \frac{1}{H} - \frac{(m-1)^2}{H m^2}\right)B_2 + B_1+   4\eta\frac{B_3}{H}\nonumber\\
& +  \frac{4(2 + \tau)}{m} +  \frac{4\sqrt{2 + \tau}}{\sqrt{m}}\sqrt{ 2\eta^2 + 2\frac{B_2}{H}} +  4\eta^2.
\end{align}
\end{lemma}

Let $H', \epsilon', \eta'$ are constants selected according to (\ref{Hconst}) ---  (\ref{etaconst}), with $\mathcal{K} = 1$, correspondingly. Set $H = H', \epsilon = \epsilon', \eta = \frac{\eta'}{2}$, and take
\begin{align}
m & \geq \max\left\{32\frac{8 \gamma(\gamma + 1 + l)(2 + \tau) s}{l C_V}, \frac{8s \log(p - s)}{\Upsilon} + 1, \frac{8 \log(p - s)}{\tau^2 c H}\right\}\label{mnewboundSDP},
\end{align}
where $\Upsilon$ is specified as before with $\mathcal{K} = 1$.
Similarly to Section \ref{secpfthm1}, we can show that the following inequality: 
\begin{align} \label{errorbound}
\sup_{j,k \in \cS_{\bbeta}} \left|\Vb^{jk} - \cov(\bm_j(Y), \bm_k(Y)) \right| \leq \frac{6 C_V}{\gamma s},
\end{align}
holds on $\widetilde{\widetilde{E\, }}_{\cS_{\bbeta}} $, with the probability of $\widetilde{\widetilde{E\, }}_{\cS_{\bbeta}}$ tending to 1. To get the bound in (\ref{errorbound}), one can observe that with the choices of constants as above all $6$ terms in  (\ref{covboundineq}) are guaranteed to be smaller than $\frac{C_V}{\gamma s}$. Here similarly to Section \ref{secpfthm1}, $H$ is large enough but fixed. Having in mind the above inequality we consider the matrix $\widetilde \Vb_{\cS_{\bbeta},\cS_{\bbeta}}$:
$$
\widetilde \Vb_{\cS_{\bbeta},\cS_{\bbeta}} = \frac{C_V}{2} \bbeta_{\cS_{\bbeta}} \bbeta_{\cS_{\bbeta}}\T + \Nb,
$$
where $\Nb$ is some symmetric noise matrix. Note that by (\ref{suppcov}), (\ref{errorbound}) gives a bound on $\|\Nb\|_{\max}$. Next we make usage of Lemma \ref{wainwriglemma6}, which is a restatement of Lemma 6 in \cite{amini2008high} and can be found in the Appendix for the reader's convenience. We start by verifying that condition (\ref{Nmatrixver}) indeed holds for the matrix $\Nb$. We have that:
\begin{align} \label{noisematrixbound}
\|\Nb\|_{\max} = \| \Vb_{\cS_{\bbeta},\cS_{\bbeta}} - C_V \bbeta_{\cS_{\bbeta}} \bbeta_{\cS_{\bbeta}}\T\|_{\max} \leq \frac{6 C_V}{\gamma s},
\end{align}
with the last inequality following from (\ref{errorbound}).
Selecting $\gamma = 240$ bounds $\|\Nb\|_{\max}$ by $\frac{C_V}{40 s}$, as required in (\ref{Nmatrixver}). Thus, by Lemma \ref{wainwriglemma6} we conclude that:
\begin{itemize}
\item[(a)] For $\gamma_1 := \lambda_{\max}(\widetilde \Vb)$ and the second largest in magnitude eigenvalue of $\widetilde \Vb$ we have $\gamma_1 > |\gamma_2|$.
\item[(b)] The corresponding principal eigenvector of $\widetilde \Vb$ --- $\widetilde \bz_{\cS_{\bbeta}}$ satisfies the following inequality:
$$\left \|\widetilde \bz_{\cS_{\bbeta}} - \bbeta_{\cS_{\bbeta}}\right \|_{\infty} \leq \frac{1}{2\sqrt{s}}.$$
\end{itemize}

Next we show that the rest of the sign matrix $\Ub$, i.e. $\Ub_{\cS_{\bbeta}^c,\cS_{\bbeta}}$ and $\Ub_{\cS_{\bbeta}^c, \cS_{\bbeta}^c}$ can be selected in such a way, so that the blocks $\widetilde \Vb_{\cS_{\bbeta}^c, \cS_{\bbeta}}$ and $\widetilde \Vb_{\cS_{\bbeta}^c, \cS_{\bbeta}^c}$ are 0. For this purpose we select $\Ub_{\cS_{\bbeta}^c, \cS_{\bbeta}} = \frac{1}{\lambda_n}\Vb_{\cS_{\bbeta}^c, \cS_{\bbeta}}$ and $\Ub_{\cS_{\bbeta}^c, \cS_{\bbeta}^c} = \frac{1}{\lambda_n}\Vb_{\cS_{\bbeta}^c, \cS_{\bbeta}^c}$. Since it is clear that the vector $(\widetilde \bz_{\cS_{\bbeta}}, 0_{\cS_{\bbeta}^c})$ is the principal eigenvector of $\widetilde \Vb$, if $\Ub$ is a sign matrix, Lemma \ref{SDPlemma} will conclude that ---  $(\widetilde \bz \T_{\cS_{\bbeta}}, 0\T_{\cS_{\bbeta}^c})\T (\widetilde \bz\T_{\cS_{\bbeta}}, 0\T_{\cS_{\bbeta}^c})$ is the optimal solution to the optimization problem, which will in turn conclude our claim.

It remains to show that the specified $\Ub$ is indeed a sign matrix. Note that by Cauchy-Schwartz for $k \in \cS_{\bbeta}^c$ and any $j$, we have:
\begin{align} \label{CSsign}
\Vb^{jk} \leq \sqrt{\Vb^{jj}}\sqrt{\Vb^{kk}}.
\end{align}
From (\ref{errorbound}) if $j \in \cS_{\bbeta}$, we have that high probability:
$$
\Vb^{jj} \leq \frac{C_V}{s} +  \frac{6 C_V}{\gamma s} = \frac{(\gamma + 6)C_V}{\gamma s}.
$$
Hence, it is sufficient to select $m, H$ large enough so that: 
$$
\Vb^{kk} \leq \frac{\gamma C_V}{4 (\gamma + 6)s}.
$$
To achieve the above bound, we make usage of the following tail inequality, for $\chi^2$ random variables which we take from \cite{laurent2000adaptive} (see Lemma 1):
\begin{align*}
\p\left(\frac{\chi^2_H}{H}  \geq 1 + 2 \sqrt{\frac{x}{H}} + \frac{2 x}{H}\right) \leq \exp(-x).
\end{align*}
Note that, $\Vb^{kk} \sim \frac{1}{mH}\chi^2_H$ for $k \in \cS_{\bbeta}^c$. Thus applying the bound above we have 
\begin{align}\label{chisqbound}
\frac{1}{mH}\chi^2_H \leq \frac{1}{m} + \frac{2}{m} \sqrt{\frac{x}{H}} + \frac{2 x}{mH} \leq \frac{2}{m} + \frac{3 x}{mH},
\end{align}
with probability at least $\exp(-x)$. Applying (\ref{chisqbound}) it can be easily seen that by selecting:
$$
x = \frac{n \gamma C_V}{24 (\gamma + 6) s},
$$
we can ensure (after using (\ref{mnewboundSDP})) that $\Vb^{kk} \leq  \frac{\gamma C_V}{4 (\gamma + 6)s}$ for all $k \in \cS_{\bbeta}^c$. Requiring:
\begin{align} \label{nlowerboundSDP}
\frac{n \gamma C_V}{24 (\gamma + 6) s} \geq 2 \log(p-s),
\end{align}
ensures that the probability of the event is asymptotically $1$ from the union bound. This combined with (\ref{CSsign}) shows that the so defined matrix $\Ub$ is indeed a sign matrix, which concludes the proof.

\section{Discussion} \label{discsec} 
Sliced inverse regression has been widely applied in various scientific problems. Though it is a successful tool in terms of data visualization, and provably works when the dimension $p$ is not large, its behavior in the high-dimensional regime $p \gg n$ is much less well understood. 

In this paper, we studied the support recovery of SIM within the class of models $\mathcal{F}_{A}$. We demonstrated that the optimal sample size of this problem is of the order $s\log(p-s)$. Two unforeseeable results of our analysis might be of particular interest for future investigations. 

Recall that a central subspace of a pair of random variables $(Y,\bX)$ is the minimal subspace $\mathcal{S}$ such that $\ind Y \bX\big| P_{\mathcal{S}}\bX$. The first implication of our results, as we hinted in Remark \ref{rmk:on:H}, is that if we focus on estimating the support of the central subspace rather than consistently estimating the intermediate matrix,  better convergence rate might be expected. For instance, the results in \cite{lin2015consistency} show that, under mild conditions, the convergence rate of the SIR estimate $\Vb$ of $\var(\bbE[\vX|Y])$ is $O_{P}(\frac{1}{H}+\sqrt{\frac{p}{n}})$. In order to get a consistent estimation of the intermediate matrix $\var(\bbE[\vX|Y])$, the slices number $H$ has to be diverging. The result regarding the choice of $H$ in this paper suggests that it would be possible to get an improved convergence rate of estimating the central subspace which is expected to be $\sqrt{\frac{p}{n}}$. 

The second consequence is that estimating the central subspace might be easier than that of sparse PCA in terms of computational cost.  At first glance, the unknown nonlinear link function $f$ could bring in difficulties in determining the optimal rate, and it might be reasonable to expect that this difficulty will increase the computational cost in general. The results in this paper however, provide the opposite evidence --- as long as $s=O(p^{1-\delta})$ for some $\delta>0$, the computationally efficient algorithms DT-SIR and SDP approach can solve the support recovery problem. In other words, the tradeoff between computational and statistical efficiency for estimating the central subspace might occur in a more subtle regime, where $s\propto p$ and $n\propto p$. Finally combining our results with the results of \cite{lin2015consistency}, calculating the minimax rate of an estimator for the SDR space in model \eqref{eqn:model:sir_single_index} appears to be a plausible exercise, and is left for future work.

\section*{Acknowledgments} This research was conveyed while the first author was a graduate student in the Department of Biostatistics at Harvard University. Lin's research is supported by the Center of Mathematical Sciences and Applications at Harvard University. Liu's research is supported by the NSF Grant DMS-1120368 and NIH  Grant R01 GM113242-01. We thank Professor Tianxi Cai, Harvard University and Professor Noureddine El Karoui, University of California, Berkeley for valuable discussions which led to improvements of the present manuscript. In addition, the authors are grateful to Professor Laurent El Ghaoui, University of California, Berkeley and Dr Youwei Zhang for providing the Matlab code used for solving the SDP problem. The authors also express gratitude to the referee and anonymous reviewers for raising important points, which further bettered the manuscript.

\appendix

\section{Technical Proofs}\label{technical:proofs}

\begin{proof}[Proof of Lemma \ref{signallemma}]
First take any vector $\bb \in \mathbb{R}^p$ such that $\bb \perp \bbeta$. We have:
\begin{align*}
\e[\bX| Y] \T \bb = \e[\bX\T \bb]  = 0,
\end{align*}
since $\bX\T \bb$ is independent of  $\bX\T \bbeta$ and $\varepsilon$. This implies that 
\begin{align}
\e[\bX| Y] = c(Y) \bbeta, \label{condexpidentity}
\end{align}
for some real valued function $c$. Since $\bbeta$ is a unit vector it follows that $c(Y) = \e[\bX\T \bbeta |Y]$. Take $j \in \cS_{\bbeta}$ and apply (\ref{condexpidentity}) to get:
\begin{align} \label{varcalc}
\var[\bm_j(Y)] = \frac{\var [\e[\bX\T \bbeta |Y]]}{s}.
\end{align}
Combining the observation above with the following two inequalities:
$$
A \leq \var(\e[\bX\T\bbeta | Y]) \leq \var(\bX\T\bbeta) = 1,
$$
gives the desired result.
\end{proof}

\begin{proof}[Proof of Proposition \ref{slicedstabsuffcond}] 
We first note that without loss of generality we can consider the function $\widetilde \bm$ to be non-negative, at the price of potentially shrinking the interval $(B_0, \infty)$ to $(B_0 + \eta, \infty)$ by any $\eta > 0$. To see this fix an $\epsilon > 0$, and define $\widetilde \bm'(x) = \widetilde \bm(x) -\widetilde \bm(B_0 + \eta)$ for $x \in (B_0 + \eta, \infty)$. Then since (\ref{mglobal}) holds on $(-\infty, -B_0) \cup  (B_0, \infty)$, clearly:
$$
|\bm(x) - \bm(y)| \leq |\widetilde \bm' (|x|) - \widetilde \bm'(|y|)|, \mbox{ for } x, y \in (-\infty, -B_0 - \eta) \mbox { or }  (B_0 + \eta, +\infty).
$$
By the convexity of the map $x \mapsto x^{2 + \xi}$ we have $\widetilde{\bm}'(x)^{2 + \xi} \leq 2^{1 + \xi}(|\widetilde{\bm}(x)|^{2 + \xi} + |\widetilde \bm(B_0 + \eta)|^{2 + \xi})$ and hence $\e[|\widetilde \bm'(|Y|)|^{(2 + \xi)}] < \infty$. Finally by definition $\widetilde \bm'$ is non-negative and non-decreasing on $(B_0 + \eta, \infty)$. 

Next note that if $Y$ has a bounded support, this proposition clearly follows from assumption (\ref{msmoothness}) alone. Thus, without loss of generality we assume that $Y$ has unbounded support (from both sides, as if one of them is bounded we can handle it in much the same way as the proof below). 

Let $\widetilde B_0 = B_0 + \eta$, for some small fixed $\eta > 0$. Fix any partition $a \in \mathcal{A}_{H}(l,K)$. Let $S_0 = \{h: a_h  \in [-\widetilde B_0, \widetilde B_0]\}$, and let $h_{m} = \min S_0, h_M = \max S_0$. Note that the following simple inequality holds for any $2 \leq h \leq h_m - 2$ or $h_M + 1 \leq h \leq H-1$:
\begin{align*}
\var[\bm(Y) | a_h < Y \leq a_{h + 1}] & \leq \inf_{t \in (a_h, a_{h  + 1}]} \e[(\bm(Y) - \bm(t))^2 | a_h < Y \leq a_{h + 1}] \\
& \leq \sup_{y, t \in (a_h, a_{h  + 1}]} (\bm(y) - \bm(t))^2  \leq  (\widetilde \bm(|a_h|) - \widetilde \bm(|a_{h + 1}|))^2.
\end{align*}
This gives us the following inequality:
\begin{align} \label{firstpart}
\sum_{h = 2}^{h_{m} - 2}\var[\bm(Y) | a_h < Y \leq a_{h + 1}] &\leq \sum_{h = 2}^{h_{m} - 2}(\widetilde \bm(|a_h|) - \widetilde \bm(|a_{h + 1}|))^2 \\
& \leq (\widetilde \bm(|a_2|) - \widetilde \bm(|a_{h_m - 1}|))^2 \nonumber,
\end{align}
where the last inequality holds since $\widetilde \bm$ is non-decreasing. Similar inequality holds for the other tail as well.

Using a similar technique we obtain the following bound on the interval: $[-\widetilde B_0, \widetilde B_0]$:
\begin{align*}
\sum_{h = h_m}^{h_M - 1} \var[\bm(Y) | a_h < Y \leq a_{h + 1}] & \leq \sum_{h = h_m}^{h_M - 1} \e [(\bm(Y) - \bm(a_h))^2 | a_h < Y \leq a_{h+1} ]\\
& \leq  \sum_{h = h_m}^{h_M - 1}  \sup_{y \in (a_h , a_{h+1}]}(\bm(y) - \bm(a_h))^2.
\end{align*}
Notice further that:
\begin{align*}
& \var[\bm(Y) | a_{h_{m}-1} < Y \leq a_{h_m}]  \leq \sup_{y \in (a_{h_{m}-1}, a_{h_m}]} (\bm(y) - \bm(-\widetilde B_0))^2\\
& \leq \sup_{y \in (a_{h_{m} - 1}, -\widetilde B_0]} (\bm(y) - \bm(-\widetilde B_0))^2 + \sup_{y \in [-\widetilde B_0, a_{h_{m}}]} (\bm(y) - \bm(-\widetilde B_0))^2.
\end{align*}
And a similar inequality holds for $\var[\bm(Y) | a_{h_{M}} < Y \leq a_{h_{M}+1}]$. Thus:
\begin{align*}
& \sum_{h = h_m - 1}^{h_M } \var[\bm(Y) | a_h < Y \leq a_{h + 1}] \leq \underbrace{ \sum_{h = h_m}^{h_M - 1}  \sup_{y \in (a_h , a_{h+1}]}(\bm(y) - \bm(a_h))^2}_{I_1} \\
& + \underbrace{\sup_{y \in (a_{h_{m}-1}, -\widetilde B_0]} (\bm(y) - \bm(-\widetilde B_0))^2}_{I_2} +  \underbrace{\sup_{y \in [-\widetilde B_0, a_{h_{m}}]} (\bm(y) - \bm(-\widetilde B_0))^2}_{I_3}\\
& + \underbrace{\sup_{y \in [\widetilde B_0, a_{h_{M} + 1}]} (\bm(y) - \bm(\widetilde B_0))^2}_{I_4} +  \underbrace{\sup_{y \in (a_{h_{M}}, \widetilde B_0]} (\bm(y) - \bm( \widetilde B_0))^2}_{I_5}.
\end{align*}
We have:
\begin{align}\label{midpart}
I_1 + I_3 + I_5 & \leq \sup_{b \in \Pi_{2|S_0|+3}(\widetilde B_0)} \sum_{i = 2}^{2|S_0|+3} (\bm(b_i) - \bm(b_{i - 1}))^2 \\
& \leq \sup_{b \in \Pi_{2|S_0| +3}(\widetilde B_0)}\left( \sum_{i = 2}^{2|S_0| + 3} |\bm(b_i) - \bm(b_{i - 1})|\right)^2. 
\end{align}
To see this, consider a partition containing the points $b_1 = -\widetilde B_0, b_3 = a_{h_m}, \ldots, b_{2|S_0| + 1} = a_{h_M}, b_{2|S_0| + 3} = \widetilde B_0$, and $b_{2k} = \operatornamewithlimits{argmax}_{y \in (b_{2k-1}, b_{2k+1}]} (\bm(y) - \bm(b_{2k-1}))^2, k \in [|S_0|]$ and $b_{2 |S_0| + 2} =  \operatornamewithlimits{argmax}_{y \in (b_{2|S_0| + 1}, \widetilde B_0]} (\bm(y) - \bm(\widetilde B_0))^2$ (note that if the $\max$ doesn't exist we can take a limit of partitions converging to it).

Next, we control $I_2$:
\begin{align*}
I_2 = \sup_{y \in (a_{h_{m}-1}, -\widetilde B_0]} (\bm(y) - \bm(-\widetilde B_0))^2 \leq  (\widetilde \bm(|a_{h_{m}-1}|) - \widetilde \bm(\widetilde B_0))^2.
\end{align*}
with the last inequality following from (\ref{mglobal}). Combining this bound with (\ref{firstpart}) we get:
\begin{align} \label{midstep1}
(\widetilde \bm(|a_2|) - \widetilde \bm(|a_{h_m - 1}|))^2 + I_2 \leq (\widetilde \bm(|a_2|) - \widetilde \bm(\widetilde B_0))^2.
\end{align}

Similarly, for $I_4$ and the other bound in (\ref{firstpart}) we have:
\begin{align} \label{midstep2}
(\widetilde \bm(|a_H|) - \widetilde \bm(|a_{h_M + 1}|))^2 + I_4 \leq (\widetilde \bm(|a_H|) - \widetilde \bm(\widetilde B_0))^2.
\end{align}

Finally, we deal with the tail part:
\begin{align} \label{lastpart}
\var[\bm(Y) | Y \leq a_2] & \leq \e[(\bm(Y) - \bm(a_2))^2 | Y \leq a_2] \\
&  \leq \e[(\widetilde \bm(|Y|) - \widetilde \bm(|a_2|))^2 | Y \leq a_2] \nonumber \\
& \leq 4 \e[(\widetilde \bm(|Y|))^2 | Y \leq a_2]\nonumber \leq 4 (\e[|\widetilde \bm(|Y|)|^{2+\xi} | Y \leq a_2])^{2/(2 + \xi)}\nonumber\\
& = 4 \left(\int_{-\infty}^{a_2} |\widetilde \bm(|y|)|^{2+\xi} d \p(Y \leq y) \p(Y \leq a_2)^{-1}\right)^{2/(2 + \xi)} \nonumber\\
& = o(1) \p(Y \leq a_2)^{-2/(2 + \xi)}.\nonumber
\end{align}
where we used the fact that $\e[|\widetilde \bm(|Y|)|^{2+\xi}]$ is bounded by assumption, and the $o(1)$ is in the sense of $|a_2| \rightarrow \infty$ . We can show a similar inequality for the other tail --- $\var[\bm(Y) | Y \geq a_H]$. 

Combining (\ref{firstpart}), (\ref{midstep1}),  (\ref{midstep2}), (\ref{midpart}) and (\ref{lastpart}) we have:
\begin{align*}
& \sum_{h = 1}^H \var[\bm(Y) | a_h < Y \leq a_{h+1}] 
\leq \sup_{b \in \Pi_{2|S_0| + 3}(\widetilde B_0)}\left( \sum_{i = 2}^{2|S_0| + 3} |\bm(b_i) - \bm(b_{i - 1})|\right)^2  \\
& + o(1) \p(Y \geq a_H)^{-2/(2 + \xi)} +  o(1) \p(Y \leq a_2)^{-2/(2 + \xi)} \\
& + (\widetilde \bm(|a_2|) - \widetilde \bm(\widetilde B_0))^2 +  (\widetilde \bm(|a_H|) - \widetilde \bm(\widetilde B_0))^2.
\end{align*}
Since $ (\widetilde \bm(|a_2|) - \widetilde \bm(\widetilde B_0))^2 \leq 4 (\widetilde \bm(|a_2|))^2$, and we know that $\widetilde \bm(|a_2|)^{2 + \xi} \frac{l}{H} \leq \widetilde \bm(|a_2|)^{2 + \xi} \p(Y \leq a_2) \rightarrow 0$, this means that $\widetilde \bm(|a_2|)^{2 } \frac{1}{H^{2/(2 + \xi)}} \rightarrow 0$. Furthremore $ o(1) \p(Y \leq a_2)^{-2/(2 + \xi)}\frac{1}{H^{2/(2 + \xi)}} = o(1)$. Finally we recall that by (\ref{msmoothness}) we have:
\begin{align*}
\sup_{b \in \Pi_{2|S_0| + 3}(\widetilde B_0)}\left( \sum_{i = 2}^{2|S_0| + 3} |\bm(b_i) - \bm(b_{i - 1})|\right)^2  = o\left(|S_0|^{2/(2 + \xi)}\right).
\end{align*}
However $|S_0| \leq \p(-\widetilde B_0 \leq Y \leq \widetilde B_0)H/l + 1$ and thus:
$$\sup_{b \in \Pi_{2|S_0| + 3}(\widetilde B_0)}\left( \sum_{i = 2}^{2|S_0| + 3} |\bm(b_i) - \bm(b_{i - 1})|\right)^2 = o\left(H^{2/(2 + \xi)}\right),
$$
which finishes the proof.
%
%
%
%
%
\end{proof}

\begin{proof}[Proof of Lemma \ref{slicedmeanconc}] Before we go to the main proof of the lemma we first formulate a simple but useful concentration inequality. 
\begin{lemma}\label{bernst} Let $\widetilde X$ be a sub-Gaussian random variable with $\|\widetilde X\|_{\psi_2} \leq \mathcal{K}$. Let $A(\widetilde X,\nu) \in \{0,1\}$ be any (randomized) acceptance rule such that $\p(A = 1) \geq q$, with $\nu$ being any random variable. Let $X_1, \ldots, X_r$ be an i.i.d. samples of the distribution $\widetilde X | A(\widetilde X, \nu) = 1$. Denote with $\mu = \e[X_i]$. Then there exist some absolute constants $C_1, C_2 > 0$ such that:
$$
\p(|\overline X - \mu| > \epsilon) \leq 2 \exp\left(-\frac{\epsilon^2 r}{C_1\mathcal{K}^2\exp(\sqrt{1 - \log(q)}) + C_2\mathcal{K} \epsilon}\right).
$$
\end{lemma}
As a Corollary to Lemma \ref{bernst}, observe that since $\sup_{q \in [0,1]} \exp(\sqrt{1 - \log(q)}) q = e$, we then have the following:
$$
\p(|\overline X - \mu| > \epsilon) \leq 2 \exp\left(-\frac{\epsilon^2 r}{C_1\mathcal{K}^2q^{-1} + C_2\mathcal{K} \epsilon}\right),
$$
for some absolute constants $C_1, C_2 > 0$.

By (\ref{DKWbound}) we know that $\p(Y \in S_h) \geq \frac{1}{H} - 2\epsilon$ on $E$, thus setting $q = \frac{1}{H} - 2\epsilon$, by Lemma \ref{bernst} conditionally on  $\{Y_{(m h)}  : h \in [H-1]\}$ we have for all $j \in G$ and all $h$:
\begin{align*}
& \p\left(\left|\overline \bX_{h,1:(m-1)}^j  - \bmu_h^j \right| > \eta \right) \leq 2 \exp\left(- \frac{\eta^2 (m-1)}{C_1\mathcal{K}^2q^{-1} + C_2\mathcal{K} \eta}\right).
\end{align*}

Note that  Lemma \ref{bernst}  is applicable in this case, since the statistics $\bX^j_{h,\pi_i}$ are i.i.d. conditionally on $Y_{(m(h-1))}$ and $Y_{(mh)}$, where $\pi$ is a random permutation, as we noticed in the main text. Therefore $\bX^j_{h,\pi_i} \stackrel{d}{=} \bX^j | A(\bX^j, \nu)$ for the acceptance rule $A(\bX^j, \nu) := A(\bX, \varepsilon) = \mathbbm{1}(f(\bX\T \bbeta,\varepsilon) \in S_h)$. Furthermore, notice that the above inequality holds regardless of the values of $\{Y_{(m h)}  : h \in [H-1]\}$, on the event $E$. 

Finally, using union bound across the slices and the indexes $j \in G$, we have that this holds for all slices or in other words 
\begin{align*}
\p\left(\max_{j \in G, h \in [H]} \left|\overline \bX_{h,1:(m-1)}^j  - \bmu_h^j \right| > \eta \right) \leq 2 |G| H \exp\left(- \frac{\eta^2 (m-1)}{C_1\mathcal{K}^2q^{-1} + C_2\mathcal{K} \eta}\right),
\end{align*}
 on the event $E$. This is precisely what we wanted to show.
\end{proof}

\begin{proof}[Proof of Lemma \ref{bernst}] Observe that the random variable $X \stackrel{d}{=} \widetilde X | A(\widetilde X, \nu)$ satisfies the following inequality:
\begin{align*}
\p(|X| \geq t) & = \p( |\widetilde X| \geq t | A(\widetilde X, \nu) = 1) = \frac{\p(|\widetilde X| \geq t, A(\widetilde X, \nu) = 1)}{\p( A(\widetilde X, \nu) = 1)} \\
&\leq q^{-1} \p(|\widetilde X| \geq t) \leq q^{-1} e \exp(-ct^2/\mathcal{K}^2),
\end{align*}
where $c$ is an absolute constant, and the last inequality follows by the fact that $\widetilde X$ is assumed to be sub-Gaussian. Clearly the above bound can be substituted with the trivial bound $1$, for values of $t \leq \sqrt{-\mathcal{K}^2/c \log \left(\frac{q}{e}\right)}$. Let $f(q) := \sqrt{-\mathcal{K}^2/c \log \left(\frac{q}{e}\right)}$. Next, for any positive integer $j \in \mathbb{N}$, we have:
\begin{align*}
\e[|X|^j] & = \int_{0}^{\infty} \p(|X| \geq t) j t^{j - 1} d t \leq \int_{0}^{ f(q)} j t^{j - 1} d t + \int_{f(q)}^{\infty} q^{-1} e \exp(-ct^2/\mathcal{K}^2) j t^{j - 1} d t.
\end{align*}
Multiplying by $\lambda^j/j!$ and summing over $j = 1,2,\ldots$, we obtain:
\begin{align*}
\e[\exp(\lambda |X|)] & \leq 1 + \lambda \int_{0}^{ f(q)} \exp(\lambda t) d t + \lambda \int_{f(q)}^{\infty} q^{-1} e \exp(-ct^2/\mathcal{K}^2) \exp(\lambda t) d t\\
& = \exp(\lambda f(q)) + \frac{\lambda e \sqrt{\pi}\mathcal{K}}{q\sqrt{c}} \exp(\mathcal{K}^2\lambda^2/4c)\left[1 - \Phi\left(\frac{f(q) - \mathcal{K}^2\lambda/(2c)}{\mathcal{K}/\sqrt{2c}}\right)\right].
\end{align*}
Assuming that $\lambda$ is small enough so that $f(q) - \mathcal{K}^2 \lambda/(2c) > 0$, we can use the well known tail bound $1 - \Phi(x) \leq \phi(x)/x$, for $x > 0$ to get:
\begin{align*}
\e[\exp(\lambda |X|)] & \leq  \exp(\lambda f(q)) + \frac{\lambda e \mathcal{K}^2 \sqrt{\pi}\exp(\mathcal{K}^2\lambda^2/4c)}{q(f(q) - \mathcal{K}^2\lambda/(2c)) \sqrt{2} c} \phi\left(\frac{f(q) - \mathcal{K}^2\lambda/(2c)}{\mathcal{K}/\sqrt{2c}}\right)\\
& = \exp(\lambda f(q))\left(1 + \frac{\lambda \mathcal{K}^2}{2 c(f(q) - \mathcal{K}^2 \lambda / (2c))}\right).
\end{align*}
Next, by the triangle inequality and the convexity of the exponent and the absolute value we have:
$$
\e[\exp(\lambda|X - \mu|)] \leq \e[\exp(\lambda|X|)] \exp[ \e[\lambda |X|]] \leq  \e[\exp(2\lambda|X|)].
$$
Set $Z := X - \mu$. We have showed the following:
\begin{align}
\e[\exp[\lambda|Z|] - 1 - \lambda |Z|] & \leq \e[\exp[\lambda|Z|]]  \leq \exp(2\lambda f(q))\left(1 + \frac{\lambda \mathcal{K}^2}{c (f(q) - \mathcal{K}^2 \lambda / c)}\right),\label{bernstein}
\end{align}
for a $\lambda$ such that $f(q) - \mathcal{K}^2 \lambda/c > 0$. By selecting $\lambda := \frac{1}{2}\sqrt{\frac{c}{\mathcal{K}^2}}$, one can easily verify that $(f(q) - \mathcal{K}^2 \lambda / c) \geq \lambda^{-1}/4$, and hence $\frac{\lambda \mathcal{K}^2}{c(f(q) - \mathcal{K}^2 \lambda / c) } \leq 1$, for any $q \leq 1$. With this choice of $\lambda$ (\ref{bernstein}) becomes:
$$
\e[\exp[\lambda|Z|] - 1 - \lambda |Z|]\lambda^{-2} \leq  \frac{8\mathcal{K}^2}{c} \exp\left(\sqrt{\frac{c}{\mathcal{K}^2}}f(q)\right) = \frac{8\mathcal{K}^2}{c}  \exp\left(\sqrt{- \log \frac{q}{e}}\right).
$$

Recall that a version Bernstein's inequality (see Lemma 2.2.11 \cite{van1996weak}) states that if the following moment condition $\e|Z|^j \leq j! \lambda^{-(j-2)} v/2$ is met for all $j \geq 2$ and $Z_i \sim Z, i \in [r]$ are mean $0$, i.i.d. then:
$$
\p\left(\left|\sum_{i = 1}^r Z_i\right| \geq \epsilon\right) \leq 2 \exp\left(-\frac{1}{2}\frac{\epsilon^2}{r v + \lambda^{-1} \epsilon}\right).
$$ 
Observe that by a Taylor expansion the condition $\e[\exp[\lambda|Z|] - 1 - \lambda |Z|]\lambda^{-2} \leq \frac{1}{2}v$ implies the moment condition from above. Hence we conclude:
$$
\p\left(\left|\frac{1}{r}\sum_{i = 1}^r X_i - \mu\right| \geq \epsilon\right) \leq 2 \exp\left(- \frac{1}{2}\frac{\epsilon^2 r}{16\mathcal{K}^2\exp(\sqrt{1 - \log(q)})/c + 2 \mathcal{K}/\sqrt{c} \epsilon}\right),
$$
which completes the proof.
\end{proof}

%
%
%

\begin{proof}[Proof of Lemma \ref{simplelemma}]
Using the sliced stability condition, for large $H$ we get:
\begin{align*}
\left| \var[\bm_j(Y)] - \sum_{h = 1}^H (\bmu_h^j)^2  \p(Y \in S_h) \right| & = \sum_{h = 1}^H \var[\bm_j(Y) | Y \in S_h] \p(Y \in S_h ) \\
& \leq \frac{CH^\kappa}{s} \left(\frac{1}{H} + 2 \epsilon \right).
\end{align*}
This shows (\ref{forms12}). Consequently we have:
\begin{align*}
\left(\frac{1}{H} - 2\epsilon\right) \sum_{h = 1}^H (\bmu_h^j)^2  \leq  \sum_{h = 1}^H (\bmu_h^j)^2 \p(Y \in S_h) \leq \frac{C_V}{s} + \frac{CH^\kappa}{s} \left(\frac{1}{H} + 2 \epsilon \right).
\end{align*}
This yields (\ref{expectationsquared}). To get (\ref{expectationabs}) we proceed as follows:\\
\begin{align*}
\left(\frac{1}{H} - 2\epsilon\right)  \sum_{h = 1}^H |\bmu_h^j|  &\leq \sum_{h = 1}^H |\bmu_h^j| \p(Y \in S_h ) \leq \sqrt{\sum_{h = 1}^H \p(Y \in S_h)}\sqrt{\sum_{h = 1}^H (\bmu_h^j)^2 \p(Y \in S_h)} \\
& \leq \sqrt{\frac{C_V}{s} + \frac{CH^\kappa}{s} \left(\frac{1}{H} + 2 \epsilon \right)},
\end{align*}
and we are done. 
\end{proof}

\begin{proof}[Proof of Lemma \ref{calculation}] Note that on the event $\widetilde{E}$ we have the following chain of inequalities:
\begin{align}
& \frac{1}{H}\sum_{h = 1}^H \left|\left( \frac{1}{m}\bX^j_{h,m} + \frac{m-1}{m} \overline \bX^j_{h,1:(m-1)}\right)^2 - \frac{(m-1)^2}{m^2}(\bmu_h^j)^2\right| \label{cont} \\
& \leq \frac{1}{H m^2} \sum_{h = 1}^H (\bX_{h,m}^j)^2 +  \frac{2(m-1)}{H m^2} \sum_{h = 1}^H |\bX_{h,m}^j| (\eta + |\bmu_h^j |) +  \frac{1}{H}\frac{(m-1)^2}{m^2}\sum_{h = 1}^H  \eta (2|\bmu_h^j| + \eta) \nonumber\\
& \leq \frac{1}{m} \frac{1}{n} \sum_{r = 1}^n (\bX_{r}^j)^2 +  \frac{2}{H m} \sqrt{\sum_{r = 1}^n (\bX_{r}^j)^2}\sqrt{ 2\sum_{h = 1}^H  (\eta^2 + (\bmu_h^j)^2)} \nonumber +  \eta^2 +  2\frac{\eta}{H}B_3. \nonumber
\end{align}
where we used that we are on the event $\widetilde{E}$ in the first inequality, and (\ref{expectationabs}), Cauchy-Schwartz and the trivial bounds $\frac{m-1}{m} < 1$, $\frac{1}{n} \sum_{h = 1}^H (\bX_{h,m}^j)^2 \leq \frac{1}{n} \sum_{r = 1}^n (\bX_{r}^j)^2$ in the second one. 

To this end, observe that $(\bX^j_r)^2, r \in [n]$ are i.i.d. random variables with sub-exponential distributions. This can be seen from the standard inequality:
$$
\|(\bX^j_r)^2\|_{\psi_1} \leq 2\|\bX^j_r\|_{\psi_2}^2 \leq 2\mathcal{K}^2.
$$
Clearly we also have $\e[(\bX^j_r)^2] \leq 2 \mathcal{K}^2$. Denote the mean $\e[(\bX^j_r)^2] = \nu$. Now we are in a position to use a Bernstein type of deviation inequality (see Proposition 5.16 in \cite{vershynin2010introduction}). We obtain:
$$
\p\left(\frac{1}{n} \sum_{r = 1}^n (\bX_{r}^j)^2 > \nu + \tau\right) \leq \exp\left(-c \min\left(\frac{\tau^2n}{4\mathcal{K}^4}, \frac{\tau n}{2 \mathcal{K}^2}\right)\right),
$$
for some absolute constant $c > 0$. Hence we infer that, when $\tau \leq 2 \mathcal{K}^2$, there exists a set $\widetilde{\widetilde{E\, }} \subset \widetilde{E}$ failing with probability at most $p\exp(-c \frac{\tau^2n}{4\mathcal{K}^4})$, such that $\frac{1}{n} \sum_{r = 1}^n (\bX_{r}^j)^2  \leq \nu + \tau \leq 2 \mathcal{K}^2 + \tau$ for all $j \in [p]$. Therefore continuing the bound on the event $\widetilde{\widetilde{E\, }}$, we get:
\begin{align*}
 (\ref{cont}) & \leq \frac{(2 \mathcal{K}^2 + \tau)}{m} +  \frac{2\sqrt{2 \mathcal{K}^2 + \tau}}{\sqrt{m}}\sqrt{ 2\eta^2 + 2\frac{B_2}{H}} +   \eta^2 +   2\eta\frac{B_3}{H},
\end{align*}
where we used (\ref{expectationsquared}). This finishes the proof.


\end{proof}

\begin{proof}[Proof of Lemma \ref{covbound}] 
Fix any $\tau \in (0,2]$. Define the event:
$$
\widetilde{\widetilde{E}\,}_{\cS_{\bbeta}} = \widetilde{E}_{\cS_{\bbeta}} \cap \{\sup_{j \in \cS_{\bbeta}} n^{-1}\sum_{i = 1}^n (\bX_r^j)^2 \leq 2 + \tau\}
$$

We first note that the following inequality holds:
\begin{align*}
& \left|\Vb^{jk} -\operatorname{sign}(\bbeta_j) \operatorname{sign}(\bbeta_k)\frac{C_V}{s} \right|  = \left| \frac{1}{H}\sum_{h = 1}^H \overline \bX^j_h \overline \bX^k_h - \operatorname{sign}(\bbeta_j) \operatorname{sign}(\bbeta_k)\frac{C_V}{s} \right| \\
& \leq \left| \frac{1}{H}\sum_{h = 1}^H (\overline \bX^j_h)^2 - \frac{C_V}{s}\right| + \frac{1}{H}\sum_{h = 1}^H  \left|\overline \bX^j_h \right| \left| \operatorname{sign}(\bbeta_j)\overline \bX^j_h - \operatorname{sign}(\bbeta_k)\overline \bX^k_h\right|, 
\end{align*}
where the LHS equals, the LHS of (\ref{covboundineq}) after using (\ref{suppcov}). Fortunately (\ref{bigbound}) and (\ref{boundlemmaxmean}) already give bounds on the first term on the event $\widetilde{\widetilde{E}\,}_{\cS_{\bbeta}}$, which can be seen with identical arguments to Lemmas \ref{simplelemma} and \ref{calculation}. We now show that the second term is small on the same event. Note that the following identity holds:
\begin{align*}
& \frac{1}{H}\sum_{h = 1}^H  \left|\overline \bX^j_h \right| \left| \operatorname{sign}(\bbeta_j)\overline \bX^j_h - \operatorname{sign}(\bbeta_k)\overline \bX^k_h \right| \\
& = \frac{1}{H}\sum_{h = 1}^H  \left| \frac{m-1}{m}\overline \bX^j_{h,1:(m-1)} + \frac{1}{m}\bX^j_{h,m} \right| \left| \operatorname{sign}(\bbeta_j)\overline \bX^j_h - \frac{m-1}{m}\sign(\bbeta_k)\bmu_h^j \right.\\
&~~~~~~~~~~~~~~~~~~~~~~~~~~~~~~~~~~~~~~~~~~~~~~~~~~~~ \left.+\frac{m-1}{m}\sign(\bbeta_k)\bmu_h^k - \operatorname{sign}(\bbeta_k)\overline \bX^k_h\right|.
\end{align*}
Thus on the event ${\widetilde{E}}_{\cS_{\bbeta}}$:
\begin{align*}
& \frac{1}{H}\sum_{h = 1}^H  \left|\overline \bX^j_h \right| \left| \operatorname{sign}(\bbeta_j)\overline \bX^j_h - \operatorname{sign}(\bbeta_k)\overline \bX^k_h\right| \\
& \leq \frac{1}{H}\sum_{h = 1}^H   \left(\bmu_h + \eta + \frac{1}{m}\left|\bX^j_{h,m} \right|\right) \left( 2\eta +  \frac{1}{m}\left| \bX^j_{h,m}\right| +  \frac{1}{m}\left| \bX^k_{h,m}\right|\right),
\end{align*}
where $\bmu_h = |\bmu_h^j| = |\bmu_h^k|$, and we used that $\frac{m-1}{m} < 1$, and the fact that on ${\widetilde{E}}_{\cS_{\bbeta}}$ we have $|\bX^j_{h,1:(m-1)} - \bmu_h^j| \leq \eta$ and similarly $|\bX^k_{h,1:(m-1)} - \bmu_h^k| \leq \eta$. Next we have:
\begin{align*}
&  \frac{1}{H}\sum_{h = 1}^H   \left(\bmu_h + \eta + \frac{1}{m}\left|\bX^j_{h,m} \right|\right) \left( 2\eta +  \frac{1}{m}\left| \bX^j_{h,m}\right| +  \frac{1}{m}\left| \bX^k_{h,m}\right|\right)\\
 & \leq 2 \frac{\eta}{H} \sum_{h = 1}^H \bmu_h + \frac{1}{Hm} \sum_{h = 1}^H (\bmu_h + \eta)(|\bX^j_{h,m}| + |\bX^k_{h,m}|) + 2\eta^2\\
& + \frac{2\eta}{mH} \sum_{h = 1}^H |\bX_{h,m}^j| + \frac{1}{m^2H} \sum_{h = 1}^H (\bX^j_{h,m})^2 + \frac{1}{2 m^2H} \sum_{h = 1}^H (\bX^j_{h,m})^2 + \frac{1}{2 m^2H} \sum_{h = 1}^H (\bX^k_{h,m})^2,
\end{align*}
where we used the simple inequality $ab \leq (a^2 + b^2)/2$. Luckily we have already controlled all of the above quantities. Using Lemma \ref{calculation} and (\ref{bigbound}) we get:
\begin{align*}
& \frac{1}{H}\sum_{h = 1}^H  \left|\overline \bX^j_h \right| \left| \operatorname{sign}(\bbeta_j)\overline \bX^j_h - \operatorname{sign}(\bbeta_k)\overline \bX^k_h \right| \\
& \leq 2\frac{\eta}{H}B_3+  \frac{2\sqrt{2 + \tau}}{\sqrt{m}}\sqrt{ 2\eta^2 + 2\frac{B_2}{H}} + 2\eta^2  + 2\frac{\eta\sqrt{2 + \tau}}{\sqrt{m}} + 2 \frac{2 + \tau}{m},
\end{align*}
where we heavily relied on the fact that on $\widetilde{\widetilde{E\,}}_{\cS_{\bbeta}}$  we have $ \frac{1}{mH} \sum_{r = 1}^n (\bX^j_{r})^2 \leq 2 + \tau$, the rest of the bounds can be seen in the proof of Lemma \ref{calculation}. (For the term note $\frac{1}{mH} \sum_{h = 1}^H |\bX^j_{h,m}|\leq \frac{1}{m} \sqrt{\sum_{h = 1}^H (\bX^j_{h,m})^2/H} \leq \frac{1}{\sqrt{m}} \sqrt{\sum_{r = 1}^n (\bX^j_{r})^2/mH}$).
Finally noting that $2\frac{\eta\sqrt{2 + \tau}}{\sqrt{m}} \leq \eta^2 + \frac{2 + \tau}{m}$ gives the desired result.
\end{proof}


\begin{lemma}\label{wainwriglemma6} Let $\Nb$ be a $s \times s$ symmetric ``noise'' matrix satisfying:
\begin{align} \|\Nb\|_{\max} \leq \frac{C_V}{40 s} \label{Nmatrixver}.
\end{align}
Then the following occurs:
\begin{itemize}
\item[(a)] Let $\gamma_1$ be the maximum eigenvalue of $\widetilde \Vb_{\cS_{\bbeta}, \cS_{\bbeta}} = \frac{C_V}{2} \bbeta_{\cS_{\bbeta}} \bbeta_{\cS_{\bbeta}}\T + \Nb$, i.e. $\gamma_1 := \lambda_{\max}(\widetilde \Vb_{\cS_{\bbeta}, \cS_{\bbeta}})$. Then we have $
|\gamma_1 -  \frac{C_V}{2}| \leq \frac{C_V}{40}$; Furthermore if $\gamma_2$ is the second largest in magnitude eigenvalue of $\widetilde \Vb_{\cS_{\bbeta}, \cS_{\bbeta}}$, we have $|\gamma_2 | \leq \frac{C_V}{40}$, and hence $\gamma_1 > |\gamma_2|$.
\item[(b)] The corresponding principal eigenvector of $\widetilde \Vb_{\cS_{\bbeta}, \cS_{\bbeta}}$ --- $\widetilde \bz_{\cS_{\bbeta}}$\footnote{Here we mean the principal eigenvector oriented so that $\widetilde \bz_{\cS_{\bbeta}}\T \bbeta_{\cS_{\bbeta}} \geq 0$.} satisfies the following inequality:
$$\left \|\widetilde  \bz_{\cS_{\bbeta}} - \bbeta_{\cS_{\bbeta}}\right \|_{\infty} \leq \frac{1}{2\sqrt{s}}.$$
\end{itemize}
\end{lemma}

\begin{proof} This Lemma is a re-statement of Lemma 6 from \cite{amini2008high}, the difference being that we require precise bounds, rather than considering simply the asymptotics. To see part (a), observe that (\ref{Nmatrixver}) implies:
\begin{align} \label{22bound}
\|\Nb\|_{2, 2} = \sup_{\|\bv\|_2 = 1} |\bv\T \Nb \bv| \leq \|\bv\|_1^2 \|\Nb\|_{\max} \leq \frac{C_V}{40}.
\end{align}
Hence by Weyl's inequality we have that:
$$
|\gamma_1 - \frac{C_V}{2}| \leq \frac{C_V}{40}, \quad |\gamma_2 - 0| \leq \frac{C_V}{40},
$$
which is exactly part (a). For the second part observe that
\begin{align} \label{inftyinftybound}
\|\Nb\|_{\infty,\infty} = \max_{i} \sum_{j} |\Nb_{ij}| \leq  s \frac{C_V}{40 s} = \frac{C_V}{40}.
\end{align}
The proof of part (b) can then be seen as in Lemma 6 of \cite{amini2008high}, by carefully exploiting (\ref{22bound}) and (\ref{inftyinftybound}).

\end{proof}

\begin{proof}[Proof of Proposition \ref{lowerboundprop}] 
The proof is based on an application of Fano's inequality, which in turn is a standard approach for showing minimax lower bounds (e.g. see \cite{cover2012elements,wainwright2009information, yang1999information, yu1997assouad} among others). In particular we base our proof on ideas from \cite{amini2008high, wainwright2009information}.

For two probability measures $P$ and $Q$, which are absolutely continuous with respect to a third probability measure $\mu$ define their KL divergence by $D(P \| Q) = \int p \log \frac{p}{q} d\mu$, where $p = \frac{dP}{d\mu}$, $q = \frac{dQ}{d\mu}$. We proceed with the following lemma:

\begin{lemma} \label{lowerboundthm} Let us have $n$ observations from a SIM $Y = f(\bX\T \bbeta, \varepsilon), \bX \sim N_p(0,\mathbb{I}_p)$. Assume that for any fixed $u, v \in \mathbb{R}$ the following regularity condition holds for the random variables $f(u,\varepsilon)$ and $f(v,\varepsilon)$:
\begin{align}
D(p(f(u, \varepsilon)) \| p(f(v, \varepsilon))) \leq \exp(\Xi(u - v)^2) - 1, \label{KLdivcond} 
\end{align}
where $\Xi$ is a positive constant. Then if
$$
n < \frac{s\log(p-s + 1)}{8 \Xi},
$$
and $s \geq 8 \Xi$, any algorithm recovering the support in our model will have errors with probability at least $\frac{1}{2}$ asymptotically.
\end{lemma}

To finish the proof simply apply Example \ref{lastexample} with $P(x) = x/(2\sigma^2)$ and $G = h = Id$.
\end{proof} 

\begin{proof}[Proof of Lemma \ref{lowerboundthm}] For simplicity of the exposition we will assume that the vector $\bbeta$ has only non-negative entries (i.e. all non-zero entries are $\frac{1}{\sqrt{s}}$). The proof extends in exactly the same way in the case when the entries of $\bbeta$ are not restricted to be positive. 

Let $\mathbb{S} \subset 2^{[p]}$, the set of all subsets of $[p]$ with $s$ elements. Clearly, $|\mathbb{S}| = {p \choose s}$.  Let $\widehat{S} : (\mathbb{R}^{p + 1})^n \rightarrow \mathbb{S}$ be any potentially random function, which is used to recover the support of $\bbeta$, based on the sample $\{(Y_i, \bX_i)\}_{i = 1}^n$. Under the 0-1 loss the risk equals the probability of error:
\begin{align}\frac{1}{{p \choose s}} \sum_{\cS_{\bbeta} \in \mathbb{S}}P_{\cS_{\bbeta}} (\widehat{S}\neq \cS_{\bbeta}), \label{errorsum}
\end{align}
where by $P_{\cS_{\bbeta}}$ we are measuring the probability under a dataset generated with $\supp(\bbeta)$ equal to the index of the measure $P_{\cS_{\bbeta}}$.

Instead of directly dealing with the sum above, we first consider the $p-s+1$ element set $\widetilde{\mathbb{S}} = \{S \in \mathbb{S} : [s-1] \subset S, |S| = s \}$, and we bound the probability of error, on any function $\widehat{S}$ (even if given the knowledge that the true support is drawn from $\widetilde{\mathbb{S}}$). Let $J$ be a uniformly distributed in $\widetilde{\mathbb{S}}$. By Fano's inequality that:
\begin{align}\label{fanosineq}
\p(\mbox{error}) \geq 1 - \frac{\calI(J; (Y,\bX)^n) + \log(2)}{\log|\widetilde{\mathbb{S}}|},
\end{align}
where $\calI(J; (Y,\bX)^n)$ is the mutual information between the sample $J$ and the sample $(Y,\bX)^n$. Note now that for the mutual information we have 
\begin{align*}\calI(J; (Y,\bX)^n) & = \calI(J; (f(\bX\T\bbeta, \varepsilon) ,\bX)^n) \\
& \leq n \calH((f(\bX\T\bbeta, \varepsilon),\bX)) - n \calH((f(\bX\T\bbeta, \varepsilon),\bX) | J),\footnotemark
\end{align*}
\footnotetext{Here we use $\calH$ to denote the entropy, not to be confused with the number of slices $H$.}where the last inequality follows from the chain inequality of entropy. We therefore need an upper bound on the last expression.

Let $i \geq s$ and set $\bbeta^{i} := (\beta_1^i, \beta_2^i, \ldots, \beta_p^i)\T$ the vector such that that $\beta_j^i = \frac{\mathbbm{1}(j \in [s-1]) + \mathbbm{1}(j = i)}{\sqrt{s}}$. When $J$ is unknown the distribution is a mixture, and hence due to the convexity of $-\log$ we have the following inequality:
\begin{align*}
&\calH((f(\bX\T\bbeta, \varepsilon), \bX)) - \calH((f(\bX\T\bbeta, \varepsilon), \bX)) | J) \\
&\leq \frac{1}{(p - s + 1)^2}\sum_{i, j \geq s} p((f(\bX\T\bbeta^{i}, \varepsilon), \bX)) \log \frac{p(f(\bX\T\bbeta^{i}, \varepsilon), \bX)}{p(f(\bX\T\bbeta^{j}, \varepsilon), \bX)}\\
& = \frac{p - s}{p - s + 1}D(p((f(\bX\T\bbeta^{s}, \varepsilon), \bX)) \| p((f(\bX\T\bbeta^{s + 1}, \varepsilon), \bX))),
\end{align*}
where the last equality follows by a symmetric argument. Since the KL divergence is invariant under changing variables, setting $U = \bX\T \bbeta^s$, $V = \bX\T \bbeta^{s + 1}$ and $\bW = \bP_{\{\bbeta^s, \bbeta^{s + 1}\}^\perp}\bX$, where  $ \bP_{\{\bbeta^s, \bbeta^{s + 1}\}^\perp}$ denotes the orthogonal projection on the space $\operatorname{span}\{\bbeta^s, \bbeta^{s + 1}\}^\perp$. We get:
\begin{align*}
& \calH((f(\bX\T\bbeta, \varepsilon), \bX)) - \calH((f(\bX\T\bbeta, \varepsilon), \bX)) | J)  \\
& \leq \frac{p - s}{p - s + 1}D(p((f(U, \varepsilon), U, V)) \| p((f(V, \varepsilon), U, V))),
\end{align*}
where we used the fact that $\bW$ is independent of $U, V, \varepsilon$. Applying assumption (\ref{KLdivcond}) we get:
\begin{align*}
D(p((f(U, \varepsilon), U, V)) \| p((f(V, \varepsilon), U, V))) & \leq \e \exp(\Xi(U - V)^2) - 1 \\
& = \sqrt{\frac{s}{s - 4\Xi}} - 1 \leq \frac{4\Xi}{s}.
\end{align*}
where the first inequality can be obtained by conditioning on $U, V$, in the equality we used the fact that $U - V \sim N(0, \frac{2}{s})$, and we assume that the value of $s$ is large enough so that $s \geq 8\Xi$. The inequality in the preceding display helps us to conclude that for large values of $s$:
$$
\calI(J; (Y,\bX)^n) \leq \frac{4\Xi(p - s)n}{s(p - s + 1)} < \frac{4\Xi n}{s}.
$$
Consequently by (\ref{fanosineq}) if $n < \frac{s\log(p-s + 1)}{8\Xi}$ we will have errors with probability at least $\frac{1}{2}$, asymptotically. To finish the conclusion, note that the sum (\ref{errorsum}), can be split into ${p \choose s - 1}$ terms, by the following operation: 

\begin{enumerate}
\item Repeat each set in $\mathbb{S}$ --- $s$ times, and denote this superset by $s \times \mathbb{S}$
\item For each $S$ of the ${p \choose s - 1}$, subsets of $[p]$ with $s-1$ elements, collect $p-s + 1$ distinct elements of $s \times \mathbb{S}$ containing $S$
\item Apply the $\frac{1}{2}$ error bound obtained from above to this local sum.
\end{enumerate}

In the end we get that the probability of error by selecting $S \subset \mathbb{S}$ uniformly is at least: $\frac{1}{s}\frac{{p \choose s-1}}{{p \choose s}} (p - s + 1) \frac{1}{2} = \frac{1}{2}$.
\end{proof}

\section{SIM Examples}\label{app:examples}

In this section we look into models of the type $Y = G(h(\bX\T \bbeta) + \epsilon)$ and show that under certain sufficient conditions they belong to a class $\mathcal{F}_{A}$ for some $A$. In addition, we provide a sub-class of these models, in which support recovery is impossible with probability at least $\frac{1}{2}$ when $\Gamma$ is small.

\begin{example} \label{exmp2} For SIM $Y = G(h(\bX\T \bbeta) + \varepsilon)$ with strictly monotone $h$ and $G$ there exists $A > 0$ (depending on $h, G$) such that $\var(\e[\bX\T \bbeta | Y]) \geq A > 0$.
\end{example}
\begin{proof} We will show more generally that if $Y = f(\bX\T \bbeta, \varepsilon)$, where $f, \varepsilon$ satisfy the condition that the function $g(z) = \e[f(Z, \varepsilon) | Z = z]$ is strictly monotone, we have: $\var(\e[Z | f(Z, \varepsilon)]) > 0$. 

Note that the condition $\exists A > 0: \var(\e[Z | f(Z,\varepsilon)]) \geq A$ is equivalent $\e[Z | f(Z, \varepsilon)]$ to not being a constant. We argue that the latter is clearly implied if, for example, $\e[Z f(Z,\varepsilon)] \neq 0$. To see this assume the contrary, i.e. $\e[Z | f(Z,\varepsilon)] = 0$ a.s., but  $\e[Z f(Z,\varepsilon)] \neq 0$. Then we have $\e[Z f(Z,\varepsilon)] = \e[\e[Z | f(Z,\varepsilon)]f(Z,\varepsilon)] = 0$, which is a contradiction. 

Next we show that our condition implies $\e[Z f(Z,\varepsilon)] \neq 0$. WLOG assume that $g$ is strictly increasing. Observe that since $Z \sim N(0,1)$ is a continuous random variable, by Chebyshev's association inequality \citep{boucheron2013concentration} we have:
$$
\e[Z f(Z, \varepsilon)] = \e[Z g(Z)] > \e[Z] \e[g(Z)] = 0,
$$
and hence as we argued earlier it follows that $\var(\e[Z | f(Z, \varepsilon)]) > 0$. Due to the independence of $\bX$ and $\varepsilon$ models with a coordinate-wise monotone $f$ function (strictly monotone in the first coordinate) belong to this class. 
\end{proof}

\begin{example} \label{exmp3} Let $Y = G(h(\bX\T \bbeta) + \varepsilon)$, where $G, h$ are strictly monotone continuous functions. Furthermore let $\varepsilon$ be such that the function $\bm(y)$ is continuous and monotone in $y$. We argue that such models are sliced stable in the sense of Definition \ref{def:sliced_stable}. 
\end{example}

\begin{proof}
Since $\bm$ is a continuous and monotone function in $y$, it follows that it is of bounded variation on any closed interval, which in turn implies (\ref{msmoothness}). Furthermore, by Example \ref{exmp2}, we are guaranteed that $\sqrt{\var(\bm_j(Y))s} = C_V$ for some $C_V > 0$. Notice also that in the case of a continuous and monotone $\bm$ one can take $\widetilde \bm(y) := |\bm(y) - \bm(-y)| < |\bm(y)| + |\bm(-y)|$, for $y > 0$. Finally we need to show that $\e[|\widetilde \bm(|Y|)|^{(2 + \xi)}]$ is finite. Due to the last inequality it suffices to control:
\begin{align*}
C_V^{2 + \xi} \EE |\widetilde \bm(Y)|^{2 + \xi} &  \leq s^{(2 + \xi)/2} \EE |\EE[\bX^j | Y]|^{2 + \xi}  =s^{(2 + \xi)/2} \EE |\EE[s^{-1/2}\bX\T \bbeta |  Y]|^{2 + \xi}\\
& \leq s^{(2 + \xi)/2} s^{-(2 + \xi)/2} \EE \EE[|\bX\T \bbeta|^{2 + \xi} | Y]  = \EE |Z|^{2 + \xi} < \infty,
\end{align*}
where $Z \sim N(0,1)$ and in the first equality we used the fact that $\ind{\bX^j - s^{-1/2}\bX\T \bbeta}{\bX\T \bbeta}$ and hence $\EE[\bX^j - s^{-1/2}\bX\T \bbeta | Y] =\EE[\bX^j - s^{-1/2}\bX\T \bbeta]= 0$. This completes the proof.
\end{proof}

\begin{remark} \label{class:of:sliced:stable:models} \normalfont To see that the monotonicity condition on $\bm(y)$ is not vacuous, we give concrete examples of SIM satisfying it, which include the simple linear regression model (\ref{mod:linear}) as a special case. Consider the models from Example \ref{exmp3} and let $\varepsilon$ have a density function satisfying $p_\varepsilon(x) \varpropto \exp(-P(x^2))$, where $P$ is any nonzero polynomial with non-negative coefficients such that $P(0) = 0$. To see that $\bm$ is monotonic and continuous, we start by obtaining a precise expression of $\bm$. It is simple to see that:
$$
\bm(y) = \frac{\e[Z | h(Z) + \varepsilon = G^{-1}(y)]}{\sqrt{\var(\bm_j(Y))s}},
$$
for almost every $y \in \supp(Y)$. By Example \ref{exmp2}, we have that $\sqrt{\var(\bm_j(Y))s} = C_V$ for some $C_V > 0$. Next we argue that $y \mapsto \bm(y)$ is monotone. This is equivalent to showing that $\bm(G(y))$ is monotone. To this end we apply Lemma A.2 of \cite{duan1991slicing}, which implies that it suffices for the random variable $h(Z) + \varepsilon|Z$ to have a monotone likelihood ratio in order for $\bm(G(y))$ to be monotone. Since the family of random variables $h(z) + \varepsilon$ is a location family, the normalizing constants of their densities do not change with $z$. This in conjuction with the fact that $h$ is increasing, implies that the monotonicity of the likelihood ratio will be implied if we show that the function $x \mapsto P(x^2) - P((x - c)^2)$ is increasing in $x$ for any fixed $c > 0$. Notice that since $P(x^2)$ is a differentiable convex function by construction, we have that  $P(x^2) - P((x - c)^2) \geq c \frac{d P(y^2)}{dy} \vert_{y = x-c} > 0$. It is worth noting that the same argument applies more generally to the case where $\varepsilon$ is a log-concave random variable (i.e. $p_\varepsilon(x) = \exp(-\varphi(x))$ where $\varphi$ is a convex function). The fact that $\bm$ is continuous follows by the continuity of $G$ and $h$.
\end{remark}

Finally, with the help of Lemma \ref{lowerboundthm} we demonstrate that some models discussed in Remark \ref{class:of:sliced:stable:models} meet the information theoretic barrier described in Proposition \ref{lowerboundprop}, and hence their support cannot be recovered by any algorithm unless $\Gamma$ is large enough.

\begin{example} \label{lastexample} 
Let $Y = G(h(\bX\T \bbeta) + \varepsilon)$, where $G, h$ are strictly monotone continuous functions and in addition $h$ is an $L$-Lipschitz function. Furthermore let $\varepsilon$ has a density as specified in Remark \ref{class:of:sliced:stable:models}, i.e. $p_\varepsilon(x) \varpropto \exp(-P(x^2))$, where $P$ is any nonzero polynomial with non-negative coefficients such that $P(0) = 0$. Then if $n < \frac{s\log(p-s + 1)}{C}$ for some constant $C > 0$ (depending on $P, G, h$) and $s$ sufficiently large, any algorithm recovering the support in our model will have errors with probability at least $\frac{1}{2}$ asymptotically.
\begin{proof} Note that all moments of the random variable $\varepsilon$ exist. Next we verify that condition (\ref{KLdivcond}) holds in this setup. Since $G$ is 1-1 and KL divergence is invariant under changes of variables WLOG we can assume our model is simply $Y = h(\bX\T \bbeta) + \varepsilon$ or in other words $f(u, \varepsilon) = h(u) + \varepsilon$. This is a location family for $u \in \RR$ and thus the normalizing constant of the densities will stay the same regardless of the value of $u$. Direct calculation yields:
\begin{align*}
D(p(f(u, \varepsilon)) \| p(f(v, \varepsilon))) & = \e [P((\xi + h(u) - h(v))^2) - P(\xi^2)] \\
& = \widetilde{P}((h(u) - h(v))^2),
\end{align*}
where $\xi$ has a density $p_{\xi}(x) \varpropto \exp(-P(x^2))$, and $\widetilde{P}$ is another nonzero polynomial with nonnegative coefficients, with $\widetilde{P}(0) = 0$ of the same degree as $P$. The last equality follows from the fact that all odd moments of $\xi$ are $0$, since $\xi$ is a symmetric about $0$ distribution. Since $h$ is $L$-Lipschitz we conclude that:
\begin{align*}
D(p(f(u, \varepsilon)) \| p(f(v, \varepsilon))) & \leq \widetilde{P}(L^2(u - v)^2).
\end{align*}
The last can be clearly dominated by $\exp(C(u-v)^2) - 1$ for a large enough constant $C$.
\end{proof}
\end{example}

\bibliographystyle{agsm}
\bibliography{support_recovery_SIR}

\section{Extra Numerical Studies Figures} \label{app:extra:fig}

\begin{figure}[H]\label{log:DT:SIR}\caption{Efficiency Curves for DT-SIR, $s = \log{p},~ \Gamma = \frac{n}{s\log(p-s)} \in [0,30]$}  \label{log:DT:SIR}
\begin{subfigure}{.5\textwidth}
  \centering
  \includegraphics[width=.8\linewidth]{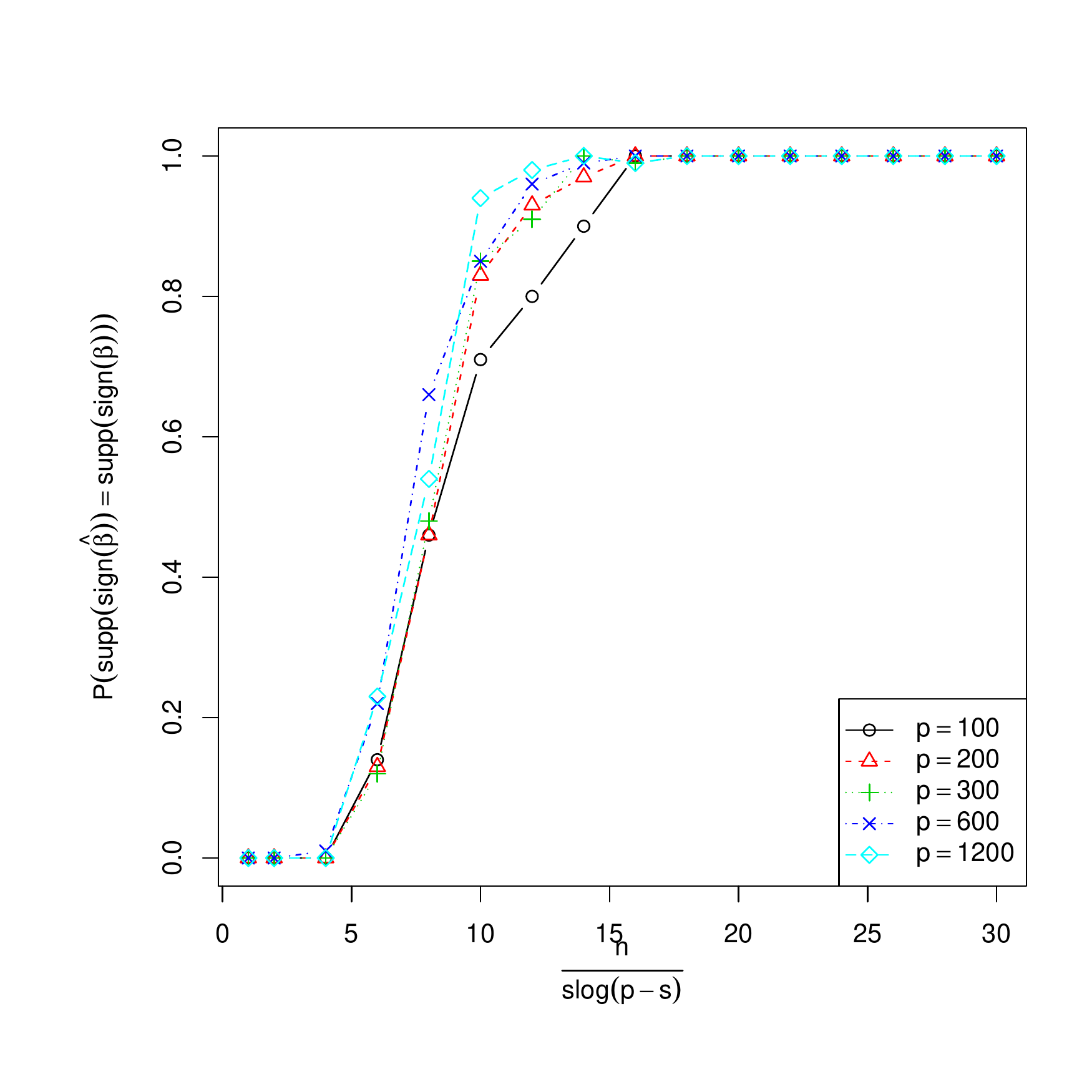}
  \caption{Model (\ref{sinmodel})}
  \label{fig:sfig1}
\end{subfigure}%
\begin{subfigure}{.5\textwidth}
  \centering
  \includegraphics[width=.8\linewidth]{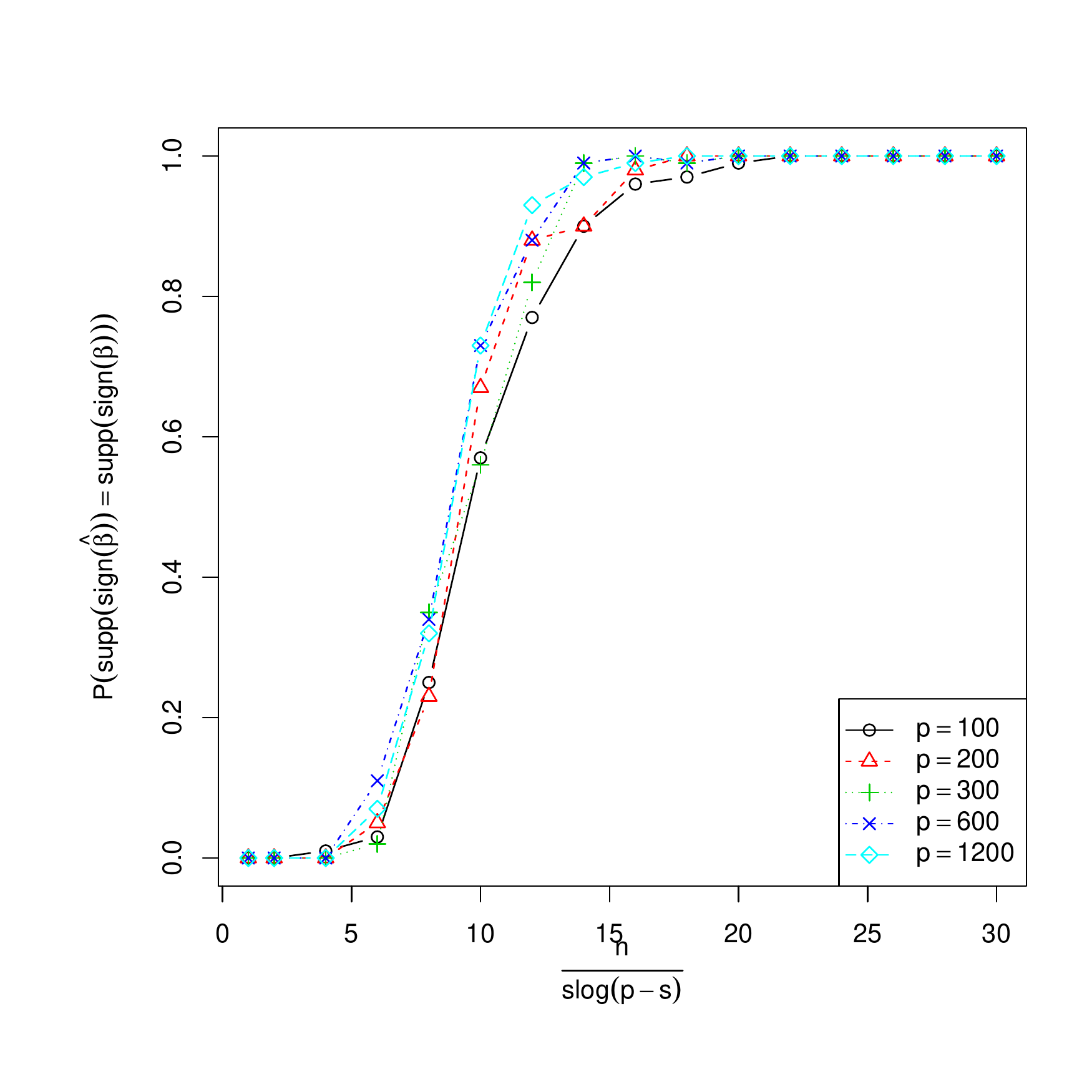}
  \caption{Model (\ref{X3model})}
  \label{fig:sfig2}
\end{subfigure}%

\begin{subfigure}{.5\textwidth}
  \centering
  \includegraphics[width=.8\linewidth]{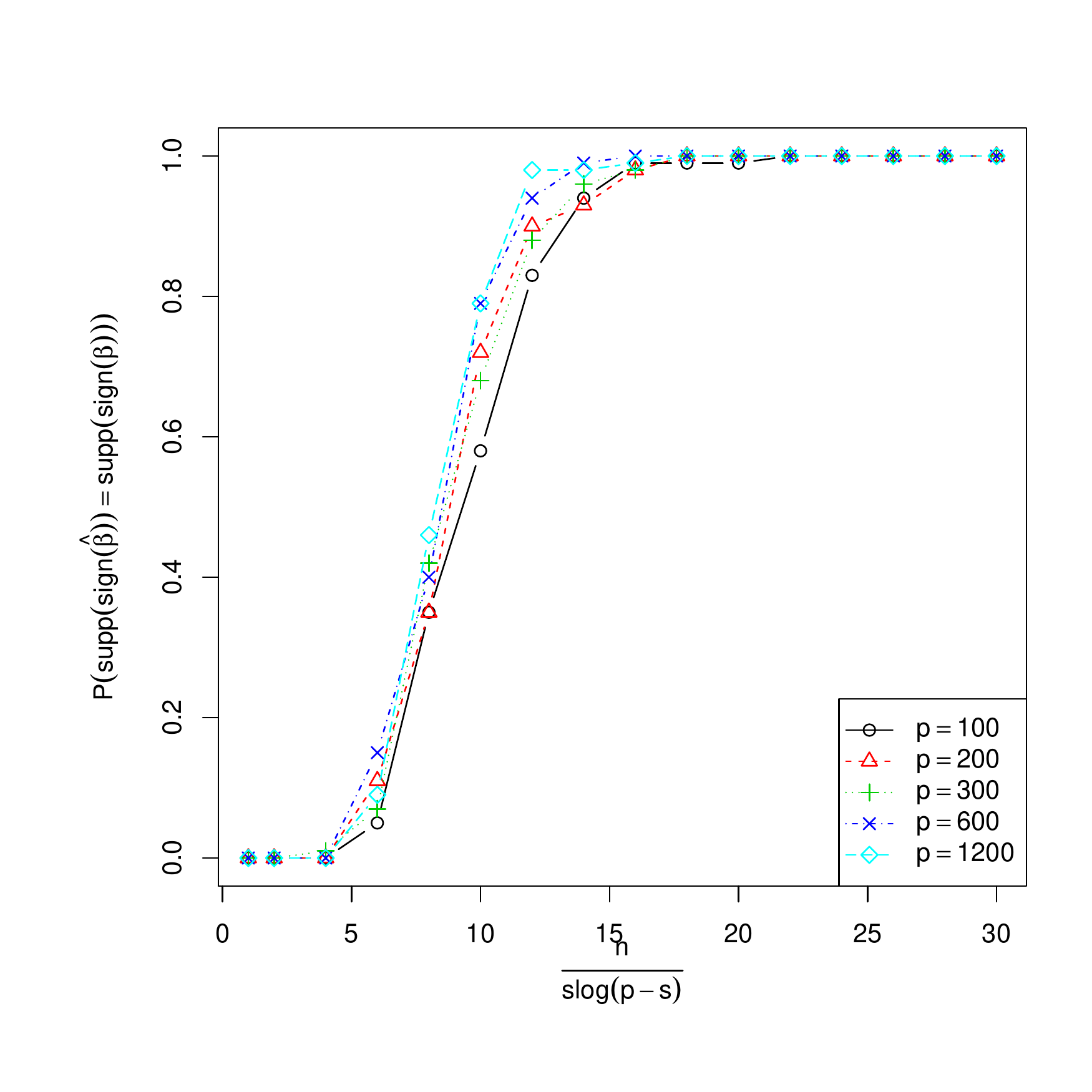}
  \caption{Model (\ref{Xe3model})}
  \label{fig:sfig3}
\end{subfigure}%
\begin{subfigure}{.5\textwidth}
  \centering
  \includegraphics[width=.8\linewidth]{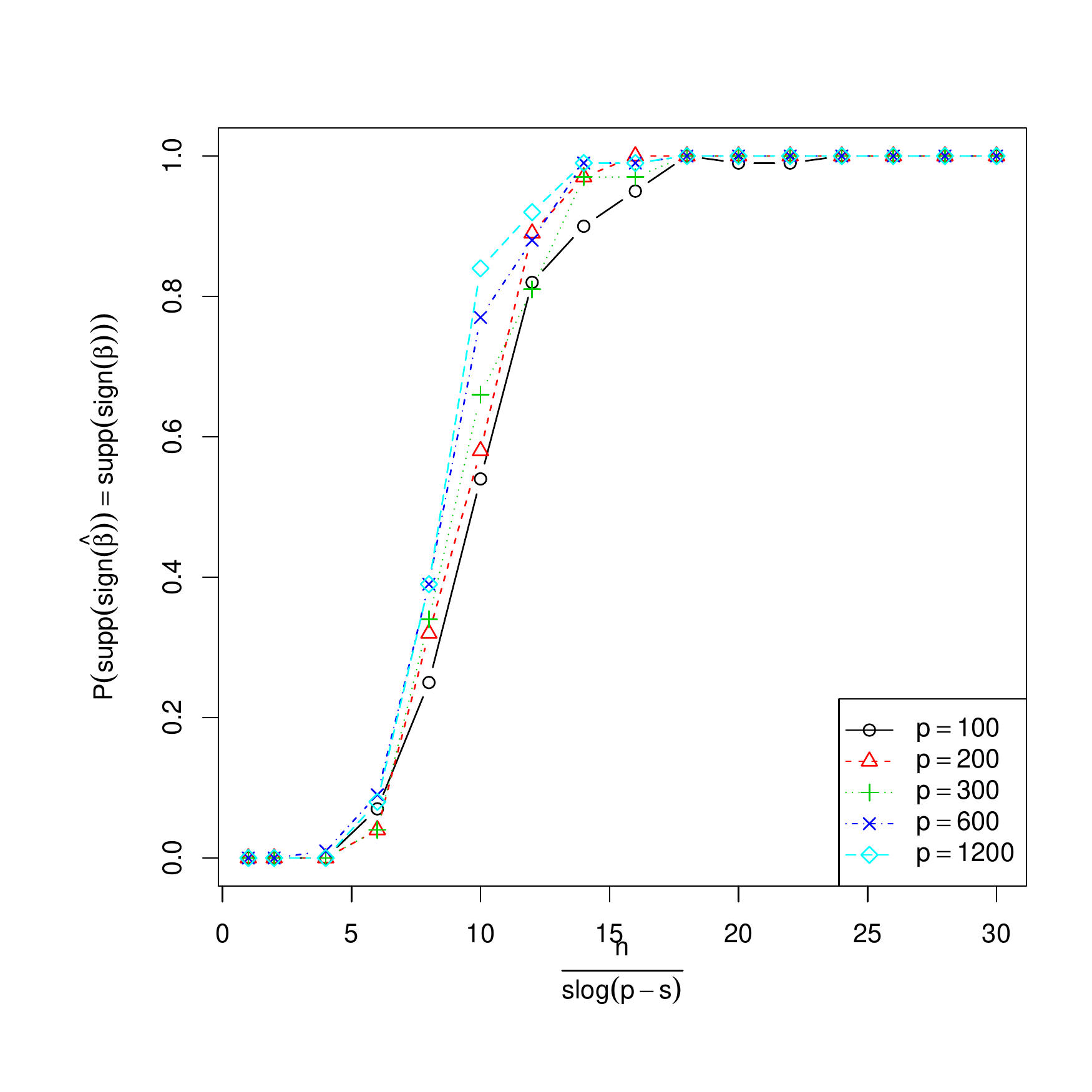}
  \caption{Model (\ref{regmodel})}
  \label{fig:sfig4}
\end{subfigure}%
\end{figure}

\begin{figure}[H]\label{log:SDP}\caption{Efficiency Curves for SDP, $s = \log{p},~ \Gamma = \frac{n}{s\log(p-s)} \in [0,40]$}  \label{log:SDPfigureslog}
\begin{subfigure}{.5\textwidth}
  \centering
  \includegraphics[width=.8\linewidth]{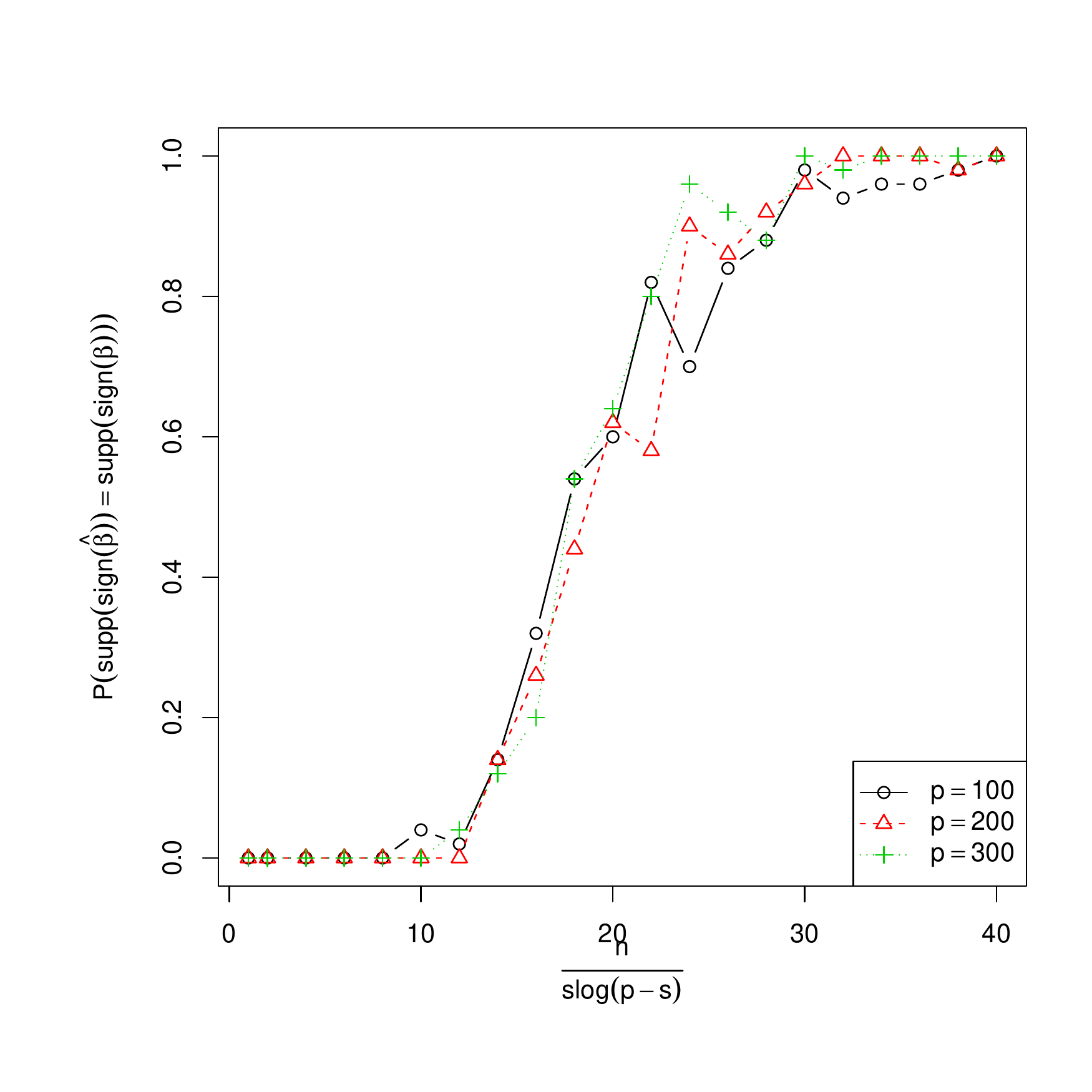}
  \caption{Model (\ref{sinmodel})}
  \label{fig:sfig1}
\end{subfigure}%
\begin{subfigure}{.5\textwidth}
  \centering
  \includegraphics[width=.8\linewidth]{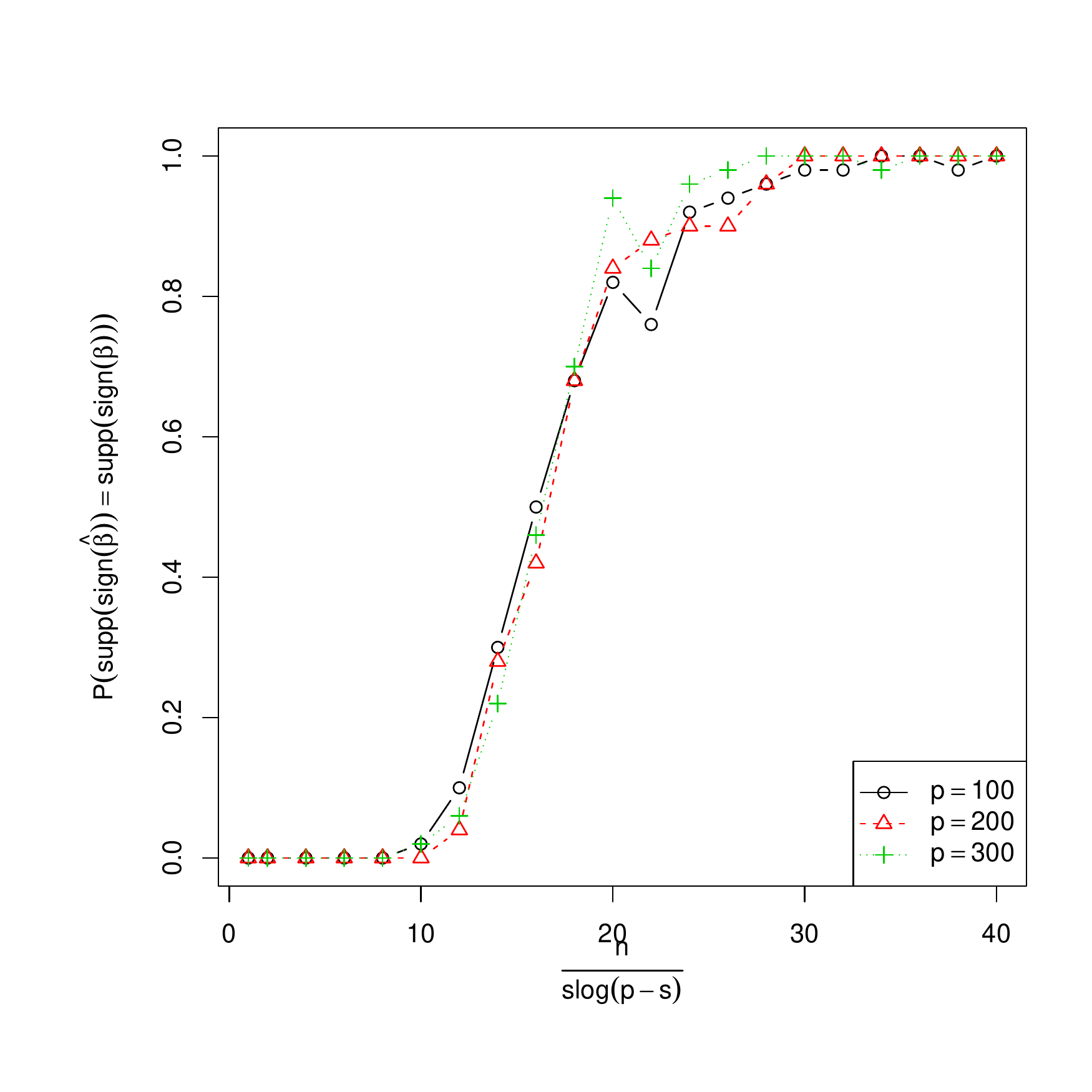}
  \caption{Model (\ref{X3model})}
  \label{fig:sfig2}
\end{subfigure}%

\begin{subfigure}{.5\textwidth}
  \centering
  \includegraphics[width=.8\linewidth]{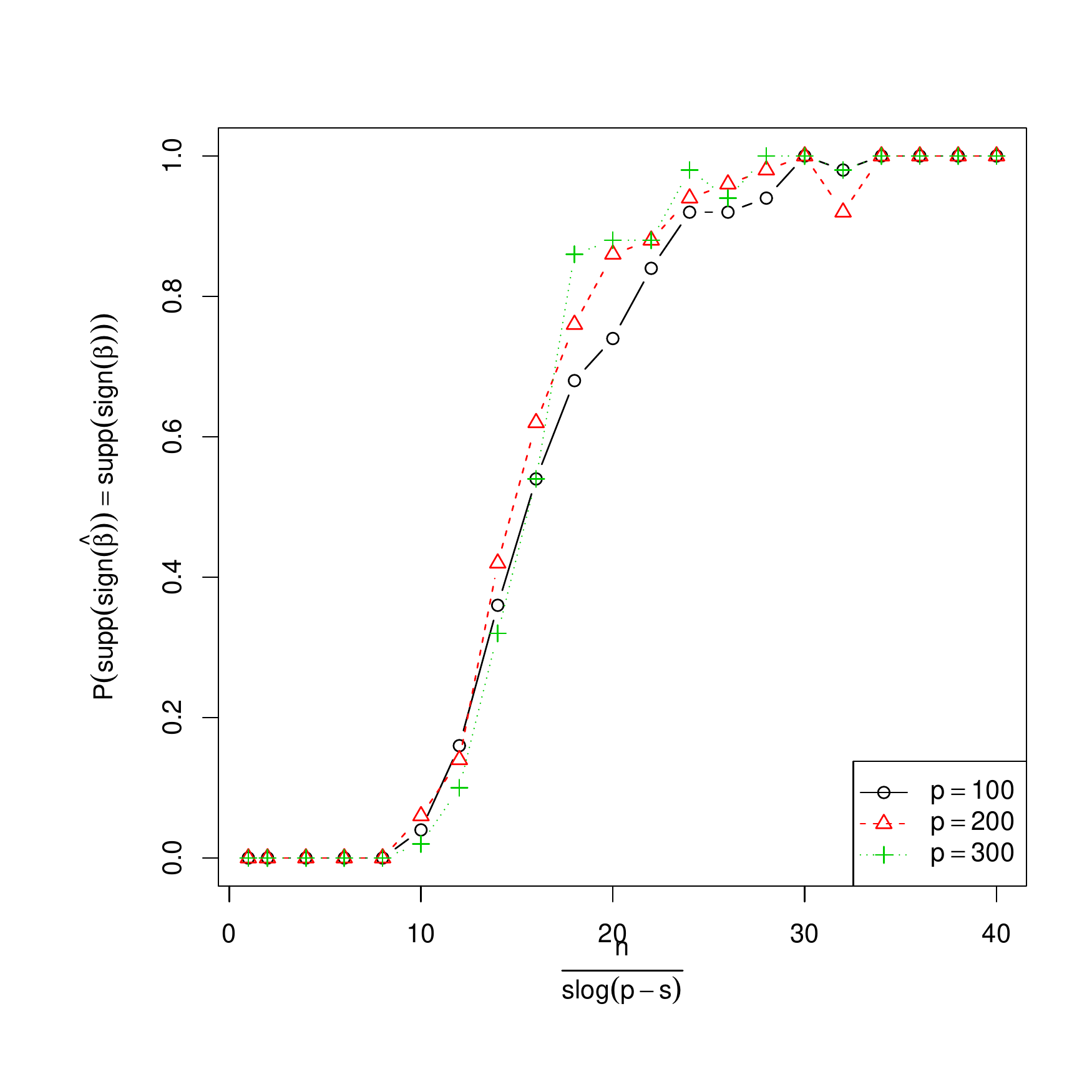}
  \caption{Model (\ref{Xe3model})}
  \label{fig:sfig3}
\end{subfigure}%
\begin{subfigure}{.5\textwidth}
  \centering
  \includegraphics[width=.8\linewidth]{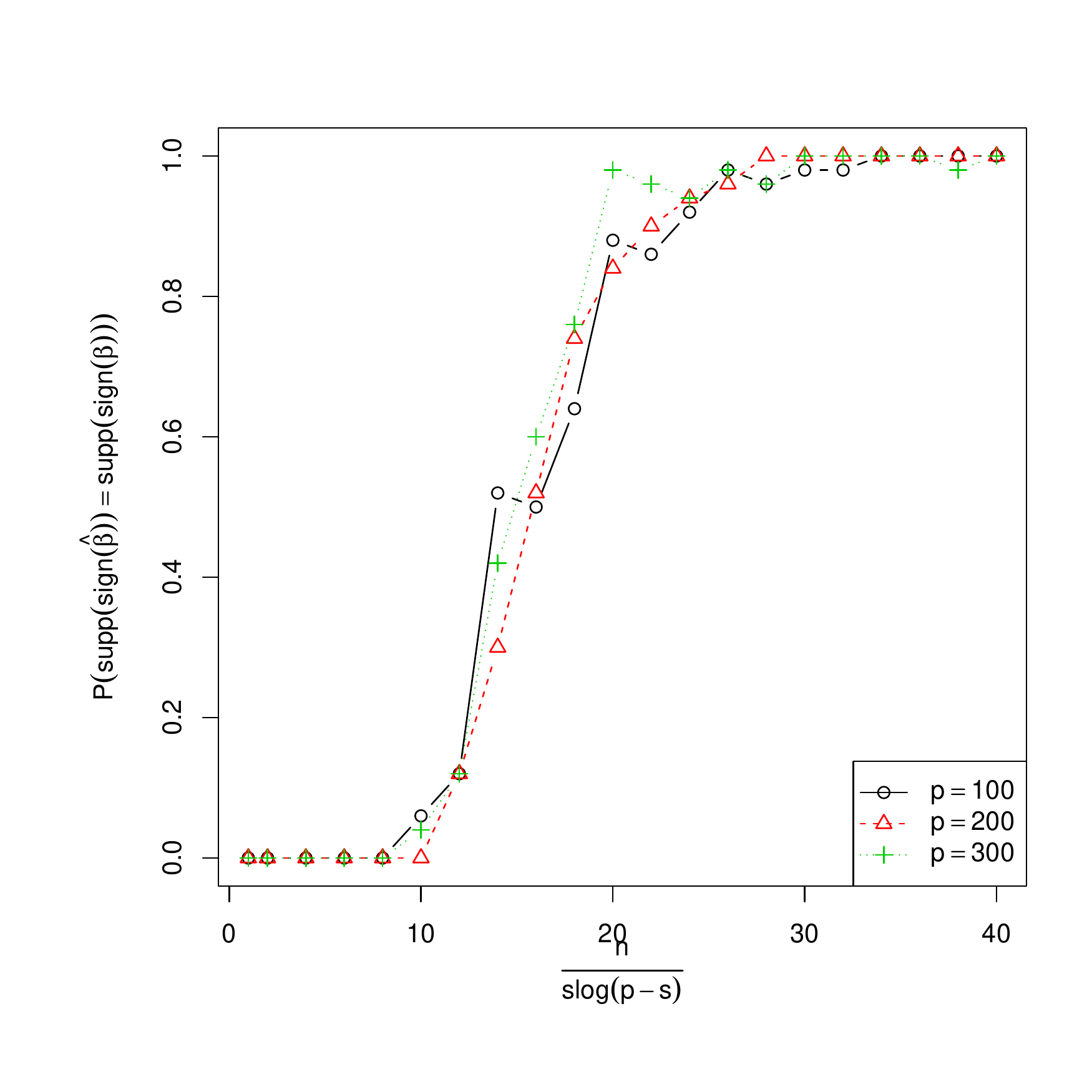}
  \caption{Model (\ref{regmodel})}
  \label{fig:sfig4}
\end{subfigure}%
\end{figure}

\begin{figure}[H]\caption{Efficiency Curves for DT-SIR, $s = \sqrt{p},~ \Gamma = \frac{n}{s\log(p-s)} \in [0,30]$, $\bbeta = \widetilde \bbeta/\|\widetilde \bbeta\|_2$ according to (\ref{random:beta:specification})}  
\label{sqrt:DT:random:beta}
\begin{subfigure}{.5\textwidth}
  \centering
  \includegraphics[width=.8\linewidth]{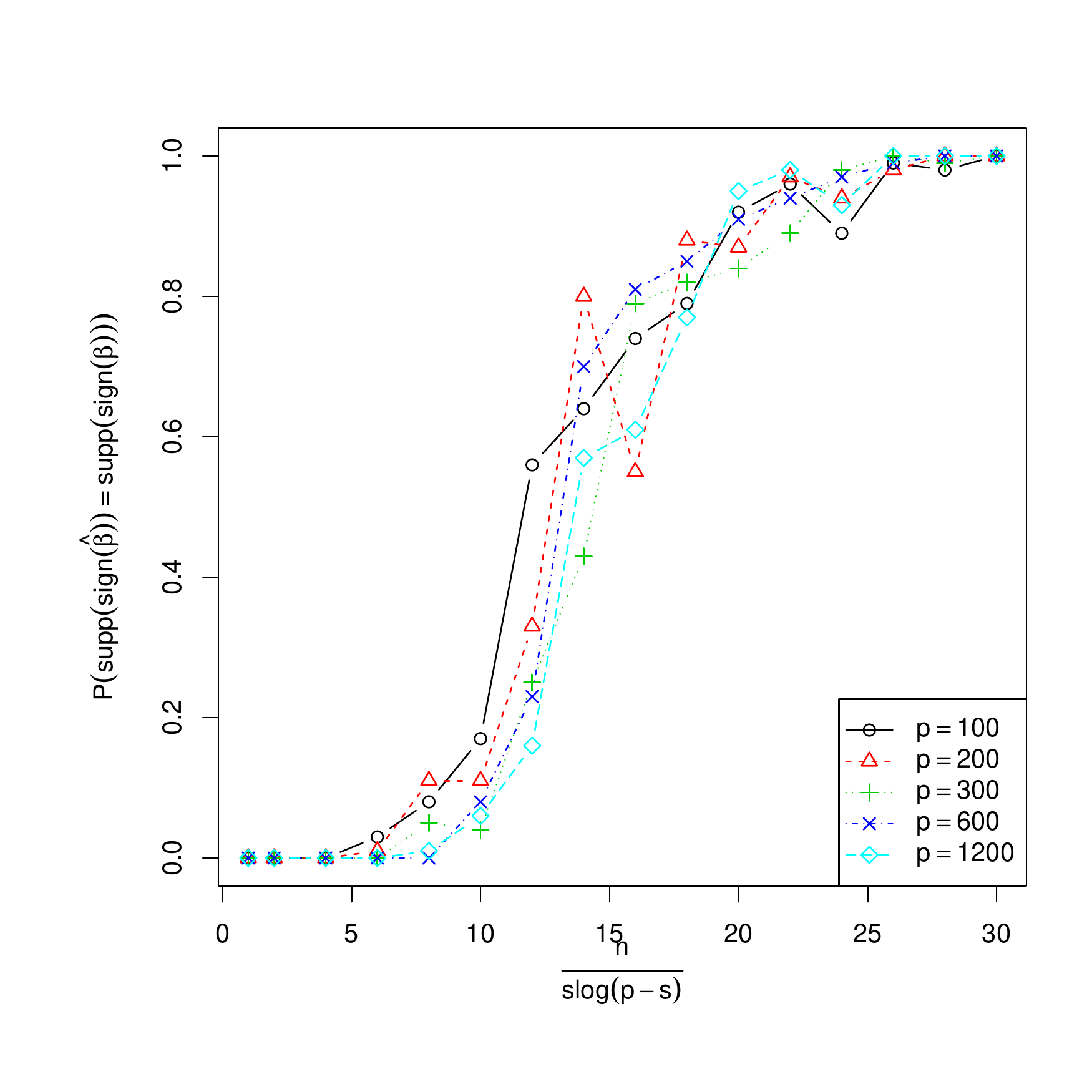}
  \caption{Model (\ref{sinmodel})}
  \label{fig:sfig1:random:beta}
\end{subfigure}%
\begin{subfigure}{.5\textwidth}
  \centering
  \includegraphics[width=.8\linewidth]{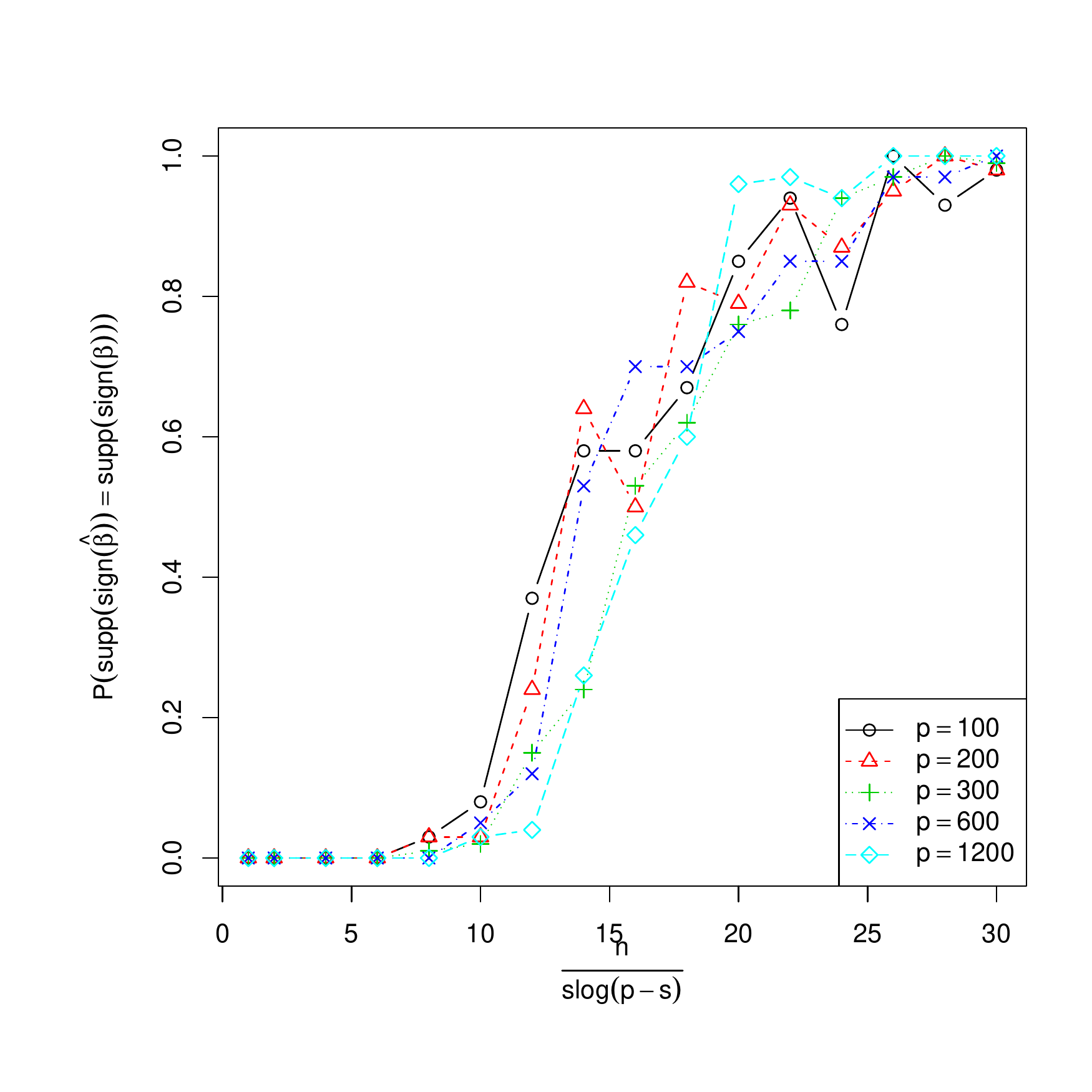}
  \caption{Model (\ref{X3model})}
  \label{fig:sfig2:random:beta}
\end{subfigure}%

\begin{subfigure}{.5\textwidth}
  \centering
  \includegraphics[width=.8\linewidth]{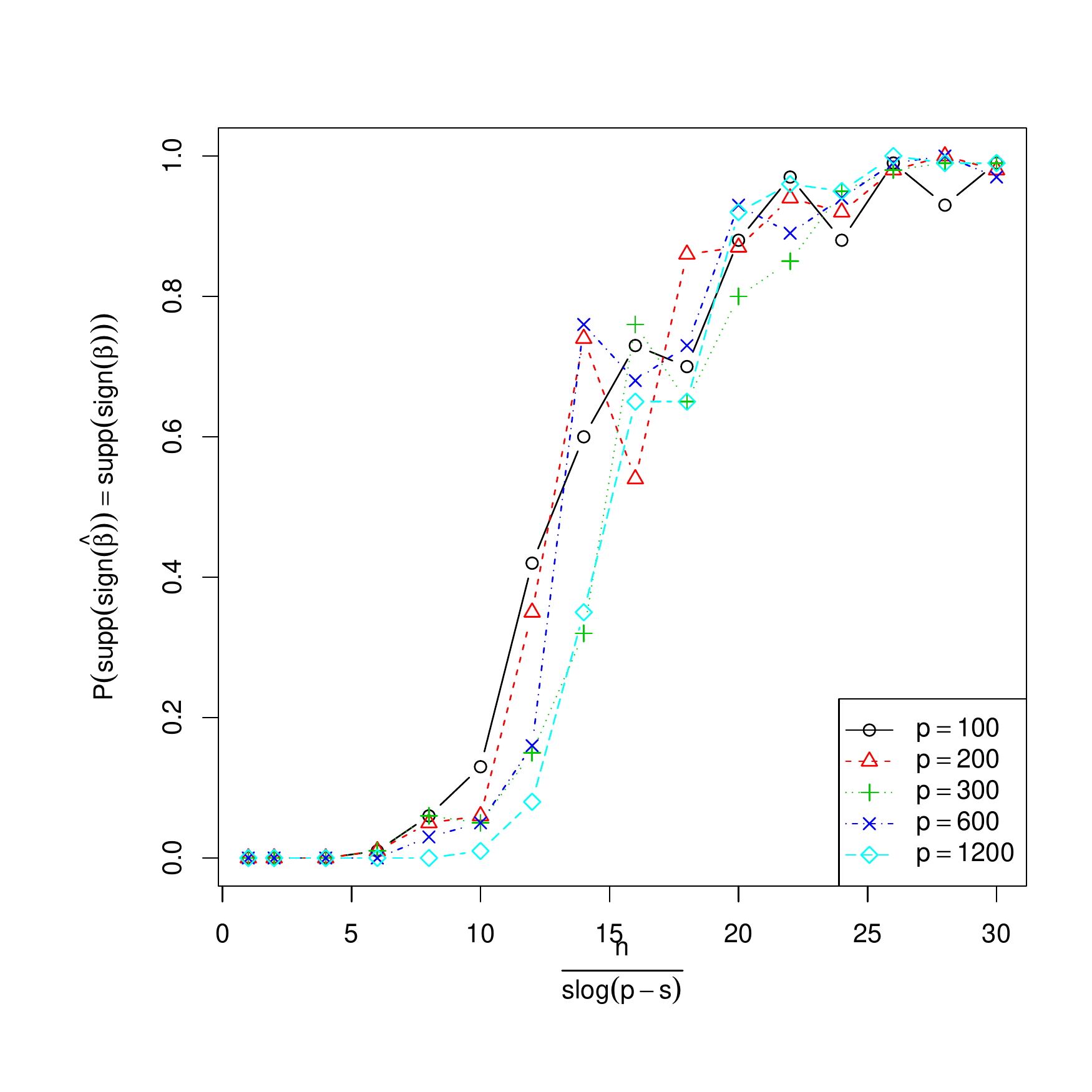}
  \caption{Model (\ref{Xe3model})}
  \label{fig:sfig3:random:beta}
\end{subfigure}%
\begin{subfigure}{.5\textwidth}
  \centering
  \includegraphics[width=.8\linewidth]{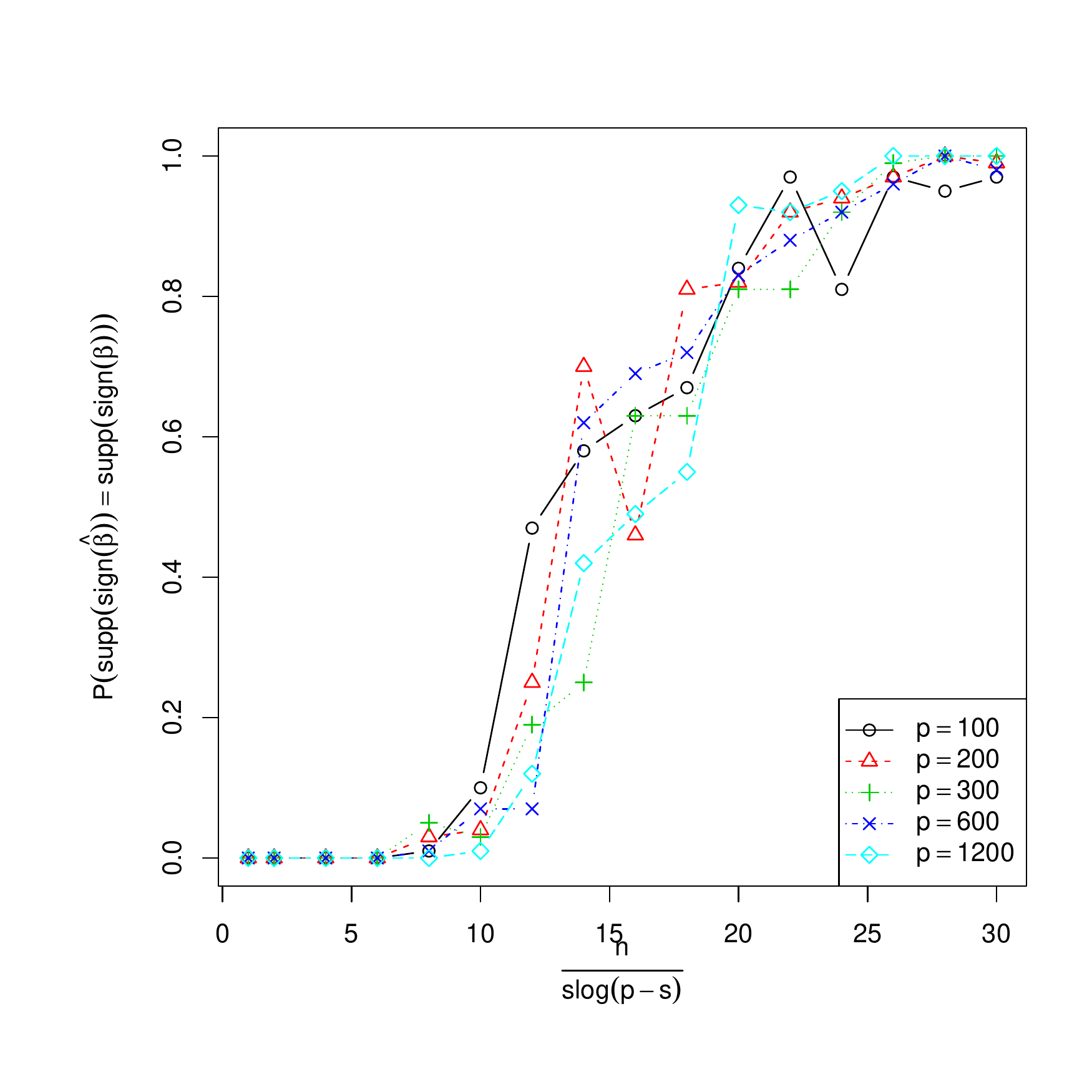}
  \caption{Model (\ref{regmodel})}
  \label{fig:sfig4:random:beta}
\end{subfigure}%
\end{figure}

\begin{figure}[H]\caption{Efficiency Curves for SDP, $s = \sqrt{p},~ \Gamma = \frac{n}{s\log(p-s)} \in [0,40]$, $\bbeta = \widetilde \bbeta/\|\widetilde \bbeta\|_2$ according to (\ref{random:beta:specification})}  
\label{sqrt:SDP:random:beta}
\begin{subfigure}{.5\textwidth}
  \centering
  \includegraphics[width=.8\linewidth]{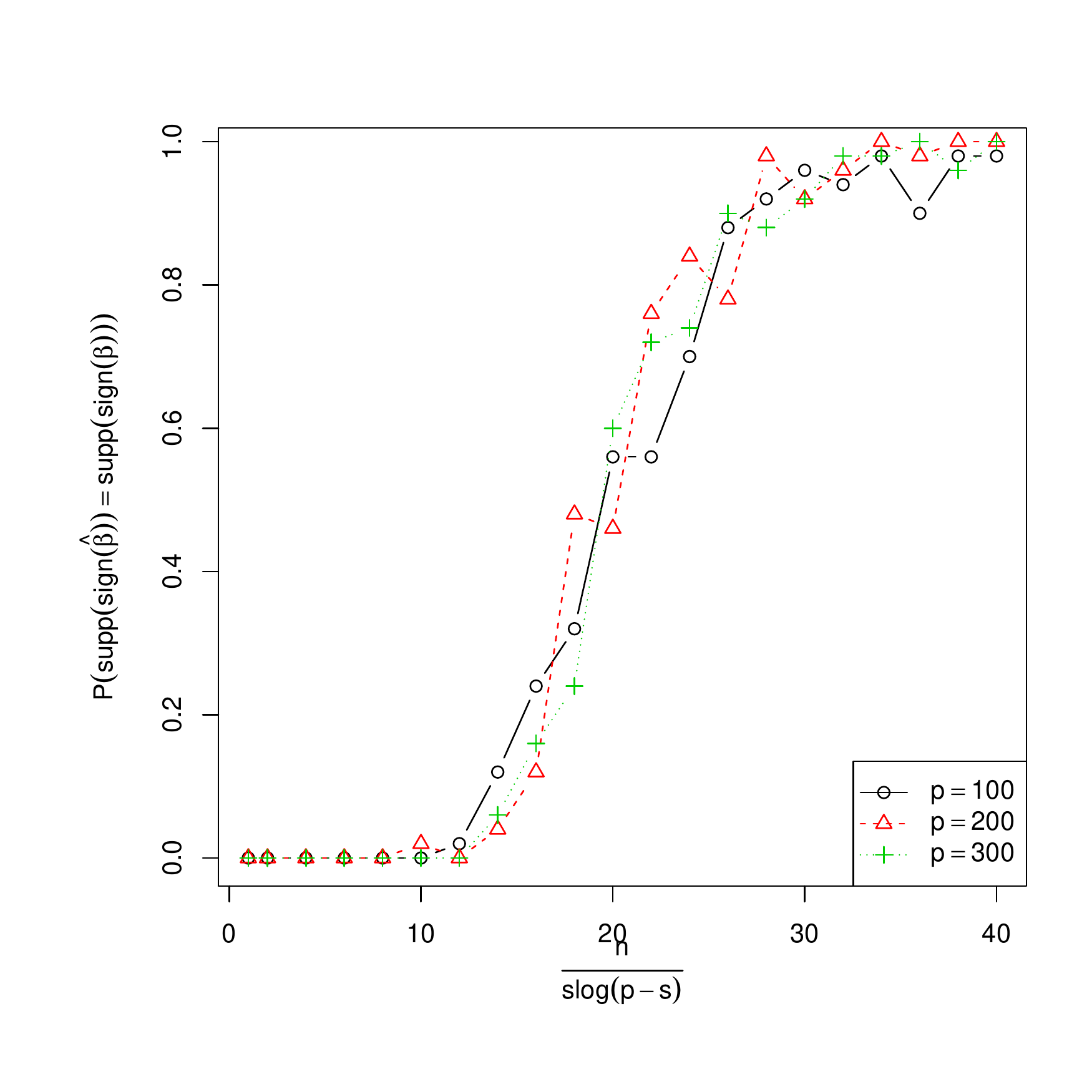}
  \caption{Model (\ref{sinmodel})}
  \label{fig:sfig1:random:beta}
\end{subfigure}%
\begin{subfigure}{.5\textwidth}
  \centering
  \includegraphics[width=.8\linewidth]{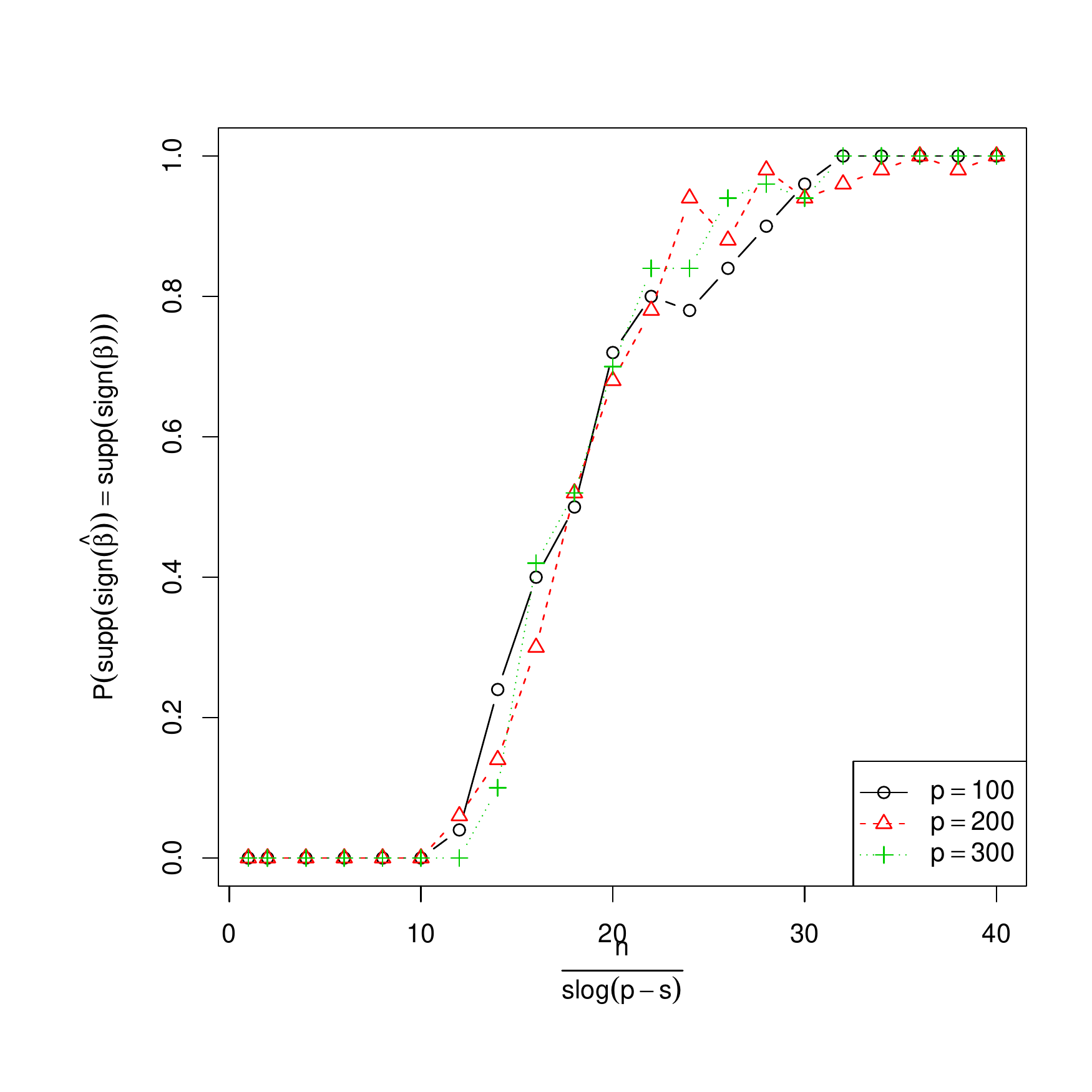}
  \caption{Model (\ref{X3model})}
  \label{fig:sfig2:random:beta}
\end{subfigure}%

\begin{subfigure}{.5\textwidth}
  \centering
  \includegraphics[width=.8\linewidth]{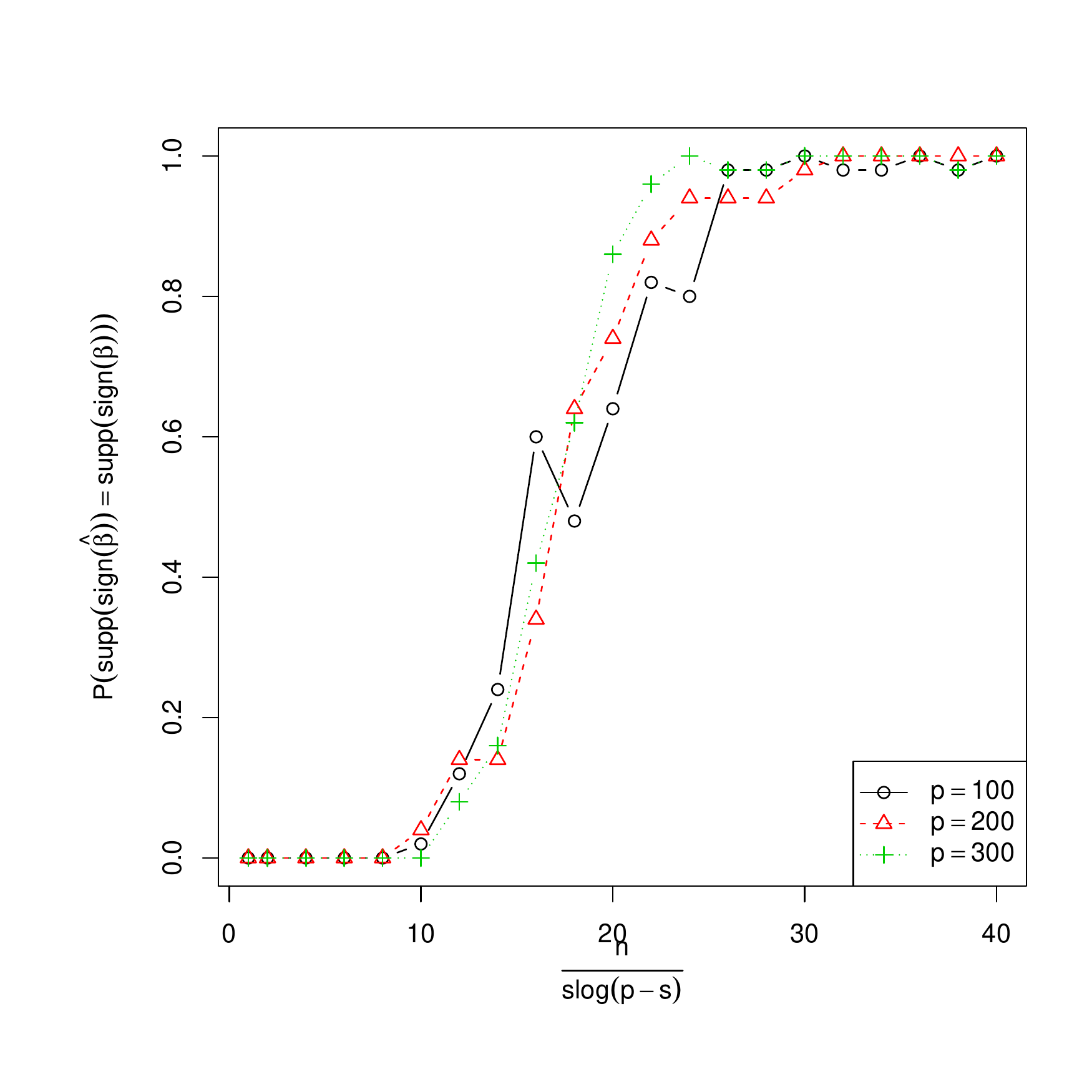}
  \caption{Model (\ref{Xe3model})}
  \label{fig:sfig3:random:beta}
\end{subfigure}%
\begin{subfigure}{.5\textwidth}
  \centering
  \includegraphics[width=.8\linewidth]{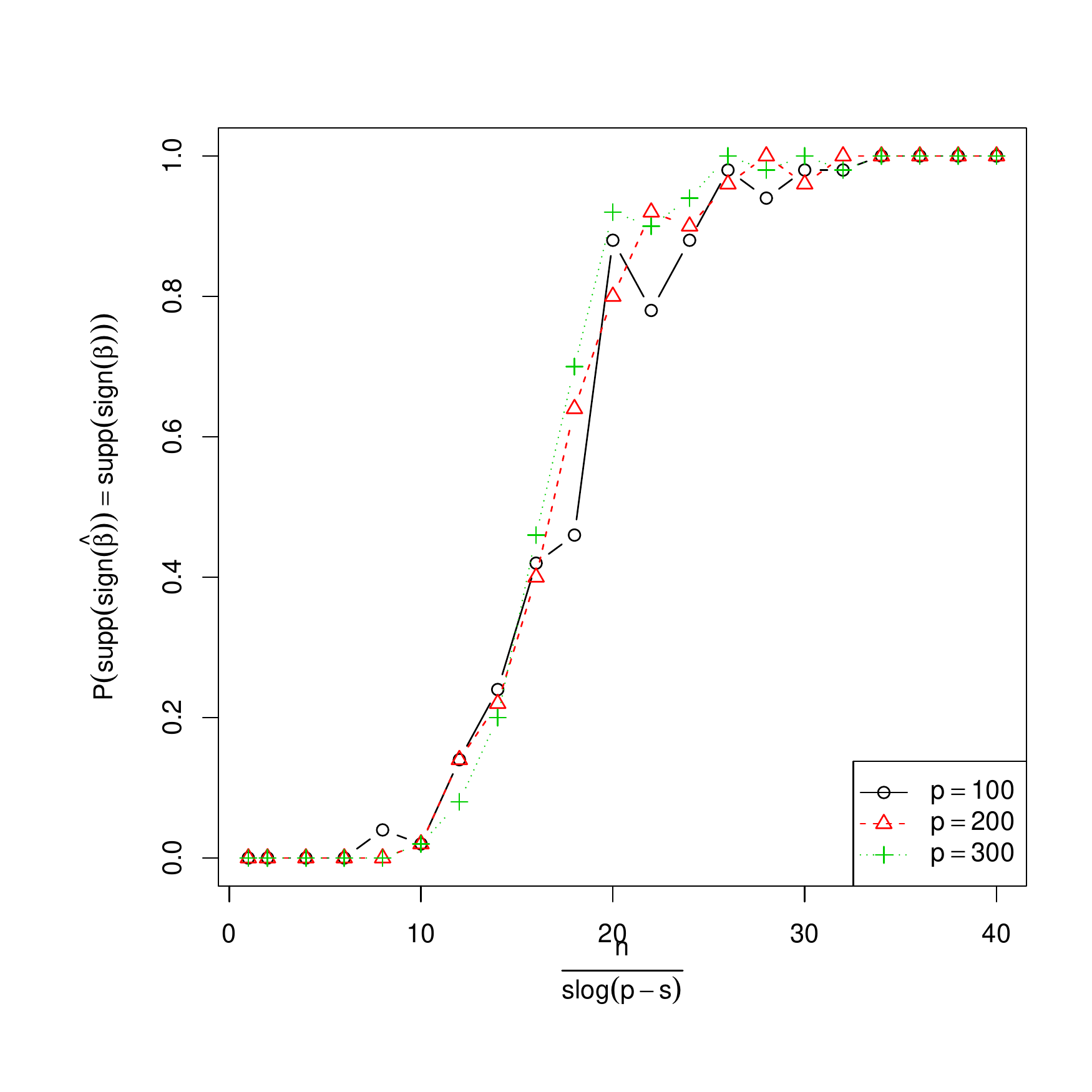}
  \caption{Model (\ref{regmodel})}
  \label{fig:sfig4:random:beta}
\end{subfigure}%
\end{figure}

\end{document}